\documentclass[11pt,reqno]{amsart}
\usepackage{bm,amsmath,amsthm,amssymb,mathtools,verbatim,amsfonts,tikz-cd,mathrsfs,diagbox,makecell}
\usepackage{enumitem}
\usepackage{float}
\usepackage{xfrac}
\usepackage{url}
\usepackage{thmtools}
\usepackage{graphicx}
\usepackage[alphabetic]{amsrefs}
\usepackage{caption}
\usepackage{morefloats}
\usepackage[top=1.4in, bottom=1.2in, left=1.4in, right=1.4in,marginpar=1in]{geometry}
\usepackage{subfigure}

\usepackage{tikz}
\tikzset{axislabel/.style={color=black!60!white, font=\tiny}}
\usepackage{adjustbox}
\usetikzlibrary{cd}
\usetikzlibrary{fit, patterns}
\usetikzlibrary{matrix,arrows,decorations.pathmorphing,decorations.pathreplacing,calligraphy,backgrounds,calc}
\usepackage[hidelinks,colorlinks=true,linkcolor=blue, citecolor=black,linktocpage=true, hypertexnames=false]{hyperref}

\usepackage{chngcntr}
\counterwithin{figure}{section}

\newtheorem{thm}{Theorem}[section]
\newtheorem{prop}[thm]{Proposition}

\newtheorem{lem}[thm]{Lemma}

\theoremstyle{definition}
\newtheorem{define}[thm]{Definition}

\theoremstyle{remark}
\newtheorem{rem}[thm]{Remark}

\numberwithin{equation}{section}

\newcommand{\ve}[1]{\boldsymbol{\mathbf{#1}}}

\newcommand{\R}{\mathbb{R}}

\newcommand{\Z}{\mathbb{Z}}
\newcommand{\N}{\mathbb{N}}

\renewcommand{\d}{\partial}
\renewcommand{\subset}{\subseteq}
\renewcommand{\supset}{\supseteq}
\renewcommand{\tilde}{\widetilde}

\newcommand{\iso}{\cong}

\DeclareMathOperator{\free}{{free}}
\DeclareMathOperator{\GL}{{GL}}
\DeclareMathOperator{\gr}{{gr}}

\DeclareMathOperator{\Id}{{Id}}
\DeclareMathOperator{\id}{{id}}

\DeclareMathOperator{\im}{{im}}

\DeclareMathOperator{\Tor}{{Tor}}

\DeclareMathOperator{\rel}{{rel}}

\DeclareMathOperator{\hori}{{hori}}

\newcommand{\lk}{\mathrm{lk}}

\newcommand{\bE}{\mathbb{E}}
\newcommand{\bF}{\mathbb{F}}

\newcommand{\bH}{\mathbb{H}}
\newcommand{\bI}{\mathbb{I}}
\newcommand{\bJ}{\mathbb{J}}

\newcommand{\bM}{\mathbb{M}}

\newcommand{\bX}{\mathbb{X}}

\newcommand{\bZ}{\mathbb{Z}}

\newcommand{\cB}{\mathcal{B}}
\newcommand{\cC}{\mathcal{C}}
\newcommand{\cD}{\mathcal{D}}

\newcommand{\cH}{\mathcal{H}}

\newcommand{\cK}{\mathcal{K}}
\newcommand{\cL}{\mathcal{L}}

\newcommand{\cR}{\mathcal{R}}
\newcommand{\cS}{\mathcal{S}}

\newcommand{\cX}{\mathcal{X}}

\newcommand{\scE}{\mathscr{E}}
\newcommand{\scF}{\mathscr{F}}

\newcommand{\scJ}{\mathscr{J}}

\newcommand{\scM}{\mathscr{M}}

\newcommand{\cCFL}{\mathcal{C\!F\!L}}
\newcommand{\cCFK}{\mathcal{C\hspace{-.5mm}F\hspace{-.3mm}K}}
\newcommand{\cHFL}{\mathcal{H\hspace{-.5mm}F\hspace{-.3mm}L}}

\newcommand{\CF}{\mathit{CF}}

\newcommand{\CFK}{\mathit{CFK}}

\newcommand{\CFL}{\mathit{CFL}}

\newcommand\CFKi{\CFK^\infty}

\newcommand{\xs}{\ve{x}}
\newcommand{\ys}{\ve{y}}
\newcommand{\zs}{\ve{z}}
\newcommand{\ws}{\ve{w}}

\renewcommand{\a}{\alpha}

\newcommand{\g}{\gamma}

\newcommand{\veps}{\varepsilon}

\renewcommand{\GL}{\mathit{GL}}

\usepackage{url}
\usepackage{leftidx}

\DeclareMathOperator{\Cone}{{Cone}}

\newcommand{\ar}{\mathrm{a.r.}}

\newcommand{\llsquare}{[\hspace{-.5mm}[}
\newcommand{\rrsquare}{]\hspace{-.5mm}]}

\newcommand{\Ord}{\operatorname{Ord}}

\allowdisplaybreaks

\renewcommand{\top}{\mathrm{top}}

\title{Unknotting number and L-space satellite operations}

\author{Daren Chen}
\address{Department of Mathematics\\Lehigh Universeity\\ Bethlehem, PA, USA}
\email{dac726@lehigh.edu}

\author{Ian Zemke}
\address{Department of Mathematics\\University of Oregon\\  Eugene, OR, USA}
\email{izemke@uoregon.edu}

\author{Hugo Zhou}
\address{Department of Mathematics\\University of Michigan\\  Ann Arbor, MI, USA}
\email{hugozhou@umich.edu}

\begin{document}
	\maketitle
	\begin{abstract}
    We give a lower bound for the unknotting number under the L-space satellite operations. 
    In the case of braided L-space satellite operations, including cabling operations, we obtain a stronger bound.
    The proof uses a lower bound for the unknotting number in knot Floer homology due to Alishahi-Eftekhary.
    The computational tool is a formula for computing the knot Floer complex under the L-space satellite operations, which was defined in the authors' previous work. 
	\end{abstract}
	\tableofcontents

\section{Introduction}
In this article, we study the Heegaard Floer torsion order of satellite knots. We recall that a \emph{satellite pattern} is a knot $P$ in the solid torus $S^1\times D^2$. Given a knot $K\subset S^3$ (called the \emph{companion}), a pattern $P$, and an integer $\lambda\in \Z$, we may form a new knot $P(K,\lambda)$ by removing a tubular neighborhood of $K$ and gluing in $(S^1\times D^2,P)$ so that $S^1\times \{\theta\}$ is mapped to the $\lambda$-framed longitude of $K$, for any fixed $\theta\in \d D^2$. 

Using the unknotting number bound of Alishahi--Eftekhary \cite{AlishahiEftekhary_Unknotting} and the cabling formula of Hanselman and Watson \cite{HWCabling}, Hom--Lidman--Park \cite{HLPUnknotting} proved that if $K$ is a non-trivial knot, then the cable $K_{p,q}$ has unknotting number at least $p$.  Furthermore, they proved that if $K$ is a non-trivial knot, then the iterated cable $K_{p_1,q_1;\dots; p_k,q_k}$ has unknotting number at least $p_1\cdots p_k$. (Here, we use the convention that the $(p,q)$-cabling operator is wraps $p$ times longitudinally and $q$ times meridianally around the boundary of $\d (S^1\times D^2)$). 

The cabling operators fit into a larger family called \emph{L-space satellite operators}, studied in \cite{CZZ} \cite{CZZApp}. We recall that a pattern $P$ is an \emph{L-space satellite operator} if the induced two component link $L_P=\mu\cup P\subset S^3$ is an L-space link. Here, $L_P$ is formed by embedding $S^1\times D^2$ as a tubular neighborhood of the unknot and $\mu$ is the core of the complementary solid torus.

In this paper, we prove an analog of Hom--Lidman--Park's result about cables for L-space satellite operators. Since cables are examples of L-space satellite operators, our proof gives an alternate proof of most of the results from \cite{HLPUnknotting}.

Our first result concerns \emph{braided} L-space satellite operators, which we recall are ones where $P$ intersects each meridianal disk $\theta \times D^2$ transversely. The \emph{winding number} of a satellite operator is the algebraic intersection number of $P$ with a meridianal disk. Note that this is the same as the linking number of the 2-component link $L_P$, so we will usually write $\ell$ for the winding number.

\begin{restatable}{thm}{IterateCable}\label{thm: iterate_braided}
   Suppose that for $i=1,\cdots,k$, $P_i$ is a braided L-space satellite pattern. If $K$ is a non-trivial knot which is not $T_{2,3}$, then $u(P_k \circ \cdots \circ P_1(K)) \geq |\ell_k \cdots \ell_1|.$  Here, $P_i(K)$ denotes $P_i(K,\lambda_i)$  for any $\lambda_i\in \Z$. 
\end{restatable}

We also prove a general bound for L-space operators. Our bound is phrased in terms of the winding number and the relative 3-genus, whose definition we now recall. If $P\subset S^1\times D^2$ is a pattern, we define $g_3^{\rel}(P)$ to be the minimum genus of an embedded surface in $S^1\times D^2$ whose boundary consists of $|\ell|$ copies of a 0-framed longitude of the solid torus and $P$.  We define 
\[
\theta(P):=g_3^{\rel}(P)-g_3(P)
\]
(where $g_3(P)$ denotes $g_3(P(U))$). The quantity $\theta$ is always nonnegative.

\begin{restatable}{thm}{IterateGeneral}\label{thm: iterate_general} Suppose that $P_1,\dots, P_k$ are L-space satellite operators with $\theta_i=g_3^{\rel}(P_i)-g_3(P_i)$. If $K$ is a non-trivial knot which is not $ T_{2,3}$, then
\[
u(P_k\circ \cdots\circ  P_1(K))\ge |\ell_k| (\cdots (|\ell_2|(|\ell_1| + \theta_1 -1) + \theta_2 -1) \cdots -1)  + \theta_k.
\]
In the above, $P_i(K)$ denotes $P_i(K,\lambda_i)$ for any $\lambda_i\in \Z$.
\end{restatable}

It follows from \cite[Remark 7.11]{CZZApp} that if $P$ has equal wrapping and winding numbers, then $\theta(P)=0$. In particular, for braids the bound from Theorem~\ref{thm: iterate_general} is
\[
u(P_k\circ \cdots P_1(K))\ge |\ell_k|(\cdots (|\ell_2|(|\ell_1|-1)-1)\cdots -1), 
\]
which is weaker than the bound in Theorem~\ref{thm: iterate_braided}. 

\begin{rem} It follows from \cite{CZZApp}*{Proposition~7.2} that $\theta(P)$ can alternatively be described as $R_{\ell/2}-g_3(P)-|\ell|/2$,   where $R_{\ell/2}\in \Z+\ell/2$ is a number defined using the $H$-function of $L_P$. See Definition~\ref{def: R_t}.
\end{rem}

The about results about the unknotting number are proved by applying Alishahi and Eftekhary's bound using the Heegaard Floer torsion order (see Section \ref{subsec: tor_ord}). We prove Theorems \ref{thm: iterate_braided} and \ref{thm: iterate_general} by proving general bounds on the torsion order of satellite knots formed using L-space satellites. We state versions of these bounds here. The first bound we prove is for general L-space satellite operators:

\begin{restatable}{thm}{TorOrd}\label{thm: tor_ord}
Suppose $P$ is an L-space satellite operator. For any nontrivial knot $K\neq T_{2,3}$ and $\lambda \in\bZ$, we have
\begin{equation*}
    \Ord(P(K,\lambda )) \geq  \max\{|\ell|(\Ord(K)-1),|\ell|\} + \theta.
\end{equation*}
\end{restatable}

Note that the above theorem implies Theorem~\ref{thm: iterate_general}:
\begin{proof}[Proof of Theorem \ref{thm: iterate_general}]
       The unknotting number satisfies $u(K)\ge \Ord(K)$ by \cite[Theorem 1.1]{AlishahiEftekhary_Unknotting} and so the result follows by repeatedly applying Theorem \ref{thm: tor_ord}.
\end{proof}

We also consider $K= T_{2,3}$, but we are only able to obtain a weaker bound. See Theorem~\ref{thm: tor_ord_T_23}. 

In the case of braided L-space satellite patterns, a similar argument to that of Theorem \ref{thm: tor_ord} yields the following stronger bound:
\begin{restatable}{thm}{TorOrdBraid}\label{thm: tor_ord_braided}
Let $P$ be a braided L-space satellite pattern. For any nontrivial knot $K$ and $\lambda\in\bZ$, we have
\begin{equation*}
    \Ord(P(K,\lambda)) \geq    \max\{|\ell|(\Ord(K)-1)+1,|\ell|\}.
\end{equation*}
\end{restatable}

We remark that Theorem~\ref{thm: iterate_braided} does not immediately follow from Theorem~\ref{thm: tor_ord_braided}. Instead,  Theorem~\ref{thm: iterate_braided} requires a slightly more detailed study of the model for the knot Floer complex of $P(K,\lambda)$ that is produced using the description from \cite{CZZ}. This is the subject of Section~\ref{sec: iterate}.

\begin{rem}
We may replace the $u(K)$ in Theorem \ref{thm: iterate_braided} and \ref{thm: iterate_general} by $u_q(K)$,  
 the \emph{proper rational tangle replacement number} of $K$, defined analogously to the unknotting number. Explicitly, $u_q(K)$ is the
minimum number of rational tangle replacements required to turn a diagram of $K$ to a diagram of the unknot, where the new tangle  connects the same tangle end points.
    By \cite[Theorem 1.2]{EftekharyRational}, $\Ord(K)$ is a lower bound for $u_q(K)$.  
\end{rem} 

\subsection{Organization}
In Section \ref{sec: preli}, we review some preliminaries on knot Floer homology and define the torsion order. In Section \ref{sec: L-space-satellite}, we recall the L-space satellite formula. We also prove a helpful property about the $H$-function of braided L-space patterns. We also give a helpful diagrammatic description of the L-space satellite formula. In Section \ref{sec: main_bound}, we prove our main bounds on the torsion order, Theorem \ref{thm: tor_ord} and Theorem \ref{thm: tor_ord_braided}. In Section \ref{sec: iterate}, we consider iterated satellites and prove Theorem \ref{thm: iterate_braided}.

\subsection{Acknowledgments}
The authors thank Jen Hom for asking whether our formula could be used to compute the torsion order, a question that led to this project. We thank Jonathan Hanselman for helpful correspondence about Theorem~\ref{thm:standard_decomp}. HZ is supported by an AMS-Simons travel grant. IZ is supported by a Sloan Fellowship and NSF Grants DMS-2204375 and DMS-2543889.

\section{Preliminaries on Heegaard Floer homology}\label{sec: preli}
In this section, we review the necessary preliminaries on Heegaard Floer homology for our computations.
\subsection{Knot and link Floer homology}
Knot Floer homology is an invariant of knots due independently to Ozsv\'{a}th and Szab\'{o} \cite{OSKnots} and Rasmussen \cite{RasmussenKnots}.  Link Floer homology generalizes this theory to links, and is due to Ozsv\'{a}th and Szab\'{o} \cite{OSLinks}.

Given a knot in $S^3$, the \emph{knot Floer complex} takes the form of a finitely generated, free chain complex $\cCFK(K)$ over the two-variable polynomial ring 
\[
R=\bF[W,Z]
\]
where $\bF=\bZ/2\bZ$. We write $U=WZ.$

The chain complex $\cCFK(K)$ has a $\Z\times \Z$ Maslov bigrading, denoted $(\gr_{\ws}, \gr_{\zs})$. Furthermore,
\[
(\gr_{\ws},\gr_{\zs})(W)=(-2,0)\quad \text{and} \quad (\gr_{\ws}, \gr_{\zs})(Z)=(0,-2).
\]
The \emph{Alexander grading} is defined by the equation
\[
A:=\frac{ \gr_{\ws}-\gr_{\zs}}{2}.
\]

We say a basis of $\cCFK(K)$ is  \emph{reduced} if no term in the differential has coefficient $1 \in \bF[W,Z]$.

If $L\subset S^3$ is an $n$-component link, we consider a version of link Floer homology, denoted $\cCFL(L)$, which takes the form of a finitely generated, free chain complex $\cCFL(L)$ over the ring 
\[
\bF[W_1,Z_1,\dots, W_n,Z_n].
\]
This chain complex has a $\Z\times \Z$-valued Maslov bigrading $(\gr_{\ws}, \gr_{\zs})$ and an $n$-component Alexander grading $A=(A_1,\dots, A_n)$. The Alexander grading takes values in the set
\begin{equation}
	\bH(L)=\prod_{i=1}^n\left( \Z+\frac{\lk(L_i,L\setminus L_i)}{2}\right),
	\label{eq:Alexander-grading-def}
\end{equation}
where $L_1,\dots, L_n$ denote the components of $L$. The Alexander grading satisfies
\[
\sum_{i=1}^{n}A_i = \frac{\gr_{\ws}-\gr_{\zs}}{2}.
\]

The differential has the following grading:
\[
(\gr_{\ws},\gr_{\zs})(\d)=(-1,-1)\quad \text{and} \quad A(\d)=(0,\dots, 0). 
\] 
Since each $W_iZ_i$, for $i=1,\cdots,n$ induces the same map on homology, the homology group  $\cHFL(L):=H_*(\cCFL(L))$ can be viewed as a module over $\bF[U]$ where  $U$ acts by any of $W_1Z_1,\dots,W_nZ_n$. Since the differential $\d$ preserves the Alexander grading, there is a submodule $\cHFL(L,\ve{s})$ for each $\ve{s}\in \bH(L).$

The following class of links will be important for our purposes:

\begin{define} A link $L$ is called an \emph{L-space link} if $S^3_\Lambda(L)$ is an L-space for all sufficiently large positive integer framing $\Lambda$.  By \cite[Theorem 12.1]{MOIntegerSurgery}, $L$ is an L-space link if and only if $\cHFL(L,\ve{s})\cong \bF[U]$ for all $\ve{s}\in \bH(L).$
\end{define}
\subsection{The $H$-function}
The $H$-function of a link $L\subset S^3$ takes the form of a map
\[
H_L\colon \bH(L)\to \Z^{\ge 0}.
\]
It is defined by the equation
\[
H_L(\ve{s})=-\frac{1}{2}\max \{ \gr_{\ws}(\xs): \xs\in \cHFL^-(L), U^n \xs\neq 0 \forall n\}.
\]
When $L$ is an L-space link,  the $H$-function has a particularly simple description:  \[H_L(\ve{s})=-\frac{1}{2}\gr_{\ws}(\xs_{\ve{s}})\] where $\xs_{\ve{s}}$ is the generator of $\cHFL(L,\ve{s})\cong \bF[U]$. Moreover, in this case, by \cite{GorNem15}*{Theorem~2.2.11} the $H$-function is determined by the multivariable Alexander polynomials of sublinks of $L$. 

In the case when $L$ is a two-component L-space link, we make the following definition, whose importance will be clear when we define a bordered bimodule associated to $L.$
\begin{define}\label{def: R_t}
	If $L$ is a two-component L-space link with linking number $\ell$ and $H$-function $H_L(t,r)$, we set
	\[
	R_t := \max\left\{r\in \Z+\ell/2 \middle|  \begin{array}{l}  H_{L}(t,r+1)=H_L(t,r)\text{ and}\\   H_L(t,r-1)=H_L(t,r)+1\end{array} \right\}. 
	\]
\end{define}
By \cite[Lemma 4.9]{CZZApp}, the function $R_t$ is nondecreasing for $t \leq \frac{\ell}{2}$ and nonincreasing for $t \geq \frac{\ell}{2}$. Moreover, $R_t = g_3(P) - \frac{\ell}{2}$ for $t \ll 0$, and $R_t = g_3(P) + \frac{\ell}{2}$ for $t \gg 0$.
\subsection{Standard complexes and local system complexes} 

In this section we recall several decomposition results of complexes over $\hat{\cR}=\bF[W,Z]/U$. These results are due to a number of authors. See \cite{DHSTmore}, \cite{KWZMnemonic}, \cite{HanselmanMinus} for closely related results. See also \cite{KWZKhovanov} \cite{HRWImmersedCurves} and  \cite{PopSnake}. We will mostly follow Hanselman's description.

Before we state the decomposition results we will use, we recall the following informal terminology. Suppose $(C,\d)$ is a free, finitely generated chain complex over a ring $R$. If $a,a'\in B$ are basis elements, we say that there is an \emph{arrow} from $a$ to $a'$, weighted by $r\in R$, if the $a'$ component of $\d(a)$ is equal to $r$.

   \begin{define}
       Let $\hat\cR=\bF[W,Z]/(U)$. If $n_1,\dots, n_m$ are positive integers, and 
$       \sigma_1,\dots, \sigma_m$ are in $\{-1,1\}$,  we define the \textit{standard complex} $\cC(\sigma_1 n_1,\dots, \sigma_m n_m)$ to be the free $\hat{\cR}$ chain complex with generators $a_0,\dots, a_m$ and with differential as follows: for each $1\le i\le m$, there is an arrow between $a_i$ and
$a_{i-1}$ weighted by $Z^{n_i}$ if $i$ is odd, and weighted by
$W^{n_i}$ if $i$ is even, for some $n_i>0$. If $\sigma_i=1$, this arrow points from $a_i$ to $a_{i-1}$, and if $\sigma_i=-1$, the arrow points from $a_{i-1}$ to $a_i$.
\end{define}

A standard complex $\cC(\sigma_1n_1,\cdots,\sigma_m n_m)$ is called a \emph{positive staircase} if $\sigma_i>0$ for all odd $i$ and $\sigma_i<0$ for all even $i$. See below as an example, namely $\cC(n_1,-n_2)$.
    \begin{equation*}
\begin{tikzcd}[column sep=1 cm, row sep=0.8 cm]
& a_2 & a_1 \ar[d, "Z^{n_1}"] \ar[l, "W^{n_2}"']
\\
&&a_0
\end{tikzcd}
    \end{equation*}
Similarly, a \emph{negative staircase} is when $\sigma_i<0$ for all odd $i$ and $\sigma_i>0$ for all even $i$.

By \cite[Theorem 1.3]{DHSTmore}, for any $K\subset S^3$, $\cCFK(K)^{\hat\cR} :=\cCFK(K)/(U)$ is locally equivalent to a unique standard complex $\cC(\sigma_1n_1,\cdots,\sigma_m n_m)$. In fact, they prove that $\cCFK(K)^{\hat\cR}$ contains this standard complex as a direct summand \cite{DHSTmore}*{Corollary~6.2}. The standard complex of a knot $K$ in $S^3$ satisfies the symmetry given by $\sigma_i n_i = - \sigma_{m+1-i} n_{m+1-i}$ for $i=1,\cdots,m$.
We call $a_0$ the \emph{ending generator} and $a_m$ the \emph{starting generator} of the standard complex to match the terminology used in \cite{HLPUnknotting} and the arrow between $a_0$ and $a_1$ (resp.~$a_m$ and $a_{m-1}$) the \emph{ending arrow} (resp.~\emph{starting arrow}). 

 The invariant
$\varepsilon(K)$ can be read off from the standard complex as follows: $\varepsilon(K)=0$   if and only if $m=0$; if $m>0$, then $\varepsilon(K)=1 \Leftrightarrow \sigma_1>0$, and $\varepsilon(K)=-1 \Leftrightarrow \sigma_1<0$.

Write $\cC$ for the standard complex $\cC(\sigma_1n_1,\cdots,\sigma_m n_m)$. Then
\begin{align*}
H_*(\cC/(Z)) &\cong \bF[W] \oplus \bigoplus_{i=1}^{m/2} \bF[W]/W^{n_{2i}} \\  H_*(\cC/(W)) &\cong \bF[Z] \oplus \bigoplus_{i=1}^{m/2} \bF[Z]/Z^{n_{2i-1}}
\end{align*}
where $a_0$ is the free generator in $\cC/(Z)$ and $a_m$ is the free generator in $\cC/(W)$.

Next, we discuss standard models for complexes which have homology which is torsion over both $\bF[W]$ and $\bF[Z]$.

\begin{define}
\label{def:local-system-complex} (Local system complex) Let $\vec{n}=( n_1,n_2,\dots,n_{2p})$ be a sequence of positive integers and let $\vec{\sigma}=(\sigma_1,\dots, \sigma_{2p})$ be a sequence of numbers in $\{-1,+1\}$. Assume, furthermore, that 
\begin{equation}\label{eq: local_system}
\sum_{i=1}^p \sigma_{2i}n_{2i}=\sum_{i=1}^p\sigma_{2i-1} n_{2i-1}=0.
\end{equation}
Let $w$ be a positive integer and $A\in \GL_{w}(\bF)$. Then there is a type-$D$ module $\cL=\cL(w,\vec{n},\vec{\sigma},A)$ over $\hat{\cR}$, defined as follows. The underlying vector space is given by $\bigoplus_{i=1}^{2p} X_i$ where each $X_i$ is identified with $\bF^{w}$. The differential is as follows. If $i\neq 1$, then the differential contains a component mapping between $X_i$ and $X_{i-1}$ which is weighted by $W^{n_i}$ if $i$ is even and $Z^{n_i}$ if $i$ is odd. If $\sigma_i>0$, then this component maps from $X_i$ to $X_{i-1}$, and if $\sigma_i<0$, this component maps from $X_{i-1}$ to $X_i$. Finally, if $\sigma_1=1$, then there is a component of $\d$ from $X_1$ to $X_{2p}\otimes \hat{\cR}$ which is given by $A\otimes Z^{n_1}$. If $\sigma_1=-1$, then there is a component of $\d$ which maps $X_{2p}$ to $X_1$ by $A^{-1}\otimes Z^{n_1}$. We refer to any complex $\cL$ constructed in the above manner as a \emph{local system complex}. 
\end{define}

\begin{thm}\label{thm:standard_decomp} If $K$ is a knot in $S^3$, then there is a decomposition
\[
    \cCFK(K)^{\hat\cR} \simeq \cC \oplus \cL_1 \oplus \cdots \oplus \cL_k
\]   where $\cC$ is a standard complex and each $\cL_i$ is a local system complex.
\end{thm}
\begin{proof}[Proof sketch]
We give a very brief sketch of the proof of Theorem~\ref{thm:standard_decomp}, and describe how it follows from Hanselman's work \cite{HanselmanMinus}. We note it also follows from \cite{KWZKhovanov}*{Theorem~5.14}. Both works build on \cite{HRWImmersedCurves}. Hanselman describes a way of representing $\cCFK(K)^{\hat{\cR} }$ as an immersed multicurve  $\Gamma_K$ in a punctured torus. The immersed multicurve is compact, and equipped with a bounding cochain. Informally, a bounding chain is a collection of self intersections of the immersed curve with itself where a holomorphic curve is allowed to turn. Hanselman proves in \cite{HanselmanMinus}*{Proposition~7.1} that the bounding chain can be taken to be of \emph{local system type} \cite{HanselmanMinus}*{Definition~6.3}. For such bounding chains, there are no distinguished points between different components of an immersed curve. Furthermore, the only self intersections occur on components which are not primitive (\textit{i.e.}, homotopic to $k$ times another curve for some $|k|>1$). Hanselman proves that such immersed curves can be taken to consist of $k$-parallel copies of a single immersed curve, connected by a simple crossing region (see \cite{HanselmanMinus}*{Figure~25}). 

 To recover the hat theory, one takes a distinguished meridian in the torus $\mu$ that runs through the puncture. One fills in the puncture with a disk containing two basepoints $w$ and $z$ on either side of the puncture. Write $\mu$ also for the Lagrangian in the non punctured torus. Then
 \[
 \cCFK(K)^{\hat{R}}\iso \CF(\Gamma_K,\mu,w,z).
\]
 In the above, $\CF(\Gamma_K,\mu,w,z)$ is the free $\hat{\cR}$-module generated by intersection points between $\Gamma_K$ and $\mu$. The differential counts immersed bigons $u$ which are weighted by $W^{n_w(u)}Z^{n_z(u)}$.

Write $\mu'$ for a parallel copy of $\mu$, which is disjoint from the region containing $w$ and $z$. If $K$ is a knot in $S^3$, then we can apply a regular homotopy to $\Gamma_K$ so that only a single component intersects $\mu'$, and furthermore there is only a single intersection point. (See \cite{HanselmanMinus}*{Section~8}). The component of $\Gamma_K$ intersecting $\mu'$ does not have a local system. 

We may assume, via regular homotopy, that there are no bigons bounded by $\mu$ and $\Gamma_K$ which do not contain one of $w$ or $z$.  We also consider the image of $\Gamma_K$ in the complement of $\mu'$, which we view as $S^1\times [-1,1]$. This multi-curve can be lifted to an immersed multi-curve in the infinite strip $\R\times [-1,1]$, which we denote $\gamma_K$. By convention, we will assume that $z$ occurs in $(-1,0)\times S^1$ and $w$ is in $(0,1)\times S^1$. 

We claim that the decomposition in Theorem~\ref{thm:standard_decomp} can be read off of the curve $\Gamma_K$, put in this position. The standard complex $\cC$ corresponds to the non-compact component of $\gamma_K$, while each compact component of $\gamma_K$ gives one of the local system complexes $\cL_i$. We see this as follows. Consider the non-compact component $\gamma_0$ of $\gamma_K$. Consider the components of $\gamma_0\setminus \mu$. Call these \emph{segments}. There is a first segment of $\gamma_K$ which goes from $\{-1\}\times \R$ to $\mu$. This first intersection point corresponds to $a_0$. The second segment could travel from $a_0$ to $\{1\}\times \R$, in which case $\cC$ has a single generator. Otherwise, the next segment of $\g_0$ travels from $a_0$ to some other intersection point $a_1$. Since the bounding chain  $\g_K$ satisfies the Maurer-Cartan condition (\cite{HanselmanMinus}*{Definition~3.18}) and is of local system type, there can be no immersed monogons with boundary on $\g_0$ which do not contain $w$ or $z$. In particular, the segment from $a_0$ to $a_1$ cannot have any self intersections. There is clearly a bigon from $a_0$ to $a_1$ which is weighted by some power of $W$. Continuing this procedure builds the complex $\cC$. One may check that there are no other immersed bigons bounding $\g_0$ and $\mu$ which cover only $w$ or $z$ other than the ones constructed above.

 Applying the analogous procedure to the non compact components of $\Gamma_K$ yields the local system complexes $\cL_1,\dots, \cL_n$. 
\end{proof}

It will be important for our purposes to realize that the summands $\cC$ and $\cL_i$ in Theorem~\ref{thm:standard_decomp} satisfy certain additional restraints, which result from $\cCFK(K)^{\hat{\cR}}$ extending to a complex over $\cR$.  It is shown in \cite[Proposition 2.8]{HLPUnknotting} that for three consecutive arrows in one direction,  the middle arrow must have length at least two. More precisely, suppose that in one of the summands in Theorem~\ref{thm:standard_decomp} there exist generators $a_0,a_1,a_2,a_3$ such that there is an arrow from $a_i$ to $a_{i-1}$ weighted by $Z^{n_i}$ for $i=1,3$ and $W^{n_2}$ for $i=2$. Then $n_2>1$. This configuration may occur either in the standard complex or in one of the local system complexes (in the case of a local system, we assume that each generator $a_i$ is in the vector space $X_i$  from Definition~\ref{def:local-system-complex}). The indices do not specify the positions of the generators in the standard complex.
    \begin{equation*}
\begin{tikzcd}[column sep=1 cm, row sep=0.8 cm]
&   a_3 \ar[d, "Z^{n_3}"] 
\\
a_1\ar[d, "Z^{n_1}"]&a_2 \ar[l, "W^{n_2}"]\\
a_0&
\end{tikzcd}
    \end{equation*}
This result can be readily extended to a longer sequence of consecutive arrows, which we state below.

\begin{lem}[Compare \cite{HLPUnknotting}*{Proposition~2.8}] \label{lem: snake_structure}
   Suppose $\{a_i:~0\leq i\leq k\}$ are the generators in a summand in the decomposition of Theorem~\ref{thm:standard_decomp} of $\cCFK(K)$, such that there is an arrow from $a_i$ to $a_{i-1}$ weighted by $Z^{n_i}$ (resp.~$W^{n_i}$) whenever $1\leq i\leq k$ is odd (resp.~even) for some $n_i>0$. Then $n_i>1$ for every $2\leq i \leq k-1.$ Moreover, if $a_0$ is the ending generator of the standard complex, then $n_1>1$.
\end{lem}
\begin{proof}
  Suppose $a_i$ for $0\le i\le k$ are generators in a local system complex $\cL = \oplus^{2p}_{i=1}X_i$, such that
 each $a_i$ generates a copy of $\bF$ in  $X_{j+i}$ for some $j\in \{1,3,\cdots,2p-1\}$, where $X_{i}=X_{i-2p}$ if $i>2p$. 
Since $\sigma_{j+i} >0$ for $i=1,\cdots,k$,  in order for \eqref{eq: local_system} to hold, there must be $1\leq i\leq 2p$ that does not take value in $\{j,\cdots,j+k\}$.  Therefore by a change of basis  we may assume $j=1$ and $j+k<2p.$ 
It follows that the $\bF[Z]$-span (resp.~$\bF[W]$-span) of $\{a_i,a_{i-1}\}$ forms a direct summand in the $\cCFK(K)/(W)$ (resp.~$\cCFK(K)/(Z)$) if $i$ is odd (resp.~even) for $1\leq i\leq k.$ 
The same condition holds by definition if $a_i$ for $0\le i\le k$ are generators in the standard complex.
  
The result follows from the same arguments as in the proof of
Proposition~2.8 of \cite{HLPUnknotting}, which we write down for completeness.

Suppose to the contradiction that $n_2=1.$
We have
\[
\partial a_2 = Wa_1 + Ub
\]
for some $b \in \cCFK(K).$  Since $\partial^2=0$, $\partial b$ must contain $Z^{n_1-1}a_0$ as a non-trivial term. This contradicts the fact that the $\bF[Z]$-span of $\{a_1,a_0\}$ forms a direct summand in $\cCFK(K)/(W)$. A similar argument proves the claim for other $n_i>1.$ 

Suppose $a_0$ is the ending generator and $n_1 =1.$ 
We have 
\[
\partial a_2 = W^{n_2}a_1 + Ub'
\]
for some $b' \in \cCFK(K).$  Since $\partial^2=0$, $\partial b'$ must contain $W^{n_2-1}a_0$ as a non-trivial term. This contradicts the fact that $a_0$ is the ending generator.
\end{proof}

\subsection{The knot Floer torsion order}\label{subsec: tor_ord}

\begin{define}\label{def: torsion_order} 
Define the \emph{knot Floer torsion order} of a knot $K$, denoted $\Ord(K)$, to be the minimal $n\in \N$ such that $Z^n$ acts trivially on the torsion submodule of $H_*(\cCFK(K)/(W)).$ It is easy to see that this is equal to the length of the longest arrow in the decomposition in Theorem~\ref{thm:standard_decomp}.
\end{define}

 \begin{lem}[Lemma 7 in \cite{PetkovaThin}] \label{lem: cfk_ord=1}
 If $\Ord(K)=1$, then $\cCFK(K)^{\hat\cR} \simeq \cC \oplus \cB_1  \oplus \cdots  \oplus \cB_k$ for $k \geq 0$ where $\cC$ is a positive or negative staircase and each $\cB_i$ is a \emph{length-one box}, defined by $\cB_i=\cB_i(1,\vec{n},\vec{\sigma},\Id)$ with $\vec{n}=(1,1,1,1)$ and  $\vec{\sigma}=(1,1,-1,-1).$ 
\end{lem}   
\begin{proof}
    Consider the $\delta$-grading defined by $\delta=A-\gr_W.$ Since $\Ord(K)=1$, the differential of $\cCFK(K)^{\hat\cR}$ consists of arrows weighted by either $W$ or $Z$ and one readily checks that both preserve the $\delta$-grading. It follows that $\cCFK(K)^{\hat\cR}$ splits  over  $\delta$-gradings. A complex is thin if and only if it consists of elements of the same $\delta$-grading. The result then follows from applying   \cite[Lemma 7]{PetkovaThin} to each $\delta$-graded summand of $\cCFK(K)^{\hat\cR}$.
\end{proof}

We will make use of the following simple fact about the torsion order:

\begin{lem}\label{lem: torsion_order_criterion} Suppose that $(A,\d_A)$, $(B,\d_B)$ and $(C,\d_C)$ are chain complexes of graded $\bF[Z]$-modules and $C$ decomposes as the mapping cone of a chain map $f\colon A\to B$. Further, assume that $H_*(B)$ is torsion (i.e. for each $x\in B$, there is some $N$ such that $Z^N x=0$). Suppose that $a\in A$ is a cycle such that $Z^n[a]=0\in H_*(A)$ and $Z^{n-1}[a]\neq 0\in H_*(A)$. Suppose that $Z^j \cdot [f(a)]=0\in H_*(B)$. Then 
    \[
    \Ord(H_*(C)) \ge n-j.
    \]
    \end{lem}
    \begin{proof}Assume $n>j$, since otherwise the claim is trivial. We consider the exact triangle 
    \[
    \begin{tikzcd}
   H_*(B)\ar[r, "i_*"] & H_*(C)\ar[r,"p_*"] & H_*(A)\ar[r, "f_*"] &H_*(B)
    \end{tikzcd}
    \]
   where $i\colon B\to C$ and $p\colon C\to A$ are the natural chain maps. The element $Z^j[a]\in H_*(A)$ maps trivially by $f_*$ to $H_*(B)$, and therefore may be written as $p_*([c])$ for some $c\in H_*(C)$. We first claim that $[c]$ represents a torsion element in $H_*(C)$, i.e., $Z^N\cdot [c]=0$ for some large $N$. If $N>n-j$, then $Z^N\cdot [c]$ is mapped by $p_*$ to $Z^{N+j}[a]$, which is zero. Hence $Z^N \cdot [c]$ is in the image of $H_*(B)$. However $H_*(B)$ is a torsion module, so $Z^N\cdot [c]$ must be zero for large $N$, establishing that $[c]$ is torsion. We next claim that $Z^{n-j-1}[c]\neq 0\in H_*(C)$. To see this observe that $p_*(Z^{n-j-1}[c])=Z^{n-1}[a]$ which is non-zero, as claimed. Therefore $\Ord(H_*(C))\ge n-j$.  
    \end{proof}

\section{Background on the surgery algebra and surgery modules}

\subsection{Type-$D$ and $A$ modules}
In this section, we  recall Lipshitz, Ozsv\'{a}th and Thurston's formalism of type-$D$, type-$A$ and type-$AA$ modules  \cite{LOTBordered} \cite{LOTBimodules}.

\begin{define} Suppose that $A$ is an associative algebra over a ring $\ve{k}$. A \emph{right type-$D$ module} $C^A$ consists of a right $\ve{k}$-module $C$ equipped with a $\ve{k}$-linear map $\delta^1\colon C\to C\otimes_{\ve{k}} A$ such that
\[
(\bI_{C}\otimes \mu_2)\circ (\delta^1\otimes \bI_{A})\circ \delta^1=0.
\]
\end{define}

\begin{rem} When $\ve{k}=\bF=\Z/2$, we conflate a type-$D$ structure over $A$ with a free chain complex over $A$.
\end{rem}
\begin{define}Suppose that $B$ is an associative algebra over a ring $\ve{k}$. A type-$A$ module, or an \emph{$A_\infty$} module $M_B$ is a right $\ve{k}$-module $M$ equipped with $\ve{k}$-linear maps
\[
m_{n+1}\colon M \otimes_{\ve{k}}\underbrace{B\otimes_{\ve{k}}\cdots \otimes_{\ve{k}} B}_n \to M
\]
for $n\ge 0$. These maps are required to further satisfy
\[
\begin{split}
&\sum_{j=0}^n m_{n-j}( m_{j+1}(\xs, b_1,\dots, b_{j}), b_{j+1},\dots, b_n))\\
+&\sum_{j=1}^{n-1} m_{n}( \xs, b_1,\dots, b_j b_{j+1},\dots, b_n)=0
\end{split}
\]
\end{define}
An $A_\infty$-module $M_B$ with $m_j=0$ for $j>1$ can be viewed as a regular chain complex over $B$, where the differential is given by $m_0$ and the multiplication over $B$ by $m_1$.
\begin{define}
If $A$ and $B$ are associative algebras over a ring $\ve{k}$, an \emph{$AA$-bimodule} ${}_{A} M_{B}$ consists of a $\ve{k}$-bimodule $M$, equipped with structure maps 
\[
m_{i|1|j}\colon A^{\otimes i}\otimes M\otimes B^{\otimes j}\to M
\]
which satisfy the following $A_\infty$-associativity relation:
\[
\begin{split}
0=&\sum_{\substack{0\le i\le n\\ 0\le j\le m}} m_{n-i|1|m-j}(a_n,\dots, m_{i|1|j}(a_i,\dots, a_1, \xs,b_1,\dots, b_j),\dots, b_m)\\
+&\sum_{i=1}^{n-1} m_{n-1|1|m}(a_n,\dots, a_{i+1}a_i,\dots, a_1, \xs,b_1,\dots, b_m)\\
+&\sum_{i=1}^{m-1} m_{n|1|m-1}(a_n,\dots, a_1, \xs, b_1,\dots, b_ib_{i+1},\dots, b_m) 
\end{split}.
\]
\end{define}
Given a type $D$-module $C^A$ and an $AA$-bimodule ${}_{A} M_{B}$,  we can form the \emph{box tensor product} $C^A \boxtimes {}_{A} M_{B}$,  a right type $A$-module over $B$. The underlying module is $C\otimes M $, and the structure maps are given by
\[
m_{n+1}(\xs\otimes \ys,b_1,\cdots,b_n)=\sum_{k=0}^\infty(\bI_{C}\otimes m_{k,1,n})\big((\delta^k(\xs)\otimes \ys) \otimes b_1 \otimes \cdots \otimes b_n\big).
\]
For example, $m_1$ is defined by the following diagram
\[ m_1=\quad
\begin{tikzcd}[column sep=0 cm, row sep=0.2cm]
     C  \ar[ddd]&[0.2cm]  \boxtimes& M\ar[dd] \\[0.4cm]
     & & \\
     & & m_{0,1,0} \ar[d]  \\[0.4cm]
     C &  \boxtimes  & M\  
	 \end{tikzcd}
     +\quad
     \begin{tikzcd}[column sep=0 cm, row sep=0.2cm]
     C  \ar[d]&[0.2cm]  \boxtimes& M\ar[dd] \\[0.3cm]
     \delta^1 \ar[dd]\ar[drr]& & \\
     & & m_{1|1|0} \ar[d]  \\[0.3cm]
     C &  \boxtimes  & M\  
	 \end{tikzcd}
     +\quad
     \begin{tikzcd}[column sep=0 cm, row sep=0.3cm]
     C  \ar[d]&[0.2cm] \boxtimes& M\ar[dd] \\[0.1cm]
     \delta^1 \ar[d]\ar[drr]& &\\
      \delta^1 \ar[d]\ar[rr] & & m_{2,1,0} \ar[d]  \\[0.1cm]
     C &  \boxtimes  & M\   
	 \end{tikzcd}
     +\quad
    \cdots
     \]
     
     \subsection{The surgery algebra}

     In this section, we recall the \emph{surgery algebra}, defined in \cite{ZemBordered}, as well as some formalism about modules over the surgery algebra.

   The second author reformulated the link surgery formula of Ozsv\'{a}th--Szab\'{o} \cite{OSIntegerSurgeries} and Manolescu--Ozsv\'{a}th \cite{MOIntegerSurgery} in terms of an associative algebra $\cK$ called \emph{the surgery algebra}. We presently recall the definition of $\cK$. It is an associative algebra over the ring of two idempotents $\ve{I}=\ve{I}_0\oplus \ve{I}_1$ (where each $\ve{I}_i=\bF=\Z/2$). Furthermore
\[
\ve{I}_0\cdot \cK\cdot \ve{I}_0=\bF[W,Z],\quad \ve{I}_0 \cdot \cK \cdot \ve{I}_1=0
\]
\[
\ve{I}_1\cdot \cK\cdot \ve{I}_0=\bF[U,T,T^{-1}] \otimes \langle \sigma, \tau\rangle \quad \text{and}, \quad \ve{I}_1\cdot \cK\cdot \ve{I}_1=\bF[U,T,T^{-1}].
\]
The algebra is subject to the relations that
\[
\sigma W=UT^{-1} \sigma, \quad \sigma Z=T \sigma,\quad \tau W=T^{-1} \tau,\quad \tau  Z=UT  \tau.
\]
Define $\hat\cK$ by quotienting $U$ in $\cK$.

\subsection{Surgery modules and bimodules}

Ozsv\'{a}th and Szab\'{o} \cite{OSIntegerSurgeries} proved a very useful surgery formula for Heegaard Floer homology. They defined a chain complex $\bX_\lambda(Y,K)$ which computes $\ve{CF}^-(Y_\lambda(K))$ whenever $K$ is a rationally null-homologous link in $Y$ and $\lambda$ is a Morse (i.e. longitudinal) framing. Here $\ve{\CF}^-(Y_{\lambda}(K))$ denotes the completion $\CF^-(Y_{\lambda}(K))\otimes_{\bF[U]} \bF\llsquare U\rrsquare$. 

The second author encoded the surgery formula $\bX_{\lambda}(Y,K)$ into the framework of type-$D$ and $A$ modules, as follows. Firstly, there is a type-$D$ structure $\cX_{\lambda}(Y,K)^{\cK}$ over $\cK$, which encodes the data of the surgery formula. This should be thought of as the invariant for $Y\setminus \nu(K)$. See \cite{ZemBordered}*{Section~8.5} for details on the construction. See \cite{HMSZ_Naturality}*{Section~5.3} for an exposition of the construction for knots in $S^3$.

Next, there is a type-$A$ module ${}_{\cK} \cD$, whose underlying $\ve{I}$-module consists of $\bF[W,Z]$ (in idempotent 0) and $\bF[U,T,T^{-1}]$ (in idempotent 1). The actions of $\ve{I}_0\cdot \cK\cdot \ve{I}_0$ and $\ve{I}_1\cdot \cK\cdot \ve{I}_1$ on ${}_{\cK} \cD$ are given by ordinary multiplication. Furthermore,
\[
m_2(\sigma, W^iZ^j)=U^i T^{j-i}\quad \text{and} \quad m_2(\tau, W^i Z^j)=U^j T^{j-i}.
\]
The maps $m_j$ vanish $j>2$. We think of ${}_{\cK} \cD$ as being the bimodule for a 0-framed solid torus (i.e. the complement of a 0-framed unknot in $S^3$).

The complex $\bX_{\lambda}(Y,K)$ is recovered by a box tensor product
\[
\bX_{\lambda}(Y,K)\simeq \cX_{\lambda}(Y,K)^{\cK}\boxtimes{}_{\cK} \cD. 
\]

\begin{rem} Although suppressed from our notation and the above discussion, we complete the algebra $\cK$ and all of the modules above at the ideal $(U)$. For the present paper, this is of little consequence since we mostly focus on the quotient $\cK/U$.
\end{rem}

The link surgery formula can be similarly reinterpreted in terms of modules over the surgery algebra $\cK$. Manolescu and Ozsv\'{a}th describe a chain complex $\cC_{\Lambda}(L)$, called the \emph{link surgery complex}, whose homology coincides with that of $\ve{\CF}^-(Y_{\Lambda}(L))$, when $Y$ is an integer homology 3-sphere and $L\subset Y$ is a link with integral framing $\Lambda$. The second author described a type-$D$ module
\[
\cX_{\Lambda}(Y,L)^{\cK\otimes_{\bF}\cdots \otimes_{\bF} \cK}
\]
which encodes the link surgery formula. See \cite{ZemBordered}*{Section~8.6} for the construction. The chain complex $\cC_{\Lambda}(L)$ is obtained by tensoring $|L|$ copies of the module ${}_{\cK} \cD$ to $\cX_{\Lambda}(Y,L)^{\cK\otimes\cdots \otimes \cK}$.

Finally, we will discuss the connected sum formulas. These are described in \cite{ZemBordered}*{Section~12}. Suppose that $L\subset Y$ is a 2-component link. The construction described above gives a type-$D$ module $\cX_{\Lambda}(Y,L)^{\cK\otimes \cK}$. In \cite{ZemBordered}*{Section~8.4}, the second author describes a type-$A$ bimodule ${}_{\cK\otimes \cK} [\bI^{\Supset}]$. We tensor $\cX_{\Lambda}(Y,L)^{\cK\otimes \cK}$ with this bimodule (along just one $\cK$ factor from each bimodule) to produce a $DA$-bimodule which we denote by ${}_{\cK} \cX_{\Lambda}(Y,L)^{\cK}$. Typically, we omit the $\Lambda$ subscript from the notation to indicate that $\Lambda=(0,0)$ or that $\Lambda$ is implicit from context.

We now suppose that $L_1$ and $L_2$ are 2-component links in $Y_1$ and $Y_2$, and we consider the $DA$-bimodules ${}_{\cK} \cX_{\Lambda_i}(Y_i,L_i)^{\cK}$. Write $L_1=K_1\cup K_1'$ and $L_2=K_2\cup K_2'$.

 It follows from \cite{ZemBordered}*{Theorem~12.1} and \cite{ZemBordered}*{Proposition~13.2} that the box tensor product
\[
{}_{\cK} \cX_{\Lambda_1}(Y_1,L_1)^{\cK}\boxtimes {}_{\cK} \cX_{\Lambda_2}(Y_2,L_2)^{\cK}
\]
is homotopy equivalent to ${}_{\cK} \cX_{\Lambda'}(Y',L')^{\cK}$ where $Y'=(Y_1\# Y_2)_{\lambda'}(K_1'\# K_2)$ and $L'=K_1\cup K_2'$. Additionally, $\lambda'$ is the sum of the framings on $K_1'$ and $K_2$ and $\Lambda'$ denotes the restriction of $\Lambda_1$ and $\Lambda_2$ to $K_1$ and $K_2'$. The analogous pairing theorems hold for connected sums of links with other numbers of components.

\section{Background on L-space satellite operations}\label{sec: L-space-satellite}

In this section, we review the notion of an L-space satellite operation and the results from \cite{CZZ} and \cite{CZZApp}. 

In Section~\ref{sec:background-L-space-basic}, we recall the definition of an L-space satellite operation and the bimodules ${}_{\cK} \cX(L_P)^{\bF[W,Z]}$ associated to an L-space satellite operation. 

In Section~\ref{subsec: bimodule_cK_diamond} we consider the reduction, where we set $W=0$. This module admits a simple model, denoted ${}_{\cK}\cX^{\diamond}(L_P)_{\bF[Z]}$,
which has the advantage that only $m_{1|1|0}$ and $m_{0|1|1}$ are nontrivial (i.e. is a genuine bimodule, instead of an $A_\infty$-bimodule). We review its definition and prove an important lemma concerning its structure. We also introduce some helpful notation.
In Section \ref{subsec: complex_Pk}, we assemble the pieces and complete the setup for the computations. 

In Section~\ref{subsec: diagrammatic}, we introduce a diagrammatic description that will be used in the proofs.

\subsection{L-space satellite operators}
\label{sec:background-L-space-basic}
Given a pattern $P \subset S^1 \times D^2$ and a companion knot $K \subset S^3$, the satellite knot $P(K,\lambda)$ is defined to be the image of $P$ under an identification of $S^1 \times D^2$ with $\nu(K)$ that sends $S^1 \times \{*\}$ to the $\lambda$-framed longitude of $K$. We typically abuse notation and view $P$ as both a knot in the solid torus, as well as the knot in $S^3$ obtained by embedding the solid torus $S^1\times D^2$ into $S^3$  as a neighborhood of the $0$-framed unknot.

Denote $L_P=\mu\cup P$ where $\mu$ is a meridian of the solid torus.
The \emph{winding number} of $P$ is given by $\ell=\lk(P,\mu)$. We may assume $\ell\geq 0$ after reversing the string orientation of $P$ if needed.

		\begin{define} We say that a satellite operator $P$  is an \emph{L-space satellite operator} if the 2-component link $L_P$ is an L-space link, i.e., $S^3_{\Lambda}(L_P)$ is an L-space for all integral framings $\Lambda\gg (0,0)$. 
				\end{define}

To a 2-component link $L$ with integral framing $\Lambda$, the third author  \cite{ZemBordered} \cite{ZemExact} constructed a $DA$-bimodule ${}_{\cK} \cX_{\Lambda}(L)^{\cK}$. In \cite{CZZ}, we computed the reduction ${}_{\cK} \cX_{\Lambda}(L)^{\bF[W,Z]}$, obtained by restricting the right algebra to idempotent 0.

 The data of ${}_{\cK} \cX(L_P)^{\bF[W,Z]}$ can be encoded in terms of the diagram of the form shown in Figure~\ref{fig:DA-bimodule-schematic}. Therein, $t\in\Z+\lk(\mu,P)/2$. Each of the complexes $\cC_t$ and $\cS$  are staircase complexes. Each $T^{t} \cS$ denotes a copy of $\cS$. The staircase complexes $\cC_t$ and $\cS$ are determined by the $H$-function of $L$. 
\begin{figure}[h]
\begin{equation*}
	\begin{tikzcd}[labels=description, column sep=1.1cm, row sep=0cm]
		\cdots
		&[-1.2cm]\cC_{t-2}
		\ar[r, bend left, "Z|L_Z"]
		\ar[d,"\sigma|L_\sigma", bend left=15, pos=.4]
		\ar[d, "\tau|L_\tau", bend right=15, pos=.6]
		& \cC_{t-1}
		\ar[r, bend left, "Z|L_Z"]
		\ar[l, bend left, "W|L_W"]
		\ar[d,"\sigma|L_\sigma", bend left=15, pos=.4]
		\ar[d, "\tau|L_\tau", bend right=15, pos=.6]
		&\cC_{t}
		\ar[r, bend left, "Z|L_Z"]
		\ar[l, bend left, "W|L_W"]
		\ar[d,"\sigma|L_\sigma", bend left=15, pos=.4]
		\ar[d, "\tau|L_\tau", bend right=15, pos=.6]
		& \cC_{t+1}
		\ar[r, bend left, "Z|L_Z"]
		\ar[l, bend left, "W|L_W"]
		\ar[d,"\sigma|L_\sigma", bend left=15, pos=.4]
		\ar[d, "\tau|L_\tau", bend right=15, pos=.6]
		& \cC_{t+2}
		\ar[l, bend left, "W|L_W"]
		\ar[d,"\sigma|L_\sigma", bend left=15, pos=.4]
		\ar[d, "\tau|L_\tau", bend right=15, pos=.6]
		&[-1.2cm] \cdots
		\\[2cm]
		\cdots
		& T^{t-2}\cS
		\ar[r, bend left, "T|1"]
		\ar[loop below,looseness=20, "U|U"]
		& T^{t-1}\cS
		\ar[r, bend left, "T|1"]
		\ar[l, bend left, "T^{-1}|1"]
		\ar[loop below,looseness=20, "U|U"]
		&T^{t}\cS
		\ar[r, bend left, "T|1"]
		\ar[l, bend left, "T^{-1}|1"]
		\ar[loop below,looseness=20, "U|U"]
		&T^{t+1}\cS
		\ar[r, bend left, "T|1"]
		\ar[l, bend left, "T^{-1}|1"]
		\ar[loop below,looseness=20, "U|U"]
		&T^{t+2}\cS
		\ar[l, bend left, "T^{-1}|1"]
		\ar[loop below,looseness=20, "U|U"]
		&\cdots 
	\end{tikzcd}
\end{equation*}
\caption{A schematic encoding the $DA$-bimodule of an L-space link $L$. Additional data (not shown) necessary to encode the $DA$ module are the actions $\delta_3^1$.}
\label{fig:DA-bimodule-schematic}
\end{figure}
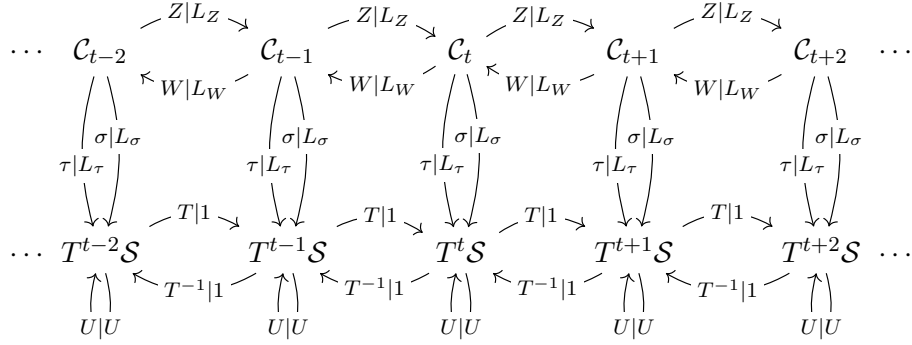

The diagram in Figure~\ref{fig:DA-bimodule-schematic} displays  the $\delta_2^1$  actions on this bimodule. For example, the arrows labeled $Z|L_Z$ indicate the action of $\delta_2^1(Z,-)$. We will write $L_Z$ for the induced map from $\cC_{t}$ to $\cC_{t+1}$.
Additionally, there will typically be several more $\delta_3^1$ actions which are not shown in the above diagram. We refer the reader to  \cite[Section 4.1]{CZZApp} for details of the construction. See also  \cite[Section 5]{CZZ}.

\subsection{Braided L-space patterns}

In this section, we prove the following:

\begin{prop}\label{prop: braided_property} Suppose that $L_P=\mu\cup P$ is an L-space pattern such that $P$ is braided about $\mu$, then $\cC_{t}\iso \cS$ for $t\ge \ell/2$ (in the language of \cite{CZZApp}, $N_{L_P}=\ell/2$), and $R_{\ell/2-1}<R_{\ell/2}$.
\end{prop}
\begin{proof} The claim that $\cC_t\iso \cS$ for $t\ge \ell/2$ follows from \cite{CZZApp}*{Proposition~7.3}. 

We consider $\widehat{\CFL}(L_P)$. This complex is obtained by setting $W_1=W_2=Z_1=Z_2=0$ in $\cCFL(L_P)$. Alternatively, we can take the $DA$-bimodule ${}_{\bF[W,Z]} \cX(L_P)^{\bF[W,Z]}$ and set $W=Z=0$ on the right, and tensor on the left with the Koszul complex
\[
\cK^{\bF[W,Z]}:=\begin{tikzcd}[labels=description] & \xs \ar[dl, "W"] \ar[dr, "Z"]\\
\ys\ar[dr, "Z"]&& \zs \ar[dl, "W"]\\
& \ws
\end{tikzcd}.
\]
The tensor product $\cK^{\bF[W,Z]} \boxtimes {}_{\bF[W,Z]} \cX(L_P)^{\bF[W,Z]}$ decomposes as the direct sum $\bigoplus_{s\in \ell/2+\Z} Y_s$, where
\[
Y_s:= \begin{tikzcd}[labels=description] & \cC_s \ar[dd, "h_{WZ}+h_{ZW}"] \ar[dl, "L_W"] \ar[dr, "L_Z"]\\
\cC_{s-1}\ar[dr, "L_Z"]&& \cC_{s+1}\ar[dl, "L_W"]\\
& \cC_{s}
\end{tikzcd}.
\]

We now consider the version obtained by setting $W=Z=0$ on the right. We write $\tilde{Y}_s$ and $\tilde{\cC}_s$ for the associated complexes.

Since $L_Z\colon \cC_{s}\to \cC_{s+1}$ is an isomorphism for all $s\ge \ell/2$, the  maximal $s$ where $\tilde{Y}_s$ is non-trivial is $s=\ell/2$. Since $L_Z\colon \cC_{\ell/2}\to \cC_{\ell/2+1}$ is an isomorphism, we see that $\tilde{Y}_{\ell/2}\simeq \tilde{X}_{\ell/2}$, where
\[
\tilde{X}_{\ell/2}:=\Cone(L_Z\colon \tilde{\cC}_{\ell/2-1}\to \tilde{\cC}_{\ell/2}).
\]

By \cite{MartinT26}*{Proposition~1}, we know that $H_*(\tilde{Y}_{\ell/2})$ has rank $2$.

Next, we observe that under our assumptions, the image of $\tilde{\cC}_{\ell/2-1}$ under $L_{Z}$ cannot have $\xs_{\ell/2}^{\top}$, the highest $A_2$-graded generator of $\tilde{\cC}_{\ell/2}$, in its image. This may be seen as follows. The $A_2$ Alexander grading of $\xs_{\ell/2-1}^{\top}$ is $R_{\ell/2-1}$, by definition.  By \cite{CZZApp}*{Lemma~4.9}, we have $R_{\ell/2-1}\le R_{\ell/2}$. Therefore, all other generators of $\tilde{\cC}_{\ell/2-1}$ have $A_2$ Alexander grading less than $R_{\ell/2-1}\le A_2(\xs_{\ell/2}^{\top})$. Since $L_Z$ preserves the $A_2$ Alexander grading, the only element of $\tilde{\cC}_{\ell/2-1}$ which could hit $\xs_{\ell/2}^{\top}$ under $L_Z$ is $\xs_{\ell/2-1}^{\top}$.  We note, however, that the $\gr_{\ws}$ grading of $\xs^{\top}_{\ell/2-1}$ is $1$, while the $\gr_{\ws}$-grading of $\xs^{\top}_{\ell/2}$ is $0$. This would imply that if $\xs^{\top}_{\ell/2}$ was a summand of $L_Z(\xs^{\top}_{\ell/2-1})$, it would need a factor of $W$, and therefore would vanish in the hat version. Furthermore, the argument above shows that no element $\xs$ of $\tilde{\cC}_{\ell/2}$ can have $\xs^{\top}_{\ell/2}$ as non-trivial component of $L_Z(\xs)$ when we set $W=Z=0$. 

Let $\xs^{1}_{\ell/2}$ and $\xs^{1}_{\ell/2-1}$ denote the next highest $A_2$-graded elements of $\tilde{\cC}_{\ell/2}$ and $\tilde{\cC}_{\ell/2-1}$. Let $\xs^{j}_{\ell/2}$ and $\xs^{j}_{\ell-1}$ denote the subsequent elements of $\tilde{\cC}_{\ell}$ and $\tilde{\cC}_{\ell/2-1}$, ordered by their $A_2$-Alexander grading.  Note that $\tilde{\cC}_s$ has at most one generator in a given $(\gr_{\ws},\gr_{\zs})$-bigrading.

We consider three cases:
\begin{enumerate}
\item $\tilde{\cC}_{\ell/2-1}$ has a single generator $\xs_{\ell/2-1}^{\top}$.
\item $\tilde{\cC}_{\ell/2-1}$ has more than one generator and $L_Z(\xs_{\ell/2-1}^{\top})=0\in \tilde{\cC}_{\ell/2}$.
\item $\tilde{\cC}_{\ell/2-1}$ has more than one generator and $L_Z(\xs_{\ell/2-1}^{\top})\neq 0\in \tilde{\cC}_{\ell/2}$.
\end{enumerate}

We first address the case that $\tilde{\cC}_{\ell/2-1}$ contains a single generator. In this case, we can use the stabilization properties of the $H$-function, to see that $R_{\ell/2-1}<R_{\ell/2}$, as follows. Firstly, since $\tilde{\cC}_{\ell/2-1}$ contains a single generator, we see that
\begin{equation}
H_{L}(\ell/2-1,t)=\begin{cases}1 & \text{ if } t\ge R_{\ell/2-1}\\
R_{\ell/2-1}-t+1& \text{ if } t<R_{\ell/2-1}.
\end{cases}
\label{eq:H_L-initial-braid}
\end{equation} 
We know also that 
\begin{equation}
H_{L}(1-\ell/2,t)=H_U(1-\ell)=\ell-1
\label{eq:H(1-ell/2)}
\end{equation} if $t\gg 0$, where $U$ is the unknot (note that by convention we assume that $\ell\ge 0$ and furthermore $\ell\neq 0$ for a braid). Using the symmetry of the H-function (i.e., $H_L(s,t)=H_L(-s,-t)-s-t$), we conclude from Equation~\eqref{eq:H(1-ell/2)} that if $t\ll 0$, then
\[
H_L(\ell/2-1,t)=H_L(1-\ell/2,-t)-(\ell/2-1)-t=\ell/2-t.
\]
In light of Equation~\eqref{eq:H_L-initial-braid}, we conclude therefore that $R_{\ell/2-1}=\ell/2-1$. Finally, since $N_{L_P}=\ell/2$ (i.e. there is a $\gr_{\ws}$-grading preserving isomorphism $\cC_t\iso \cS$ to $t\ge \ell/2$), we know that $R_{\ell/2}= \ell/2+g_3(P)>\ell/2-1=R_{\ell/2-1}$,   completing the proof in this case.

We henceforth assume that $\tilde{\cC}_{\ell/2-1}$ has more than one generator.

In the case that $L_Z(\xs_{\ell/2-1}^{\top})=0$, we know that $H_*(\tilde{X}_{\ell/2})$ is spanned by $\xs_{\ell/2}^{\top}$ and $\xs_{\ell/2-1}^{\top}$. Therefore over the hat flavor, $L_Z$ must send $\xs_{\ell/2-1}^{j}$ to $\xs_{\ell/2}^j$ for $j>0$. Since $L_Z$ is also a chain map over the minus theory, by considering the $\xs_{\ell/2}^{\top}$ component of $[\d, L_Z](\xs_{\ell/2-1}^{1})$  we conclude that $L_Z(\xs_{\ell/2-1}^{\top})=\xs_{\ell/2}^{\top}\otimes W^i$ for some $i>0$. (By considering the $H$-function of $L_P$ further, it is possible to show that $i=1$, but this is not important here). In particular, this implies that 
\[
R_{\ell/2-1}=A_2(\xs_{\ell/2-1}^{\top})=A_2(\xs_{\ell/2}^{\top})-i=R_{\ell/2}-i<R_{\ell/2},
\]
concluding the proof in this case.

We consider the final case, where $L_{Z}(\xs_{\ell/2-1}^{\top})\neq 0$ in the hat theory. As described above, the image of this element cannot be $\xs^{\top}_{\ell/2}$, and therefore it must be mapped to some $\xs_{\ell/2}^{i}$ for some $i>0$. However this implies that
\[
R_{\ell/2-1}=A_2(\xs_{\ell/2-1}^{\top})=A_2(\xs_{\ell/2}^i)<A_2(\xs_{\ell/2}^{\top})=R_{\ell/2},
\]
completing the proof.
\end{proof}

\subsection{The bimodule  ${}_{\hat{\cK}} \cX^{\diamond}(L_P)_{\bF[Z]}$} \label{subsec: bimodule_cK_diamond} To compute the torsion order of the satellites, we first set  $W=0$ in the output of ${}_{\hat{\cK}} \cX(L_P)^{\bF[W,Z]}$, and pass to a type $AA$-bimodule. Let
 \[{}_{\hat{\cK}} \cX(L_P)_{\bF[Z]}:={}_{\hat{\cK}} \cX(L_P)^{\bF[Z]}\boxtimes {}_{\bF[Z]} \bF[Z]_{\bF[Z]}.\] 
 In \cite[Section 4.3]{CZZApp}, a homotopy equivalent bimodule  ${}_{\cK} \cX^{\diamond}(L_P)_{\bF[Z]}$ is constructed by replacing each $\cC_t$ and $\cS$ with its homology in $\bF[Z]$, as follows. 
 
As a type-$A$ module, the complex  $\cC_t^{\bF[Z]} \boxtimes {}_{\bF[Z]}\bF[Z]_{\bF[Z]}$ is quasi-isomorphic to 
 \begin{equation}\label{eq: C_t-min-def}
  \cC^{\diamond}_t := \bF[Z] \oplus \bF[Z] / Z^{\xi^t_1} \oplus \cdots  \oplus \bF[Z] / Z^{\xi^t_{k_t}}   
 \end{equation}
 for some integers $\xi_1^t,\dots, \xi_{k_t}^{t}$. We equip $\cC_t^\diamond$ 
with vanishing differential and the obvious action of $\bF[Z].$  Since $\cC_{t}^{\bF[W,Z]}$ is a staircase complex, $k_t$ is the number of steps (i.e. the number of arrows weighted by $Z$-powers), and $\xi_1^t,\dots, \xi^{t}_{k_t}$ are the lengths of the steps which are weighted by powers of $Z$.

Let $\xs^t_0$ denote the free generator and  $\xs^t_1, \dots, \xs^t_{k_t}$ denote the generators of the torsion summands, ordered as in Equation~\eqref{eq: C_t-min-def}. We assume that the generators are given the same ordering as they appear in the staircase $\cC_t$ itself, so that
\[
\gr_{\zs}(\xs^t_i)=\gr_{\zs}(\xs^t_0)+2\sum^i_{j=1}\xi^t_j, i=1,\dots,k.
\]

We similarly have 
\begin{equation}\label{eq: S-min-def}
  \cS^{\diamond} := \bF[Z] \oplus \bF[Z] / Z^{\xi'_1} \oplus \cdots  \oplus \bF[Z] / Z^{\xi'_{k}} ,
 \end{equation}
 with vanishing differential and the obvious action of $\bF[Z].$ 
Denote by $\xs'_0$  the free generator and $\xs'_i, i>0$, the  generators of the torsion summands, ordered analogously to the generators of $\cC_t$.

Later on, the quantity $\xi'_1$ will become important. We make the convention that $\xi'_1 = 0$ if $g_3(P)=0$.

The underlying right $\bF[Z]$-module of ${}_{\cK} \cX^{\diamond}(L_P)_{\bF[Z]}$ is obtained by replacing each $\cC_t$ with $\cC^{\diamond}_t$  and each $\cS$ with $\cS^{\diamond}$ in ${}_{\cK} \cX(L_P)_{\bF[Z]}$. 

The generators in ${}_{\cK} \cX^{\diamond}(L_P)_{\bF[Z]}$ inherit a $(\gr_{\ws},\gr_{\zs})$-bigrading from ${}_{\cK} \cX(L_P)_{\bF[Z]}$. According to the construction of ${}_{\cK} \cX(L_P)^{\bF[W,Z]}$ in \cite[Section 5.2]{CZZ},  together with \cite[Equation (4.3)]{CZZApp}, we have
\begin{equation} \label{eq: gradingxs}
    (\gr_{\ws},\gr_{\zs})(\xs^{t}_0) =  \begin{cases}
      (0, - 2R_{t} -2t )    &t\geq \frac{\ell}{2} \\
      (-\ell +2t, - 2R_{t} -2\ell )   &t< \frac{\ell}{2}
    \end{cases} \quad \text{and} \quad
    (\gr_{\ws},\gr_{\zs})(\xs'_0) =(0,-2g_3(P)),
\end{equation}
where $R_t$ is the quantity defined using the $H$-function in Definition \ref{def: R_t}. 
\begin{lem}[Lemma 4.10 in \cite{CZZApp}]
\label{lem:X-diamond-non-A_infty-bimodule}
    The only non-vanishing actions $m_{i|1|j}$ on ${}_{\cK} \cX^{\diamond}(L_P)_{\bF[Z]}$ are $m_{1|1|0}$ and $m_{0|1|1}$. Furthermore, the right action of $\bF[Z]$ sends $\cC_t^{\diamond}$ to itself, and coincides with the obvious $\bF[Z]$-module structure from its description in Equation~\eqref{eq: C_t-min-def}. Similarly, the right $\bF[Z]$-module structure sends $T^i\otimes \cS^{\diamond}$ to itself, and coincides with the natural analogous action.
\end{lem}

We introduce shorthand notation for several quantities frequently used throughout the paper. Their roles will become clear in Lemma \ref{lem: C_S_structure_maps}. 
\begin{define}\label{def: shorthand}    
 Define the quantities
\begin{equation}\label{eq: shorthand}
\begin{aligned}
    &\theta =R_{\frac{\ell}{2}}-g_3(P) - \frac{\ell}{2}, \quad 
    \omega = R_{\frac{\ell}{2}}-R_{\frac{\ell}{2}-1},    \\ &  \kappa = R_{\frac{\ell}{2}-1}-g_3(P) + \frac{\ell}{2}, \quad \text{and} \quad  \Xi =  \frac{\ell}{2}+g_3(P) -R_{\frac{\ell}{2}-1}-1.
\end{aligned}
\end{equation}
 \end{define}
 There is redundancy in the above definitions:
 it is easy to see that \[\ell - \kappa = \omega - \theta =  \Xi +1.\] 
 It follows from \cite[Lemma 4.9]{CZZApp} that   $\theta, \omega $ and $\kappa$ are all  nonnegative (because they are the exponents of $Z$-powers which appear in the structure map of a $DA$-bimodule). 
\begin{lem}\label{lem: braided_property}
    If $P$ is a braided  L-space satellite pattern, then \[\theta=0, \quad \ell \geq \omega>0 \quad \text{and} \quad   \Xi \geq 0.\]
\end{lem}
\begin{proof}
    By Proposition \ref{prop: braided_property}, $\omega= R_{\frac{\ell}{2}} - R_{\frac{\ell}{2}-1} >0 $  and $R_t$ is constant for $t \geq \frac{\ell}{2}.$
    By \cite[Lemma 4.9]{CZZApp},  $R_t$ is nondecreasing for $t \leq \frac{\ell}{2}$, $R_t = g_3(P) - \frac{\ell}{2}$ for $t \ll 0$ and $R_t = g_3(P) + \frac{\ell}{2}$ for $t \gg 0$. It follows that $R_{\frac{\ell}{2}}=g_3(P) + \frac{\ell}{2}$ and $R_{\frac{\ell}{2}} - R_{\frac{\ell}{2}-1} \leq R_{\frac{\ell}{2}} - g_3(P) + \frac{\ell}{2} = \ell.$ Hence $\theta = 0, \ell \geq \omega$ and $\Xi=\omega-
    \theta -1 = \omega -1 \geq 0.$
\end{proof}

 By Lemma~\ref{lem:X-diamond-non-A_infty-bimodule}, the bimodule ${}_{\hat{\cK}} \cX^{\diamond}(L_P)_{\bF[Z]}$ is an $AA$-bimodule with only $m_{1|1|0}$ and $m_{0|1|1}$ non-trivial, i.e., it is an ordinary bimodule over $(\hat{\cK},\bF[Z])$. In particular there is an $\bF[Z]$-torsion submodule. In \cite{CZZApp}, we studied the quotient of ${}_{\hat{\cK}} \cX^{\diamond}(L_P)_{\bF[Z]}$ by this torsion submodule, which we denoted ${}_{\hat{\cK}} \cX^{\top}(L_P)_{\bF[Z]}$.   The structure maps of ${}_{\hat{\cK}} \cX^{\top}(L_P)_{\bF[Z]}$  are described in \cite[Lemma 4.17]{CZZApp}.

 The following is an extension of \cite[Lemma 4.17]{CZZApp}:
\begin{lem} \label{lem: C_S_structure_maps} The structure maps of ${}_{\hat\cK} \cX^{\diamond}(L_P)_{\bF[Z]}$ restricted to $\xs^t_0$ when $t=\frac{\ell}{2}$ and $\frac{\ell}{2}-1$ are as follows:
    \begin{align*}
            L_Z(\xs^{\ell/2-1}_0)  &= \begin{cases}
            \xs^{\ell/2}_1 Z^{\xi_1^{\ell/2} + 1 - \omega } \quad &\text{if }  \omega > 1 \\
            0  \quad &\text{if }  \omega \leq 1
        \end{cases}   \\
         L_W(\xs^{\ell/2}_0)  &= \xs^{\ell/2-1}_0 Z^ \omega  \\
       L_{\tau}(\xs^{\ell/2}_0)  &= \xs'_0 Z^{\ell+\theta}  \hspace{2.5 em}  L_{\sigma}(\xs^{\ell/2}_0)  = \xs'_0 Z^ \theta     \\
       L_{\tau}(\xs^{\ell/2-1}_0) &= \xs'_0 Z^\kappa \hspace{3em}  L_{\sigma}(\xs^{\ell/2-1}_0) = \begin{cases}
            \xs'_1 Z^{\xi'_1 - \Xi} \quad &\text{if }   \Xi > 0  \\
            0  \quad &\text{if }   \Xi \leq 0
        \end{cases}.     
    \end{align*}
    Furthermore, if $\Xi >0$ then $g_3(P)>0$ and $\xi'_1 - \Xi \geq 0.$ Similarly, if $\omega > 1,$ then $\xi_1^{\ell/2} + 1 - \omega \geq 0.$
    In the above, recall that $L_Z$ denotes the action of $Z$, (e.g. $L_Z(\xs_0^{\ell/2-1})=m_{1|1|0}(Z,\xs_0^{\ell/2-1})$), and similarly for the other actions.
\end{lem}

\begin{proof}
    To prove the result, we consider the definition of the structure maps on the full bimodule ${}_{\cK} \cX(L_P)^{\bF[W,Z]}$, before passing to ${}_{\hat\cK} \cX^{\diamond}(L_P)_{\bF[Z]}$. We may view $\xs^{t}_i$ and $\xs'_i$ as generators of the staircases $\cC_{t}$ and $\cS$, respectively.
    
   We recall from \cite[Section 4.1]{CZZApp} that, as maps with domain $\cC_t$, the maps in ${}_{\cK}\cX(L_P)_{\bF[Z]}$ satisfy the following grading shift formulas:
\begin{equation}\label{eq: grading_shift}
\begin{aligned}
(\gr_{\ws},\gr_{\zs})(L_W) &= (-2,0),\\
(\gr_{\ws},\gr_{\zs})(L_Z) &= (0,-2),\\
(\gr_{\ws},\gr_{\zs})(L_\sigma) &= (0, 2t+\ell),\\
(\gr_{\ws},\gr_{\zs})(L_\tau) &= (-2t+\ell,0).
\end{aligned}
\end{equation}

By Equation~\eqref{eq: gradingxs}, we have
\[
\begin{split}
(\gr_{\ws},\gr_{\zs})(\xs^{\ell/2-1}_0) =& (-2,-2 R_{\ell/2 -1} - \ell ) \\ (\gr_{\ws},\gr_{\zs})(\xs^{\ell/2}_0) =& (0,-2 R_{\ell/2} - \ell ). \end{split}\]

Using the definition of the $R_t$ invariants  and the above grading formula for $L_Z$, we observe that 
\[
(\gr_{\ws},\gr_{\zs})(L_Z(\xs^{\ell/2-1}_0)) = (-2, -2 R_{\ell/2 -1} - \ell -2) = (\gr_{\ws},\gr_{\zs})(\xs^{\ell/2}_0) + (-2,2\omega-2).
\]
 If $\omega \leq 1$, we  can set $L_Z(\xs^{\ell/2-1}_0) = \xs^{\ell/2}_0\otimes WZ^{1- \omega} = 0$ on the full $DA$-bimodule ${}_{\cK} \cX(L_P)^{\bF[W,Z]}$. After setting $W=0$ on the right and forming the bimodule ${}_{\hat{\cK}} \cX(L_P)_{\bF[Z]}$, we obtain $L_Z(\xs^{\ell/2-1})=0$.

If $\omega > 1,$ 
the only possible generator in $\cC_{\ell/2}$ with the same bigrading as $L_Z(\xs^{\ell/2-1}_0)$ is of the form $\xs^{\ell/2}_1 Z^p$ for some $p$. 
By definition, in the module ${}_{\cK} \cX(L_P)^{\bF[W,Z]}$, the element $L_Z(\xs^{\ell/2-1}_0)$ is a non-zero generator of the homology of $\cC_{\ell/2}$. Therefore,
 we must have $L_Z(\xs^{\ell/2-1}_0) = \xs^{\ell/2}_1 Z^p$ for some $p\ge 0$.  Comparing  $\gr_{\zs}$ values  yields
\[-2 R_{\ell/2 -1} - \ell - 2 = -2 R_{\ell/2}  - \ell+ 2 \xi_1^{\ell/2} -2p\]
 and we solve that $ p= \xi_1^{\ell/2} + 1 - \omega .$

We now consider the claim about $L_\sigma (\xs^{\ell/2-1})$. Similarly to the above argument, we have 
\[
(\gr_{\ws},\gr_{\zs})(L_{\sigma}(\xs^{\ell/2-1}_0))=(-2,-2 R_{\ell/2 -1} + \ell -2)= (\gr_{\ws},\gr_{\zs})(\xs'_0) + (-2, 2\Xi).
\]
 If $\Xi \leq 0,$ we can take $L_{\sigma}(\xs^{\ell/2-1}_0) = \xs'_0 \otimes WZ^{-\Xi},$ which vanishes after passing from ${}_{\cK} \cX(L_P)^{\bF[W,Z]}$ to ${}_{\hat{\cK}} \cX^{\diamond}(L_P)_{\bF[Z]}$.

If $\Xi > 0,$ the only possible generator in $\cS$ with the same bigrading as $L_{\sigma}(\xs^{\ell/2-1}_0)$ is $\xs'_1 Z^q$ for some $q$. 
By definition, $L_{\sigma}(\xs^{\ell/2-1}_0)$ is nonzero, thus  we must have $L_{\sigma}(\xs^{\ell/2-1}_0) = \xs'_1 Z^q$ where $q \geq 0$ over ${}_{\cK} \cX(L_P)^{\bF[W,Z]}$. Therefore we have the same map over ${}_{\hat\cK} \cX^{\diamond}(L_P)_{\bF[Z]}$.
Comparing  $\gr_{\zs}$  yields
\[-2 R_{\ell/2 -1} + \ell -2 = -2 g_3(P) + 2 \xi'_1  -2q\]
 and we solve that $ q= \xi'_1+ R_{\ell/2-1}-\ell/2-g_3(P) +1 =\xi'_1 - \Xi.$ 

 The other equations are proven in \cite[Lemma 4.17]{CZZApp}.
 
 The claim that $\Xi>0$ implies that $g_3(P)>0$ and $\xi_1'\ge \Xi$ follows from the above argument: if $\Xi>0$ then $\cS$ must have a generator $\xs_1'$ in addition to $\xs_0'$, and therefore $g_3(P)>0$. Since $\xi_1'-\Xi$ appears as an exponent of $Z$, it must be nonnegative.
\end{proof}

 \subsection{The complex $\cCFK(P(K,\lambda))_{\bF[Z]}$} \label{subsec: complex_Pk}
 Following the approach in \cite[Section 5]{CZZApp}, we compute the knot Floer complex of $P(K,\lambda)$ via  the formula
 \begin{equation}\label{eq: CFK_PK_Z}
     \cCFK(P(K,\lambda))_{\bF[Z]} \simeq \cX_{\lambda}(K)^{\hat\cK} \boxtimes {}_{\hat\cK}\cH_{-}^{\hat\cK}\boxtimes {}_{\hat\cK} \cX^{\diamond}(L_P)_{\bF[Z]}.
 \end{equation}
 Denote by $\bX^{\diamond}(P,K,\lambda)_{\bF[Z]}$ the complex in the right hand side of \eqref{eq: CFK_PK_Z}.
Here $\cX_{\lambda}(K)^{\hat\cK}$ denotes the type-$D$ module associated to the knot $K$ with framing $\lambda\in \Z$, and ${}_{\hat\cK}\cH_-^{\hat\cK}$ denotes the bimodule corresponding to the negative Hopf link.
We now review the components appearing in this formula. 

The bimodule ${}_{\hat\cK}\cH_{-}^{\hat\cK}\boxtimes {}_{\hat\cK} \cX^{\diamond}(L_P)_{\bF[Z]}$ takes the form
\begin{equation}
{}_{\hat\cK} \cH_-^{\hat\cK} \boxtimes {}_{\hat\cK} \cX^{\diamond}(L_P)_{\bF[Z]}=
\begin{tikzcd}[column sep=1.6cm, row sep=1.6cm, labels=description] \scE^{\diamond}_{*,*}\ar[d, "f^K+f^{-K}"]\ar[r,"f^\mu+f^{-\mu}"]  & \scF^{\diamond}_{*,*}\ar[d,"f^K+f^{-K}"]\\
\scJ^{\diamond}_{*,*}\ar[r, "f^\mu+f^{-\mu}"]& \scM^{\diamond}_{*,*}
\label{eq:bimodule-X-P-n-cube}
\end{tikzcd}
\end{equation}
where $\scE^{\diamond}_{*,*}$, $\scF^{\diamond}_{*,*}$, $\scJ^{\diamond}_{*,*}$, $\scM^{\diamond}_{*,*}$ are built from the complexes and maps appearing in the bimodule ${}_{\cK} \cX^{\diamond}(L_P)_{\bF[Z]}$. The indices $(s,t)$ 
take values in
\[
\bH(P)=\left(\bZ+\frac{1}{2}\right)\times \left(\bZ+\frac{\lk(P,\mu)+1}{2}\right)
\]
viewed as the Alexander gradings associated to the two components of the Hopf link $H$ in $H\# L_P$. We have 
\[
\scE^{\diamond}_{s,t}=\begin{cases}
    \cC^{\diamond}_{t+\frac{1}{2}} &\text{ if }s>0\\
    \cC^{\diamond}_{t-\frac{1}{2}} &\text{ if } s<0\\
\end{cases}
\]
and $\scF^{\diamond}_{s,t} = \cS^{\diamond}$. We write $f^{\pm \mu}$ for the components of the structure map from $\scE_{*,*}^{\diamond}$ to $\scF_{*,*}^{\diamond}$. We illustrate these actions below:
\begin{equation} \label{eq: scE_scF_actions}
	 \begin{tikzcd}[labels=description, row sep=2cm] 
\scE^{\diamond}_{*,t} 
	 	\ar[d, "f^{\pm \mu}"]
	 	\\ 
	 	\scF^{\diamond}_{*,t}
	 \end{tikzcd}
	 \hspace{-.3cm}
	 =
	 \hspace{-.3cm}\begin{tikzcd}[column sep=.8cm, row sep=0cm]
	 	\cdots
	 	&[-1cm]\scE^{\diamond}_{-\frac{5}{2},t}
	 	\ar[d,"L_\sigma",  pos=.4]        
	 	& \scE^{\diamond}_{-\frac{3}{2},t}
	 	\ar[l ,"W|1"']
	 	\ar[d,"L_\sigma",  pos=.4]         \ar[dl,"L_\tau",  pos=.4]
        \ar[dl,"L_\tau",  pos=.4]
	 	&\scE^{\diamond}_{-\frac{1}{2},t}
	 	\ar[r, bend left,out=35,in=150, "Z|L_Z"]
	 	\ar[l, "W|1"']
	 	\ar[d,"L_\sigma",  pos=.4]         \ar[dl,"L_\tau",  pos=.4]
	 	&[.6cm] \scE^{\diamond}_{\frac{1}{2},t}
	 	\ar[r, "Z|1"]
	 	\ar[l, out=-160,in=-20,pos=.51,  "W|L_W"'{yshift=1pt}]
	 	\ar[d,"L_\sigma", pos=.4] \ar[dl,"L_\tau",  pos=.4]
	 	& \scE^{\diamond}_{\frac{3}{2},t}
	 	\ar[r, "Z|1"]
	 	\ar[d,"L_\sigma",  pos=.4]         \ar[dl,"L_\tau",  pos=.4]
	 	& \scE^{\diamond}_{\frac{5}{2},t}
	 	\ar[d,"L_\sigma", pos=.4] \ar[dl,"L_\tau",  pos=.4]
	 	&[-1cm] \cdots
	 	\\[1.2cm]
	 	\cdots
	 	&[.4cm]
	 	\scF^{\diamond}_{-\frac{5}{2},t}
	 	&\scF^{\diamond}_{-\frac{3}{2},t}
	 	\ar[l,  "W|1"]
	 	&\scF^{\diamond}_{-\frac{1}{2},t}
	 	\ar[r,  "Z|1"']
	 	\ar[l,  "W|1"]
	 	&\scF^{\diamond}_{\frac{1}{2},t}
	 	\ar[r,  "Z|1"']
	 	&\scF^{\diamond}_{\frac{3}{2},t}
	 	\ar[r,  "Z|1"']
	 	&\scF^{\diamond}_{\frac{5}{2},t}
	 	&\cdots 
	 \end{tikzcd}
\end{equation}
 We have $\scJ^{\diamond}_{s,t}=
    \cC^{\diamond}_{t+\frac{1}{2}}$ and $\scM^{\diamond}_{s,t}=
    \cS^{\diamond}$. The maps between $\scE^{\diamond}_{*,*}$ and $\scJ^{\diamond}_{*,*}$ are follows.
    \begin{equation} \label{eq: scE_scJ_actions}
    	 \begin{tikzcd}[labels=description,row sep=2.2cm] \scE^{\diamond}_{*,t}\ar[d, "f^{\pm K}"]
	 	\\  \scJ^{\diamond}_{*,t}
	 \end{tikzcd}
	 \hspace{-.2cm}
	 =
	 \hspace{-.3cm}
         \begin{tikzcd}[ {column sep=2cm,between origins}, row sep=.6cm]
	 	\cdots
	 	&[-1cm]\scE^{\diamond}_{-\frac{5}{2},t}
	 	\ar[d,"{ \tau|1 }"  pos=.4]        
	 	& \scE^{\diamond}_{-\frac{3}{2},t}
	 	\ar[l ,"W|1"']
	 	\ar[d,"{ \tau|1 }",  pos=.4]         
	 	&\scE^{\diamond}_{-\frac{1}{2},t}
	 	\ar[r, bend left,out=35,in=150, "Z|L_Z"]
	 	\ar[l, "W|1"']
	 	\ar[d,"{\begin{array}{l} \sigma|L_Z \\ \tau|1 \end{array}}"xshift=-4pt,  pos=.4]        
	 	&[.6cm] \scE^{\diamond}_{\frac{1}{2},t}
	 	\ar[r, "Z|1"]
	 	\ar[l, out=-160,in=-20,pos=.51,  "W|L_W"'{yshift=1pt}]
	 	\ar[d,"{\begin{array}{l} \sigma|1 \\ \tau|L_W \end{array}}"xshift=-4pt, pos=.4] 
	 	& \scE^{\diamond}_{\frac{3}{2},t}
	 	\ar[r, "Z|1"]
	 	\ar[d,"{ \sigma|1 }",  pos=.4]         
	 	& \scE^{\diamond}_{\frac{5}{2},t}
	 	\ar[d,"{ \sigma|1}", pos=.4] 
	 	&[-1cm] \cdots
	 	\\[1cm]
	 	\cdots
	 	& \scJ^{\diamond}_{-\frac{5}{2},t}
	 	\ar[r,bend left=15, "T|1"]
	 	&\scJ^{\diamond}_{-\frac{3}{2},t}
	 	\ar[r,bend left=15,  "T|1"]
	 	\ar[l,bend left=15, "T^{-1}|1"]
	 	&\scJ^{\diamond}_{-\frac{1}{2},t}
	 	\ar[r,bend left=15, "T|1"]
	 	\ar[l,bend left=15,  "T^{-1}|1"]
	 	&\scJ^{\diamond}_{\frac{1}{2},t}
	 	\ar[r,bend left=15, "T|1"]
	 	\ar[l,bend left=15,  "T^{-1}|1"]
	 	&\scJ^{\diamond}_{\frac{3}{2},t}
	 	\ar[r,bend left=15,  "T|1"]
	 	\ar[l,bend left=15,  "T^{-1}|1"]
	 	&\scJ^{\diamond}_{\frac{5}{2},t}
	 	\ar[l,bend left=15, "T^{-1}|1"]
	 	&\cdots 
	 \end{tikzcd}
    \end{equation}
The rest of the bimodule admits a similar description. We refer the reader to \cite{CZZApp}*{Figures~5.1 and 5.2} for further details. See also \cite[Section 9.3]{CZZ} the description over the full algebras $\cK$ and $\bF[W,Z]$. 

The type $D$-module $\cX_{\lambda}(K)^{\hat\cK}$ is constructed as follows. Fix a basis $B$ of $\cCFK(K)$ from the decomposition in Theorem~\ref{thm:standard_decomp}. We set $\cX_{\lambda}(K)\cdot \ve{I}_0$ to be the $\bF$-span of $B$ and $\cX_{\lambda}(K)\cdot \ve{I}_1$ to be the $\bF$-span of a single generator $p$.  The structure map of $\cX_{\lambda}(K)^{\hat\cK}$ admits a decomposition \[\delta^1 = \delta^1_{\hat{\cR}} + \delta^1_{\sigma,\tau}\]
where $\delta^1_{\hat\cR}$ denotes the terms in $\delta^1$ with algebraic outputs in $\hat\cR$ and $\delta^1_{\sigma,\tau}$ denotes the terms in $\delta^1$ with algebraic outputs a multiple of $\sigma$ or $\tau$. Given $a\in B,$ suppose $\partial a = \sum_{b_k \in B} r_k b_k$ where $r_k \in \hat\cR$. We declare $ \delta^1_{\hat{\cR}} (a) = \sum_{b_k \in B}  b_k \otimes r_k$. We declare $\delta^1_{\sigma,\tau}$ to be zero except for  $a_0$ and $a_m$, the ending and starting generators of the standard complex, respectively. We set 
\begin{equation}\label{eq: CFK_connecting_map}
\delta^1_{\sigma,\tau}(a_0) =  p \otimes T^{\lambda - 2\tau(K)}\tau \qquad \delta^1_{\sigma,\tau}(a_m) =  p \otimes \sigma.
\end{equation}
For a fixed $t$, we write
\[ \cCFK(K)^{\hat\cR}\boxtimes {}_{\hat\cR} \scE^{\diamond}_{*,t} =\bigoplus_{s\in \Z+\frac{1}{2}} E^{\diamond}_{s,t}\]
where the Alexander grading is additive in the first component in the tensor product and $E^{\diamond}_{s,t}$ denotes the Alexander grading $s$ summand. 
We define $F^{\diamond}_{s,t}$, $J^{\diamond}_{s,t}$, and $M^{\diamond}_{s,t}$ analogously. 

The complex $\bX^{\diamond}(P,K,n)_{\bF[Z]}$ can be described as a 2-dimensional hypercube of the form
\begin{equation}
\bX^{\diamond}(P,K,n)_{\bF[Z]}=\begin{tikzcd}[column sep=1.1cm, row sep=1.1cm]
 \bE^{\diamond} \ar[r, "\Phi^{\mu}+\Phi^{-\mu}"] \ar[d, "\Phi^{K}+\Phi^{-K}",swap] & \bF^{\diamond}\ar[d, "\Phi^{K}+\Phi^{-K}"]\\
\bJ^{\diamond} \ar[r, swap,"\Phi^{\mu}+\Phi^{-\mu}"]& \bM^{\diamond}
\end{tikzcd}
\end{equation}
where
\[
\bE^{\diamond}=\bigoplus_{(s,t)\in \bH(P)} E^{\diamond}_{s,t}
\]
and $\bF^{\diamond}, \bJ^{\diamond}$ and $\bM^{\diamond}$ are defined analogously. The maps  between the  hypercube vertices are defined  by $\Phi^{\pm \mu}= \bI_{\cX_{\lambda}(K)} \boxtimes f^{\pm \mu}$ and $\Phi^{\pm K}= \bI_{\cX_{\lambda}(K)} \boxtimes f^{\pm K}$ respectively. We call $\Phi^{\pm \mu}$ and $\Phi^{\pm K}$ the \emph{length-$1$ differentials} and call the internal differentials in each complex $\bE^{\diamond}, \bF^{\diamond}, \bJ^{\diamond}$ and $\bM^{\diamond}$ the \emph{length-$0$ differentials}. The model has no diagonal maps from $\bE^{\diamond}$ to $\bM^{\diamond}$.

\subsection{A diagrammatic description of $\bX(P,K,\lambda)$} \label{subsec: diagrammatic}
We give a diagrammatic description of $\bX(P,K,\lambda)$. This is inspired by the filtrations defined in \cite{HeddenLevineSurgery}.

Fix a $t\in \mathbb{Z}+\dfrac{\ell}{2}$. Recall from \cite[Section 9]{CZZ}, we have the decomposition

\[
\cCFK(K)^{\bF[W,Z]}\boxtimes {}_{\bF[W,Z]} \scE_t^{\bF[W,Z]}=\bigoplus_{s\in \Z+1/2} E_{s,t},
\]
where $E_{s,t}$ is the subcomplex of the above tensor product where the first coordinate of the Alexander grading (corresponding to $K$) takes value $s$. The $DA$-bimodule ${}_{\bF[W,Z]} \scE_t^{\bF[W,Z]}$
is specified by a diagram of the following form:
\[
\scE_{*,t}=\begin{tikzcd}[labels=description, column sep=1cm, row sep=0cm]
	\cdots
	&[-1cm]
	\scE_{-\frac{5}{2},t}
	\ar[r, bend left,out=45,in=135, "Z|U"]
	& \scE_{-\frac{3}{2},t}
	\ar[r, bend left,out=45,in=135, "Z|U"]
	\ar[l, bend left,out=45,in=135, "W|1"]
	&\scE_{-\frac{1}{2},t}
	\ar[r, bend left,out=45,in=135, "Z|L_Z"]
	\ar[l, bend left,out=45,in=135, "W|1"]
	\ar[loop above, looseness=15, "{(W,Z)|h_{W,Z}}"]
	&[.6cm]\scE_{\frac{1}{2},t}
	\ar[r, bend left,out=45,in=135, "Z|1"]
	\ar[l, bend left,out=45,in=135, "W|L_W"]
	\ar[loop above, looseness=15, "{(Z,W)|h_{Z,W}}"]
	& \scE_{\frac{3}{2},t}
	\ar[r, bend left,out=45,in=135, "Z|1"]
	\ar[l, bend left,out=45,in=135, "W|U"]
	& \scE_{\frac{5}{2},t}
	\ar[l, bend left,out=45,in=135, "W|U"] 
	&[-1cm] \cdots
\end{tikzcd},
\]
and $\scE_{s,t}$ denotes the type-$D$ module over $R$,
\begin{equation}
	\scE_{s,t}=\begin{cases} \cC_{t-\frac{1}{2}} & \text{ if } s<0\\
		\cC_{t+\frac{1}{2}} & \text{ if } s>0
	\end{cases}
\end{equation}
We view $\scE_{s,t}$ as having Alexander grading in $(s,t,*)$.

A diagrammatic way to perform this box tensor product is as follows.
Consider a complex $\CFKi(K,s+\frac{1}{2})$ over $\bF[U]$. (Here, $s\in 1/2+\Z$, as in the preceding paragraphs). Its underlying module consists of elements of the form $[a,i,j]$, where $a$ is a generator of $\cCFK(K)^{\bF[W,Z]}$, $i,j\in \bZ$ satisfying
\[ A(a) + i-j   = s+\frac{1}{2}.\] 
The $U$-action is
\[ U[a,i,j] = [a,i-1,j-1]. \]
The differential is given as follows: for each term $b \otimes W^p Z^q$ in $\delta^1(a)$ in $\cCFK(K)^{\bF[W,Z]}$, we declare $[b,i-p,j-q]$ to be a term in $\partial[a,i,j]$.

To get $E_{s,t}$, place  $\CFKi(K,s+\frac{1}{2})$ in a $\mathbb{Z}\times \mathbb{Z}$ grid, with index $(i,j)$.
Color the upper half $\{i\ge j\}$ of the grid by gray, and the bottom half $\{i< j\}$ by white. 

First, identify the elements
$[a,i,j]$ in the same  $U$-translation orbit, and replace it by $a|\scE_{i-j-\frac{1}{2},t}.$ 
As a result, an orbit 
in the gray region is replaced  by the staircase $\cC_{t-\frac{1}{2}}$, and an orbit in the white region by the staircase $\cC_{t+\frac{1}{2}}$. 

Then replace the arrows by maps between staircases as follows: a horizontal arrow in the gray region is replaced by a map with coefficient $1$ (i.e. it sends each generator to its corresponding generator with coefficient $1$), and a vertical arrow of length $n$ in the gray region is replaced by a map with coefficient $U^n$. Symmetrically, a horizontal arrow of length $n$ in the white region is replaced by a map with coefficient $U^n$, and a vertical one is replaced by a map with coefficient $1$. If an arrow passes from the gray region to the white one, replace it by $L_Z$. Conversely, if it passes from the white region to the gray region, replace it by $L_W$. If there are two adjacent arrows crossing from one region to the other and then back, add a diagonal arrow with weight $h_{W,Z}$ or $h_{Z,W}$ depending on the order of the two arrows. 

See Figure \ref{fig:E for RHT input} and \ref{fig:E for box input} for an example when the input is the right hand trefoil and the length-$1$ box complex respectively. 

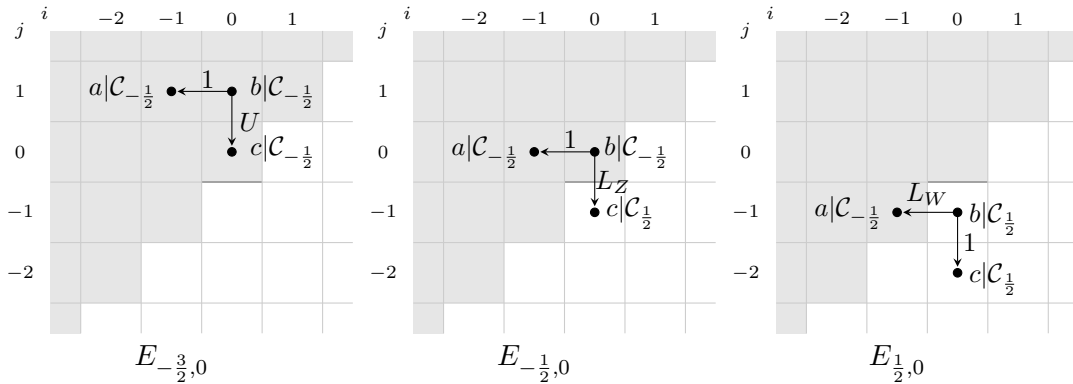
\begin{figure}[h!]
	\begin{tikzpicture}[scale=0.8]
		\begin{scope}[on background layer]
			\foreach \j in {-6,0,6}
			{\fill[gray!20] (\j+-2,1.5) rectangle (\j+1.5,0.5);
				\fill[gray!20] (\j+-2,2.5) rectangle (\j+2.5,1.5);
				\fill[gray!20] (\j+-2,3) rectangle (\j+3,2.5);
				\fill[gray!20] (\j+-2,0.5) rectangle (\j+0.5,-0.5);
				\fill[gray!20] (\j+-2,-0.5) rectangle (\j+-0.5,-1.5);
				\fill[gray!20] (\j+-2,-1.5) rectangle (\j+-1.5,-2);
				\draw (\j+0.5,0.5) to (\j+1.5,0.5);
			} 
			\foreach \j in {-6,0,6}
			\foreach \i in {-1,...,3}
			{\draw[thin, black!20!white]  (\j+\i-0.5, 3) -- (\j+\i-0.5, -2);
				\draw[thin, black!20!white]  (\j+3, \i-0.5) -- (\j-2, \i-0.5); }
		\end{scope}      
	
	    \foreach \j in {-6}
	    {\filldraw (\j+1, 2) circle (2pt) node[] {};
	    	\filldraw (\j, 2) circle (2pt) node[] {};
	    	\filldraw (\j+1, 1) circle (2pt) node[] {};
	    	\draw [-stealth] (\j+1, 2) -- (\j+0.1, 2);
	    	\draw [-stealth] (\j+1, 2) -- (\j+1, 1.1);
	    }
		\foreach \j in {0}
		{\filldraw (\j+1, 1) circle (2pt) node[] {};
			\filldraw (\j, 1) circle (2pt) node[] {};
			\filldraw (\j+1, 0) circle (2pt) node[] {};
			\draw [-stealth] (\j+1, 1) -- (\j+0.1, 1);
			\draw [-stealth] (\j+1, 1) -- (\j+1, 0.1);
		}
		\foreach \j in {6}
		{\filldraw (\j+1, 0) circle (2pt) node[] {};
			\filldraw (\j, 0) circle (2pt) node[] {};
			\filldraw (\j+1, -1) circle (2pt) node[] {};
			\draw [-stealth] (\j+1, 0) -- (\j+0.1, 0);
			\draw [-stealth] (\j+1, 0) -- (\j+1, -0.9);
		}
		
		\node   at (-6.8,2) {{\small $a|\cC_{-\frac{1}{2}}$}};
		\node   at (-4.15,2) {{\small $b|\cC_{-\frac{1}{2}}$}};
		\node   at (-4.15,1) {{\small $c|\cC_{-\frac{1}{2}}$}};
		\node   at (-4.7,1.5) {{\small $U$}};
		\node   at (-5.4,2.2) {{\small $1$}};
		
		\node   at (-0.8,1) {{\small $a|\cC_{-\frac{1}{2}}$}};
		\node   at (1.7,1) {{\small $b|\cC_{-\frac{1}{2}}$}};
		\node   at (1.6,0) {{\small $c|\cC_{\frac{1}{2}}$}};
		\node   at (0.6,1.2) {{\small $1$}};
		\node   at (1.3,0.5) {{\small $L_Z$}};
		
		\node   at (5.2,0) {{\small $a|\cC_{-\frac{1}{2}}$}};
		\node   at (7.6,-0.1) {{\small $b|\cC_{\frac{1}{2}}$}};
		\node   at (7.6,-1.1) {{\small $c|\cC_{\frac{1}{2}}$}};
		\node   at (6.5,0.3) {{\small $L_W$}};
		\node   at (7.2,-0.5) {{\small $1$}};
		
		\node   at (-6,-2.5) {{$E_{-\frac{3}{2},0}$}};
		\node   at (0,-2.5) {{$E_{-\frac{1}{2},0}$}};
		\node   at (6,-2.5) {{$E_{\frac{1}{2},0}$}};
		\foreach \j in {-6,0,6}
		{
		\node at (\j-1,3.2){{\tiny$-2$}};
		\node at (\j,3.2){{\tiny$-1$}};
		\node at (\j+1,3.2){{\tiny$0$}};
		\node at (\j+2,3.2){{\tiny$1$}};
		\node at (\j-2.1,3.3){{\tiny$i$}};
		\node at (\j-2.5,3){{\tiny$j$}};
		\node at (\j-2.5,2){{\tiny$1$}};
		\node at (\j-2.5,1){{\tiny$0$}};
		\node at (\j-2.5,0){{\tiny$-1$}};
		\node at (\j-2.5,-1){{\tiny$-2$}};
		} 
	\end{tikzpicture}      
\caption{Example of $E_{s,t}$, when $K=T_{2,3}$, and $t=0$. We choose a representative for each  $U$-translation orbit of $[a,i,j]$.}
\label{fig:E for RHT input}
\end{figure}
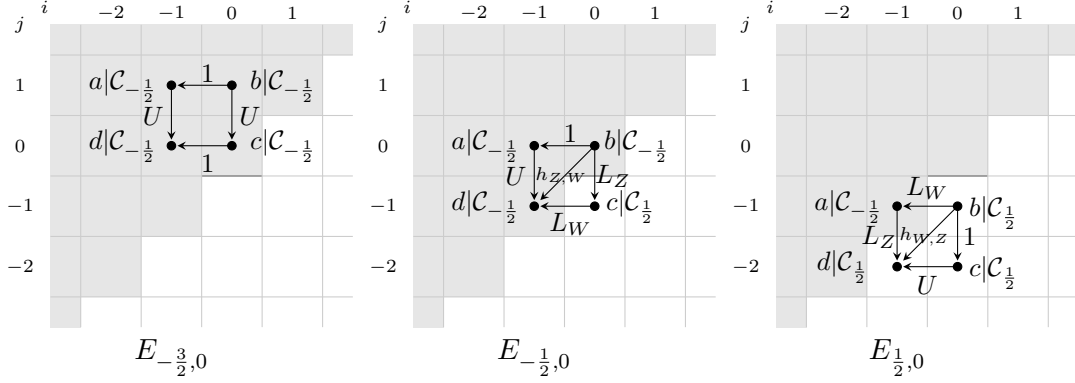
\begin{figure}[h!]
	\begin{tikzpicture}[scale=0.8]
		\begin{scope}[on background layer]
			\foreach \j in {-6,0,6}
			{\fill[gray!20] (\j+-2,1.5) rectangle (\j+1.5,0.5);
				\fill[gray!20] (\j+-2,2.5) rectangle (\j+2.5,1.5);
				\fill[gray!20] (\j+-2,3) rectangle (\j+3,2.5);
				\fill[gray!20] (\j+-2,0.5) rectangle (\j+0.5,-0.5);
				\fill[gray!20] (\j+-2,-0.5) rectangle (\j+-0.5,-1.5);
				\fill[gray!20] (\j+-2,-1.5) rectangle (\j+-1.5,-2);
				\draw (\j+0.5,0.5) to (\j+1.5,0.5);
			} 
			\foreach \j in {-6,0,6}
			\foreach \i in {-1,...,3}
			{\draw[thin, black!20!white]  (\j+\i-0.5, 3) -- (\j+\i-0.5, -2);
				\draw[thin, black!20!white]  (\j+3, \i-0.5) -- (\j-2, \i-0.5); }
		\end{scope}      
		
		\foreach \j in {-6}
		{\filldraw (\j+1, 2) circle (2pt) node[] {};
			\filldraw (\j, 2) circle (2pt) node[] {};
			\filldraw (\j+1, 1) circle (2pt) node[] {};
			\filldraw (\j, 1) circle (2pt) node[] {};
			\draw [-stealth] (\j+1, 2) -- (\j+0.1, 2);
			\draw [-stealth] (\j, 2) -- (\j, 1.1);
			\draw [-stealth] (\j+1, 2) -- (\j+1, 1.1);
			\draw [-stealth] (\j+1, 1) -- (\j+0.1, 1);
		}
		\foreach \j in {0}
		{\filldraw (\j+1, 1) circle (2pt) node[] {};
			\filldraw (\j, 1) circle (2pt) node[] {};
			\filldraw (\j+1, 0) circle (2pt) node[] {};
			\filldraw (\j, 0) circle (2pt) node[] {};
			\draw [-stealth] (\j+1, 1) -- (\j+0.1, 1);
			\draw [-stealth] (\j+1, 1) -- (\j+1, 0.1);
			\draw [-stealth] (\j+1, 0) -- (\j+0.1, 0);
			\draw [-stealth] (\j, 1) -- (\j, 0.1);
			\draw [-stealth] (\j+1, 1) -- (\j+0.1, 0.1);
		}
		\foreach \j in {6}
		{\filldraw (\j+1, 0) circle (2pt) node[] {};
			\filldraw (\j, 0) circle (2pt) node[] {};
			\filldraw (\j+1, -1) circle (2pt) node[] {};
			\filldraw (\j, -1) circle (2pt) node[] {};
			\draw [-stealth] (\j+1, 0) -- (\j+0.1, 0);
			\draw [-stealth] (\j+1, 0) -- (\j+1, -0.9);
			\draw [-stealth] (\j+1, -1) -- (\j+0.1, -1);
			\draw [-stealth] (\j, 0) -- (\j,-0.9);
				\draw [-stealth] (\j+1, 0) -- (\j+0.1, -0.9);
		}
		
		\node   at (-6.8,2) {{\small $a|\cC_{-\frac{1}{2}}$}};
		\node   at (-4.15,2) {{\small $b|\cC_{-\frac{1}{2}}$}};
		\node   at (-4.15,1) {{\small $c|\cC_{-\frac{1}{2}}$}};
		\node   at (-6.8,1) {{\small $d|\cC_{-\frac{1}{2}}$}};
		\node   at (-4.7,1.5) {{\small $U$}};
		\node   at (-6.3,1.5) {{\small $U$}};
		\node   at (-5.4,2.2) {{\small $1$}};
		\node   at (-5.4,0.7) {{\small $1$}};
		
		\node   at (-0.8,1) {{\small $a|\cC_{-\frac{1}{2}}$}};
		\node   at (1.7,1) {{\small $b|\cC_{-\frac{1}{2}}$}};
		\node   at (1.6,0) {{\small $c|\cC_{\frac{1}{2}}$}};
		\node   at (-0.8,0) {{\small $d|\cC_{-\frac{1}{2}}$}};
		\node   at (0.6,1.2) {{\small $1$}};
		\node   at (1.3,0.5) {{\small $L_Z$}};
		\node   at (-0.3,0.5) {{\small $U$}};
		\node   at (0.6,-0.3) {{\small $L_W$}};
		\node   at (0.45,0.5) {{\tiny $h_{Z,W}$}};

		\node   at (5.2,0) {{\small $a|\cC_{-\frac{1}{2}}$}};
		\node   at (5.1,-1) {{\small $d|\cC_{\frac{1}{2}}$}};
		\node   at (7.6,-0.1) {{\small $b|\cC_{\frac{1}{2}}$}};
		\node   at (7.6,-1.1) {{\small $c|\cC_{\frac{1}{2}}$}};
		\node   at (6.5,0.3) {{\small $L_W$}};
		\node   at (5.7,-0.5) {{\small $L_Z$}};
		\node   at (7.2,-0.5) {{\small $1$}};
		\node   at (6.5,-1.3) {{\small $U$}};
		\node   at (6.45,-0.5) {{\tiny $h_{W,Z}$}};
		
		\node   at (-6,-2.5) {{$E_{-\frac{3}{2},0}$}};
		\node   at (0,-2.5) {{$E_{-\frac{1}{2},0}$}};
		\node   at (6,-2.5) {{$E_{\frac{1}{2},0}$}};
		\foreach \j in {-6,0,6}
		{
			\node at (\j-1,3.2){{\tiny$-2$}};
			\node at (\j,3.2){{\tiny$-1$}};
			\node at (\j+1,3.2){{\tiny$0$}};
			\node at (\j+2,3.2){{\tiny$1$}};
			\node at (\j-2.1,3.3){{\tiny$i$}};
			\node at (\j-2.5,3){{\tiny$j$}};
			\node at (\j-2.5,2){{\tiny$1$}};
			\node at (\j-2.5,1){{\tiny$0$}};
			\node at (\j-2.5,0){{\tiny$-1$}};
			\node at (\j-2.5,-1){{\tiny$-2$}};
		} 
	\end{tikzpicture}      
	\caption{Example of $E_{s,t}$, when the input is the length-$1$ box complex and $t=0$.}
	\label{fig:E for box input}
\end{figure}
Similarly, for $F_{s,t}$,  we have the decomposition  \[\cCFK(K)^{\bF[W,Z]}\boxtimes {}_{\bF[W,Z]}\scF_{*,t}^{\bF[W,Z]} =\bigoplus_{s\in \Z+1/2} F_{s,t},\] where ${}_{\bF[W,Z]}\scF_{*,t}^{\bF[W,Z]} $ is the $DA$ bimodule given by the following diagram:
\[
\scF_{*,t}=\begin{tikzcd}[labels=description, column sep=.8cm, row sep=0cm]
	\cdots
	&[-.8cm] \scF_{-\frac{5}{2}, t}
	\ar[r,bend left,out=45,in=135, "Z|U"]
	\ar[loop below,looseness=20, "U|U"]
	& \scF_{-\frac{3}{2},t}
	\ar[r, bend left,out=45,in=135, "Z|U"]
	\ar[l, bend left,out=45,in=135, "W|1"]
	\ar[loop below,looseness=20, "U|U"]
	&\scF_{-\frac{1}{2},t}
	\ar[r, bend left,out=45,in=135, "Z|1"]
	\ar[l, bend left,out=45,in=135, "W|1"]
	\ar[loop below,looseness=20, "U|U"]
	&\scF_{\frac{1}{2},t}
	\ar[r, bend left,out=45,in=135, "Z|1"]
	\ar[l, bend left,out=45,in=135, "W|U"]
	\ar[loop below,looseness=20, "U|U"]
	&\scF_{\frac{3}{2},t}
	\ar[l, bend left,out=45,in=135, "W|U"]
	\ar[loop below,looseness=20, "U|U"]
	&[-.8cm]\cdots 
\end{tikzcd}
\]
In the above, we set
\[
\scF_{s,t}=\cS
\]
for all $s,t$.

A diagrammatic way to compute $F_{s,t}$ is as follows. Divide a $\mathbb{Z}\times \mathbb{Z}$ grid into strips of the form 
\[\max\left\{i,j\right\}=k, \text{ for }k\in \mathbb{Z},\]
 and draw a wall between each adjacent regions.

 To get $F_{s,t}$, place $\CFKi(K,s+\frac{1}{2})$ in the $\mathbb{Z}\times \mathbb{Z}$ grid.
 Similar as before, identify the elements
$[a,i,j]$ in the same  $U$-translation orbit, and replace it by $a|\scF_{i-j-\frac{1}{2},t}=\cS.$

Then, replace an arrow by a map with coefficient $U^n$ if it crosses the walls $n$ times. In particular, if it stays in a region $\max\left\{i,j\right\}=k$ for a fixed $k$, then replace it by a map with coefficient $1$.

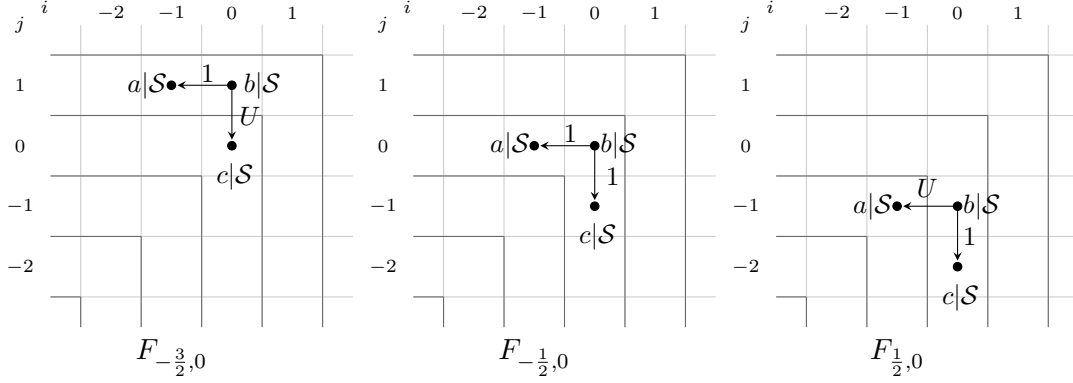
\begin{figure}[h!]
	\begin{tikzpicture}[scale=0.8]
		\begin{scope}[on background layer]
			
			\foreach \j in {-6,0,6}
			\foreach \i in {-1,...,3}
			{\draw[thin, black!20!white]  (\j+\i-0.5, 3) -- (\j+\i-0.5, -2);
				\draw[thin, black!20!white]  (\j+3, \i-0.5) -- (\j-2, \i-0.5); }
			\foreach \j in {-6,0,6}
		{\draw[thin, black!60!white] (\j-2, 2.5) -- (\j+2.5, 2.5);
			\draw[thin, black!60!white] (\j+2.5, 2.5) -- (\j+2.5, -2);
			\draw[thin, black!60!white] (\j+1.5, 1.5) -- (\j+1.5, -2);
			\draw[thin, black!60!white]  (\j+0.5, 0.5) -- (\j+0.5, -2);
			\draw[thin, black!60!white]  (\j+1.5, 1.5) -- (\j-2, 1.5);
			\draw[thin, black!60!white] (\j+0.5, 0.5) -- (\j-2, 0.5);
			\draw[thin, black!60!white] (\j-2, -0.5) -- (\j-0.5, -0.5);
			\draw[thin, black!60!white] (\j-0.5, -0.5) -- (\j-0.5, -2);
			\draw[thin, black!60!white] (\j-2, -1.5) -- (\j-1.5, -1.5);
			\draw[thin, black!60!white] (\j-1.5, -1.5) -- (\j-1.5, -2);}
		\end{scope}      
		
		\foreach \j in {-6}
		{\filldraw (\j+1, 2) circle (2pt) node[] {};
			\filldraw (\j, 2) circle (2pt) node[] {};
			\filldraw (\j+1, 1) circle (2pt) node[] {};
			\draw [-stealth] (\j+1, 2) -- (\j+0.1, 2);
			\draw [-stealth] (\j+1, 2) -- (\j+1, 1.1);
		}
		\foreach \j in {0}
		{\filldraw (\j+1, 1) circle (2pt) node[] {};
			\filldraw (\j, 1) circle (2pt) node[] {};
			\filldraw (\j+1, 0) circle (2pt) node[] {};
			\draw [-stealth] (\j+1, 1) -- (\j+0.1, 1);
			\draw [-stealth] (\j+1, 1) -- (\j+1, 0.1);
		}
		\foreach \j in {6}
		{\filldraw (\j+1, 0) circle (2pt) node[] {};
			\filldraw (\j, 0) circle (2pt) node[] {};
			\filldraw (\j+1, -1) circle (2pt) node[] {};
			\draw [-stealth] (\j+1, 0) -- (\j+0.1, 0);
			\draw [-stealth] (\j+1, 0) -- (\j+1, -0.9);
		}
		
		\node   at (-6.4,2) {{\small $a|\cS$}};
		\node   at (-4.5,2) {{\small $b|\cS$}};
		\node   at (-4.95,0.5) {{\small $c|\cS$}};
		\node   at (-4.7,1.5) {{\small $U$}};
		\node   at (-5.4,2.2) {{\small $1$}};
		
		\node   at (-0.4,1) {{\small $a|\cS$}};
		\node   at (1.4,1) {{\small $b|\cS$}};
		\node   at (1.05,-0.5) {{\small $c|\cS$}};
		\node   at (0.6,1.2) {{\small $1$}};
		\node   at (1.3,0.5) {{\small $1$}};
		
		\node   at (5.6,0) {{\small $a|\cS$}};
		\node   at (7.4,0) {{\small $b|\cS$}};
		\node   at (7.05,-1.5) {{\small $c|\cS$}};
		\node   at (6.5,0.3) {{\small $U$}};
		\node   at (7.2,-0.5) {{\small $1$}};
		
		\node   at (-6,-2.5) {{$F_{-\frac{3}{2},0}$}};
		\node   at (0,-2.5) {{$F_{-\frac{1}{2},0}$}};
		\node   at (6,-2.5) {{$F_{\frac{1}{2},0}$}};
		\foreach \j in {-6,0,6}
		{
			\node at (\j-1,3.2){{\tiny$-2$}};
			\node at (\j,3.2){{\tiny$-1$}};
			\node at (\j+1,3.2){{\tiny$0$}};
			\node at (\j+2,3.2){{\tiny$1$}};
			\node at (\j-2.1,3.3){{\tiny$i$}};
			\node at (\j-2.5,3){{\tiny$j$}};
			\node at (\j-2.5,2){{\tiny$1$}};
			\node at (\j-2.5,1){{\tiny$0$}};
			\node at (\j-2.5,0){{\tiny$-1$}};
			\node at (\j-2.5,-1){{\tiny$-2$}};
		} 
	\end{tikzpicture}      
	\caption{Example of $F_{s,t}$, when $K=T_{2,3}$.}
\end{figure}

\section{Proof of the main bound}\label{sec: main_bound}
In this section, we prove our main result:
\TorOrd*

Since $K$ is a nontrivial knot, we have
  $\Ord(K)=n>0$. It follows that there is a $Z^n$-arrow connecting two generators $a_1$ and $a_0$ in either the standard complex or one of the local system complexes in the decomposition in Theorem~\ref{thm:standard_decomp}. We divide the proof into the following two cases:
  \begin{enumerate}[label=($A$-\arabic*), ref=$A$-\arabic*] 
    \item\label{case:arrow-1} The $Z^n$-weighted arrow connecting $a_0$ and $a_1$ is not the last arrow in the standard complex for $\cCFK(K)$. We call such an arrow a \emph{non-ending arrow}.
  \item\label{case:arrow-2} The $Z^n$-weighted arrow connecting $a_0$ and $a_1$ is the last arrow in the standard complex for $\cCFK(K)$. We call such an arrow an \emph{ending arrow}. 
  \end{enumerate}
In both cases, we study the form that the complex $\cCFK(P(K,\lambda))_{\bF[Z]}$ takes with respect to the tensor product description in Equation~\eqref{eq: CFK_PK_Z}. 

 In the first case, there are adjacent generators in the standard complex or local system which are connected to $a_1$ and $a_0$ by arrows weighted by powers of $W$. In the second case, one of these generators will be connected by an arrow with a power of $W$, while the other generator will admit a differential in the module $\cX_{\lambda}(K)^{\hat{\cK}}$ which is weighted by a multiple of $\sigma$ or $\tau$.

 \subsection{Non-ending arrows} \label{subsec: case_non_ending}
 
 In this section, we consider case~\eqref{case:arrow-1}, where neither generator $a_0$ or $a_1$ is the final generator in the standard complex of $\cCFK(K)$.

We now divide non-ending arrows into four special classes, depending on the configurations of adjacent arrows. We will consider each class separately. In each case, our strategy will be to partially compute $\cCFK(P(K,\lambda))_{\bF[Z]}$ following the framework outlined in Section~\ref{subsec: complex_Pk}, and then apply Lemma~\ref{lem: torsion_order_criterion} to obtain a bound on the torsion order.

 \begin{define}
 \label{def:non-ending-types}
Suppose that $\cCFK(K)^{\hat{\cR}}$ is a knot Floer complex, which is decomposed as a direct sum of a standard complex $\cC$, and local system complexes $\cL_1,\dots, \cL_n$. Suppose that $Z^n$ is a non-ending arrow of $\cCFK(K)^{\hat{\cR}}$. There are four consecutive generators $a'_0, a_{0} ,a_1 , a'_1$ (the indices  do not correspond to the position in the standard complex) and a pair of nonzero integers  $(\eta_0,\eta_1)$ such that there is an arrow weighted by $Z^n$ from $a_1$ to $a_0$,  an arrow weighted by $W^{|\eta_i|}$ from $a'_{i} $ to $ a_{i}$ if $\eta_i>0 $ or from $a_{i} $ to $ a'_{i}$ if $\eta_i<0$ 
and there are no other arrows entering or leaving $a_i, i=0,1.$  In particular, the $\bF[Z]$-span of $\{ a_{0}, a_1\}$ forms a direct summand in $\cCFK(K)/(W)$ and the $\bF[W]$-span of $\{ a_{i}, a'_i\}$ for $i=0,1$ forms a direct summand in $\cCFK(K)/(Z)$.  
We call such an arrow a \emph{non-ending $Z^n$-arrow of type $(\sigma_0,\sigma_1)$}, where $\sigma_i = \operatorname{sgn}(\eta_i)$ for $i=0,1$. See Figure~\ref{fig: non_ending_arrow_cases} for an illustration.
\end{define}
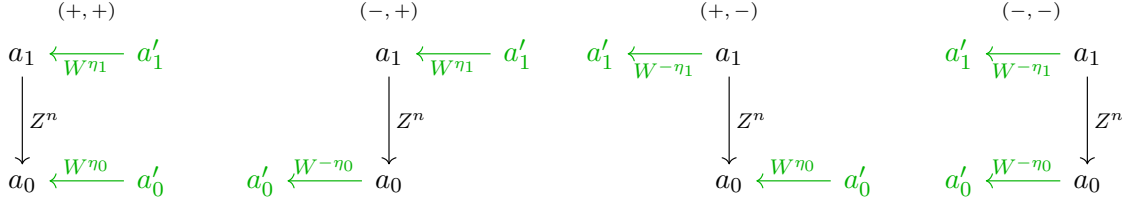
\begin{figure}[hbtp!]\captionsetup{width=\textwidth}
\begin{tikzpicture}
 \node at (-6,1.4) {{\tiny$(+,+)$}};
    \node at (-6,0) [inner sep=0pt] {
\begin{tikzcd}[column sep=1 cm, row sep=0 cm]
 a_1 \ar[d, "Z^n"]   &  {\color{green!70!black} a'_1} \ar[l, color = green!70!black, "W^{\eta_1}"]
\\[1 cm]
 a_0  &  {\color{green!70!black} a'_0} \ar[l, color = green!70!black, "W^{\eta_0}"']
\end{tikzcd}
};
 \node at (-2,1.4) {{\tiny$(-,+)$}};
\node at (-2,0) [inner sep=0pt] {
\begin{tikzcd}[column sep=1 cm, row sep=1 cm]
& a_1 \ar[d, "Z^n"]  &  {\color{green!70!black} a'_1} \ar[l, color = green!70!black, "W^{\eta_1}"]
\\
{\color{green!70!black} a'_0}  & a_0  \ar[l, color = green!70!black, "W^{-\eta_0}"']  &  
\end{tikzcd}
};
 \node at (2.5,1.4) {{\tiny$(+,-)$}};
\node at (2.5,0) [inner sep=0pt] {
\begin{tikzcd}[column sep=1 cm, row sep=1 cm]
{\color{green!70!black} a'_1}  
& a_1 \ar[d, "Z^n"] \ar[l, color = green!70!black, "W^{-\eta_1}"]  &  
\\
  & a_0   &  {\color{green!70!black} a'_0} \ar[l, color = green!70!black, "W^{\eta_0}"']
\end{tikzcd}
};
 \node at (6.5,1.4) {{\tiny$(-,-)$}};
\node at (6.5,0) [inner sep=0pt] {
\begin{tikzcd}[column sep=1 cm, row sep=1 cm]
{\color{green!70!black} a'_1}  
& a_1 \ar[d, "Z^n"] \ar[l, color = green!70!black, "W^{-\eta_1}"]  
\\
{\color{green!70!black} a'_0}  & a_0  \ar[l, color = green!70!black, "W^{-\eta_0}"']  
\end{tikzcd}
};
\end{tikzpicture}
  \caption{The label above each figure indicates the sign of the pair $(\eta_0,\eta_1)$. There may be additional $Z$-weighted arrows to or from $a_1'$ and $a_0'$, which are not shown.}
    \label{fig: non_ending_arrow_cases}
\end{figure}
\begin{rem}
The sign of $\eta_i$  matches the corresponding case in \cite{HLPUnknotting}.
    For example, the type $(-,-)$ corresponds to $\eta^{--}_n$ in \cite[Figure 3]{HLPUnknotting}. 
\end{rem}

We are now ready to state our result regarding $\Ord(P(K, \lambda))$,  in the four cases depicted in Figure \ref{fig: non_ending_arrow_cases}.
As it turns out,  the result does not depend on the framing $\lambda,$ therefore in the following proposition we suppress it from the notation and write $P(K)$ for simplicity.

\begin{prop}\label{prop: non_ending_general_case}
    Suppose there is a non-ending $Z^n$-arrow  in $\cCFK(K)$ for some positive integer and $P$ is an L-space satellite pattern.  
    \begin{enumerate}
        \item \label{it: non_ending_general_case_1} Type $(+,+)$: $\Ord(P(K)) \geq \ell n + \theta$;
        \item \label{it: non_ending_general_case_2} Type $(-,+)$: $\Ord(P(K)) \geq \ell (n-1) + \theta$;
         \item \label{it: non_ending_general_case_3} Type $(+,-)$: $\Ord(P(K)) \geq \ell n + \max\{\kappa,\theta\}$;
         \item \label{it: non_ending_general_case_4} Type $(-,-)$ and $n>1$: $\Ord(P(K)) \geq \ell (n-1) + \max\{\kappa,\theta\}$
    \end{enumerate}
    where $\theta = R_{\frac{\ell}{2}} - g_3(P) - \frac{\ell}{2}$ and $\kappa = R_{\frac{\ell}{2}-1} - g_3(P) + \frac{\ell}{2}.$ 
\end{prop} 
The proof strategy is to construct a pair of free generators $\alpha_0$ and $\alpha_1$ in $\bX^{\diamond}(P,K,\lambda)_{\bF[Z]}$ together with an arrow of the desired length from $\alpha_1$ to $\alpha_0$.

The generators $\alpha_0$ and $\alpha_1$ will be constructed in  \[
\bigoplus_{s\in 1/2+\Z} E^{\diamond}_{s,\frac{\ell-1}{2}}\oplus  \bigoplus_{s\in 1/2+\Z} F^{\diamond}_{s,\frac{\ell-1}{2}}\] where
$ \bigoplus_s E^{\diamond}_{s,\frac{\ell-1}{2}} = \cCFK(K)^{\hat\cR}\boxtimes_{\hat\cR} \scE^{\diamond}_{*,\frac{\ell-1}{2}}$ 
and $\bigoplus_s F^{\diamond}_{s,\frac{\ell-1}{2}} =  \cCFK(K)^{\hat\cR}\boxtimes_{\hat\cR} \scF^{\diamond}_{*,\frac{\ell-1}{2}}.$ 
Unless stated otherwise, we work throughout with the basis induced by the above tensor products. 
Recall that $ \scF^{\diamond}_{s,\frac{\ell-1}{2}} = \cS^{\diamond}$ and 
 \begin{equation}\label{eq: scE-s-def}
    \scE^{\diamond}_{s,\frac{\ell-1}{2}} = \begin{cases}
        \cC^{\diamond}_{\frac{\ell}{2}-1} \quad &\text{if} \quad s<0\\
        \cC^{\diamond}_{\frac{\ell}{2}} \quad &\text{if} \quad s>0.
    \end{cases}
\end{equation}

\begin{define} \label{def: xs-shorthand}
We introduce a shorthand notation as follows. 
 Denote the generator
 $\xs^{t}_0$ of the free $\bF[Z]$-summand in $\scE^{\diamond}_{s,\frac{\ell-1}{2}} $ (see Equation~\eqref{eq: C_t-min-def}) by $ x_{s}$, where $t = \frac{\ell}{2}-1$ or $\frac{\ell}{2}$ depending on $s$ as in Equation~\eqref{eq: scE-s-def}. Similarly denote the generator $\xs^t_1\in \scE^{\diamond}_{s,\frac{\ell-1}{2}}$ by  $ \tilde{x}_s$ if it exists. 
Denote the free generator $\xs'_0$ of the free $\bF[Z]$-tower in $ \scF^{\diamond}_{s,\frac{\ell-1}{2}} = \cS^{\diamond}$ by $x'_s$ and the generator $\xs'_1\in \scF^{\diamond}_{s,\frac{\ell-1}{2}} = \cS^{\diamond}$ by  $ \tilde{x}'_s$ if it exists. Note that $x_s, \tilde{x}_s \in \cC^{\diamond}_{\frac{\ell}{2}-1}$ when $s<0$ and $x_s, \tilde{x}_s \in \cC^{\diamond}_{\frac{\ell}{2}}$ when $s>0$.\end{define}

 By Lemma \ref{lem: C_S_structure_maps},
the structure maps $f^{\pm \mu}\colon \scE^{\diamond}_{*,\frac{\ell-1}{2}} \to \scF^{\diamond}_{*,\frac{\ell-1}{2}} $ restricted to the free generators are as follows
\begin{equation}\label{eq: f_pm_mu_EF}
	 \begin{tikzcd}[labels=description, row sep=2cm] \scE^{\diamond}_{*,\frac{\ell-1}{2}} 
	 	\ar[d, "f^{\pm \mu}"]
	 	\\ 
	 	\scF^{\diamond}_{*,\frac{\ell-1}{2}}
	 \end{tikzcd}
	 \hspace{-.3cm}
	 \supset
	 \hspace{-.3cm}\begin{tikzcd}[column sep=1.5cm, row sep=0cm ] 
& & 
& {\color{gray!80}\tilde{x}_{\frac{1}{2}}}
  \ar[r, "{\color{gray!80}Z|1}", color=gray!80]
& {\color{gray!80}\tilde{x}_{\frac{3}{2}}}
& \cdots
\\[0.5 cm]
	 	\cdots
	 	&[-1.5 cm]
	 	 x_{-\frac{3}{2}}
         \ar[dd, color=gray!80, bend left,dashed, out=-15,in=-145, pos=0.3, "Z^{\xi'_1 - \Xi}"]
	 	&[0.5 cm] x_{-\frac{1}{2}}
           \ar[ur, color=gray!80, bend right,dashed, out=35,in=145, pos=0.5, "Z|Z^{\xi_1^{\ell/2} + 1 - \omega }"]
        \ar[dl,  "Z^{\kappa}"]
	 	\ar[dd, color=gray!80, bend left,dashed, out=-25,in=-145, pos=0.3, "Z^{\xi'_1 - \Xi}"]
	 	\ar[l,  "W|1"'] 	
	 	&[.6cm] x_{\frac{1}{2}}
         \ar[dl, "Z^{\ell+\theta}"]
	 	\ar[r, "Z|1"]
        \ar[l,  "W|Z^ \omega "']
	 	\ar[d,"Z^ \theta", pos=.4]
	 	& x_{\frac{3}{2}}
        \ar[dl, "Z^{\ell+\theta}"]
	 	\ar[d,"Z^ \theta",  pos=.4]
	 	&[-1.5cm] \cdots
	 	\\[2cm]
	 	\cdots
	 	&[.4cm]
	 	x_{-\frac{3}{2}}'
	 	&[0.5 cm]x_{-\frac{1}{2}}'
	 	\ar[r,  "Z|1"']
	 	\ar[l,  "W|1"]
	 	&x_{\frac{1}{2}}'
	 	\ar[r,  "Z|1"']
	 	&x_{\frac{3}{2}}'
	 	&\cdots 
        \\[0.5 cm]
\cdots
& {\color{gray!80}\tilde{x}_{-\frac{3}{2}}'}
& {\color{gray!80}\tilde{x}_{-\frac{1}{2}}'}
  \ar[r, "{\color{gray!80}Z|1}"', color=gray!80]
  \ar[l, "{\color{gray!80}W|1}",  color=gray!80]
& {\color{gray!80}\tilde{x}_{\frac{1}{2}}'}
  \ar[r, "{\color{gray!80}Z|1}"', color=gray!80]
& {\color{gray!80}\tilde{x}_{\frac{3}{2}}'}
& \cdots
	 \end{tikzcd}
\end{equation}

In the above diagram, the gray dashed arrow between the first two rows exists only if $ \omega  > 1$;   The downward gray dashed arrows
 exist only when   $\Xi >0.$ The generators $x_s$ and $x_s'$ are free generators, while the tilde labeled generators are torsion. Compared with Equation (\ref{eq: scE_scF_actions}), here we restrict to certain generators in each $\scE^{\diamond}_{s,\frac{\ell-1}{2}}$, $\scF^{\diamond}_{s,\frac{\ell-1}{2}}$, and omit zero maps.

    Recall from \cite[Section 4]{CZZApp} that we have the following decomposition of chain complexes of $AA$-bimodules: \[{}_{\hat{\cK}} \cX^{\diamond}(L_P)_{\bF[Z]} = \Cone\Big({}_{\hat{\cK}} \cX^{\free}(L_P)_{\bF[Z]} \xrightarrow{f}{}_{\hat{\cK}} \cX^{\Tor}(L_P)_{\bF[Z]}\Big),\]
    which gives a decomposition of chain complexes of $\bF[Z]$-modules by box tensoring $\cX_{\lambda}(K)^{\hat\cK} \boxtimes {}_{\hat\cK}\cH_{-}^{\hat\cK}$ on the left:
    
    \begin{equation}\bX^{\diamond}(P,K,\lambda)_{\bF[Z]} = \Cone \Big(\bX^{\free}(P,K,\lambda)_{\bF[Z]} \xrightarrow{\id\boxtimes\id\boxtimes f} \bX^{\Tor}(P,K,\lambda)_{\bF[Z]}\Big).
    \label{eq:mapping-cone-free-tor}
    \end{equation}
    
    The description of ${}_{\hat{\cK}} \cX^{\free}(L_P)_{\bF[Z]}$ is completely determined by the values of $R_t$ by \cite[Lemma 4.17]{CZZApp}, while 
    ${}_{\hat{\cK}} \cX^{\Tor}(L_P)_{\bF[Z]}$ depends on more information from the $H$-function of $L_P$. However, to get a lower bound on the torsion order, we will mostly focus on ${}_{\hat{\cK}} \cX^{\free}(L_P)_{\bF[Z]}$, and some information about the map $f$, as described in Lemma \ref{lem: C_S_structure_maps}. Our strategy is to find a pair of generators $\alpha_0,\alpha_1$ in $\bX^{\free}(P,K,\lambda)_{\bF[Z]}$ in the mapping cone decomposition of $\bX^{\diamond}(P,K,\lambda)_{\bF[Z]}$ from Equation~\eqref{eq:mapping-cone-free-tor}, such that
    \[\partial\alpha_0=0,\,\, \partial\alpha_1 = \alpha_0 Z^N\] for some $N$, and then apply Lemma \ref{lem: torsion_order_criterion}.

\begin{figure}[hbtp!]
	\captionsetup{width=\textwidth}
	\subfigure[When $\eta_0,\eta_1 >0$.]{
		\centering
		\begin{tikzpicture}[scale=0.5]
			\begin{scope}
				\foreach \j in {-5,5,15}{
					\fill[gray!20] (\j+-1,1.5) rectangle (\j+1.5,0.5);
					\fill[gray!20] (\j+-1,2.5) rectangle (\j+2.5,1.5);
					\fill[gray!20] (\j+-1,3) rectangle (\j+3,2.5);
					\fill[gray!20] (\j+-1,0.5) rectangle (\j+0.5,-0.5);
					\fill[gray!20] (\j+-1,-0.5) rectangle (\j+-0.5,-1.5);
				}
			\foreach \j in {-10,-5,...,10,15}
				{\foreach \i in {0,...,3}
				{\draw[thin, black!20!white]  (\j+\i-0.5, 3) -- (\j+\i-0.5, -2);}
			\foreach \i in {-1,...,3}
			{ \draw[thin, black!20!white]  (\j+3, \i-0.5) -- (\j-1, \i-0.5); }}
			\foreach \j in {-10,0,10}
				{\draw[thin, black!60!white] (\j-1, 2.5) -- (\j+2.5, 2.5);
					\draw[thin, black!60!white] (\j+2.5, 2.5) -- (\j+2.5, -2);
					\draw[thin, black!60!white] (\j+1.5, 1.5) -- (\j+1.5, -2);
					\draw[thin, black!60!white]  (\j+0.5, 0.5) -- (\j+0.5, -2);
					\draw[thin, black!60!white]  (\j+1.5, 1.5) -- (\j-1, 1.5);
					\draw[thin, black!60!white] (\j+0.5, 0.5) -- (\j-1, 0.5);
					\draw[thin, black!60!white] (\j-1, -0.5) -- (\j-0.5, -0.5);
					\draw[thin, black!60!white] (\j-0.5, -0.5) -- (\j-0.5, -2);
					\draw[thin, black!60!white] (\j-0.5, -1.5) -- (\j-0.5, -2);}
			\end{scope} 
			\foreach \j in {-1,0,1}
			{ \foreach \i in {0,1}
				{\filldraw ({\j*10+1+5*\i}, 2-\j-\i) circle (2pt)
					node[] (a1\i\j) {};
					\filldraw ({\j*10+1+5*\i}, -\j-\i) circle (2pt) 
					node[] (a0\i\j) {};
					\ifnum\j<1        \ifnum \i=1         \draw[densely dashed,-stealth] ($(a1\i\j)$) -- ($(a0\i\j)+(0,0.1)$);      \else \draw[-stealth] ($(a1\i\j)$) -- ($(a0\i\j)+(0,0.1)$); \fi \else \draw[-stealth] ($(a1\i\j)$) -- ($(a0\i\j)+(0,0.1)$); \fi
					\filldraw [color=green!70!black] ({\j*10+2+5*\i}, 2-\j-\i) circle (2pt) node[] (a1'\i\j) {};
					\filldraw [color=green!70!black] ({\j*10+2+5*\i}, -\j-\i) circle (2pt) node[] (a0'\i\j) {};
					\draw [color=green!70!black, -stealth] ($(a1'\i\j)$) -- ($(a1\i\j)+(0.1,0)$);
					\draw [color=green!70!black, -stealth] ($(a0'\i\j)$) -- ($(a0\i\j)+(0.1,0)$);
				}
				\draw[red, bend left=20, -stealth] ($(a01\j)+ (-0.1,-0.1)$) to node[midway, below] {{\tiny $Z^{\ell+\theta}$}} ($(a00\j) + (0.1,-0.1)$);
			}      
			\draw[red, bend right=20, -stealth] ($(a010) + (0.1,-0.1)$)  to node[midway, below] {{\tiny $Z^{\theta}$}} ($(a001)+ (-0.1,-0.1)$);
			\draw[red, bend right=20, -stealth] ($(a01-1) + (0.1,-0.1)$)  to node[midway, below] {{\tiny $Z^{\theta}$}} ($(a000)+ (-0.1,-0.1)$);
			\node [left] at (a00-1) {{\tiny$a_0$}};
			\node [left] at (a10-1) {{\tiny$a_1$}};
			\node [color=green!70!black,right] at (a1'0-1) {{\tiny$a'_1$}};
			\node [color=green!70!black,right] at (a0'0-1) {{\tiny$a'_0$}};
			\node [] at ($(a00-1)+(-.35,1)$) {{\tiny$U^2$}};
			\node [] at ($(a01-1)+(-.55,1)$) {{\tiny$UL_Z$}};
			\node [] at ($(a000)+(-.3,1)$) {{\tiny$U$}};
			\node [] at ($(a010)+(-.5,1)$) {{\tiny$L_Z$}};
			\node [] at ($(a001)+(-.3,1)$) {{\tiny$1$}};
			\node [] at ($(a011)+(-.3,1)$) {{\tiny$1$}};
			\node [color=green!70!black] at ($(a10-1)+(0.6,-0.3)$) {{\tiny$1$}};
			\node [color=green!70!black] at ($(a11-1)+(0.6,0.3)$) {{\tiny$1$}};
			\node [color=green!70!black] at ($(a100)+(0.6,0.3)$) {{\tiny$1$}};
			\node [color=green!70!black] at ($(a110)+(0.6,0.3)$) {{\tiny$Z^{\omega}$}};
			\node [color=green!70!black] at ($(a101)+(0.8,0.3)$) {{\tiny$U$}};
			\node [color=green!70!black] at ($(a111)+(0.8,0.3)$) {{\tiny$U$}};
			\node [color=green!70!black] at ($(a00-1)+(0.75,0.3)$) {{\tiny$U$}};
			\node [color=green!70!black] at ($(a01-1)+(0.75,0.3)$) {{\tiny$U$}};
			\node [color=green!70!black] at ($(a000)+(0.75,0.3)$) {{\tiny$U$}};
			\node [color=green!70!black] at ($(a010)+(0.75,0.3)$) {{\tiny$U$}};
			\node [color=green!70!black] at ($(a001)+(0.75,0.3)$) {{\tiny$U$}};
			\node [color=green!70!black] at ($(a011)+(0.75,0.3)$) {{\tiny$U$}};
			\draw[blue] 
			($(a00-1)+(-5pt,-5pt)$) rectangle ($(a00-1)+(5pt,5pt)$);
			\draw[violet] (a01-1) circle (5pt);
			\draw[violet] (a010) circle (5pt);
			\draw[violet] (a101) circle (5pt);
            \node   at (-9,5) {{\tiny $\mathfrak{a}=A(a_0)$}};
			\node   at (-9,4) {{\tiny$F^{\diamond}_{\mathfrak{a}-\frac{1}{2},\frac{\ell-1}{2}}$}};
			\node   at (-4,4) {{\tiny$E^{\diamond}_{\mathfrak{a}+\frac{1}{2},\frac{\ell-1}{2}}$}};
				\node   at (1,4) {{\tiny$F^{\diamond}_{\mathfrak{a}+\frac{1}{2},\frac{\ell-1}{2}}$}};
			\node   at (6,4) {{\tiny$E^{\diamond}_{\mathfrak{a}+\frac{3}{2},\frac{\ell-1}{2}}$}};
			\node   at (11,4) {{\tiny$F^{\diamond}_{\mathfrak{a}+\frac{3}{2},\frac{\ell-1}{2}}$}};
			\node   at (16,4) {{\tiny$E^{\diamond}_{\mathfrak{a}+\frac{5}{2},\frac{\ell-1}{2}}$}};
			
			\draw[red, -stealth] (-11.5,4) to node[midway, above] {{\tiny $L_{\sigma}$}} (-10.5,4);
			\draw[red, -stealth] (-5.5,4) to node[midway, above] {{\tiny $L_{\tau}$}} (-7.5,4);
			\draw[red, -stealth] (-2.5,4) to node[midway, above] {{\tiny $L_{\sigma}$}} (-0.5,4);
			\draw[red, -stealth] (4.5,4) to node[midway, above] {{\tiny $L_{\tau}$}} (2.5,4);
			\draw[red, -stealth] (7.5,4) to node[midway, above] {{\tiny $L_{\sigma}$}} (9.5,4);
			\draw[red, -stealth] (14.5,4) to node[midway, above] {{\tiny $L_{\tau}$}} (12.5,4);
			\draw[red, -stealth] (17.5,4) to node[midway, above] {{\tiny $L_{\sigma}$}} (18.5,4);

\node[axislabel] at (-11.1,-2.3) {$i$};
\node[axislabel] at (-11.5,-2) {$j$};
\node[axislabel] at (-10,-2.3) {$-1$};
\node[axislabel] at (-9,-2.3) {$0$};
\node[axislabel] at (-8,-2.3) {$1$};

\node[axislabel] at (-11.5,2) {$1$};
\node[axislabel] at (-11.5,1) {$0$};
\node[axislabel] at (-11.5,0) {$-1$};
\node[axislabel] at (-11.5,-1) {$-2$};   
		\end{tikzpicture}      
		\label{subfig: non_ending_1}
	}
	\subfigure[When $\eta_0<0,\eta_1 >0$.]{
		\centering
		\begin{tikzpicture}[scale=0.5]
			\begin{scope}
	\foreach \j in {-5,5,15}{
		\fill[gray!20] (\j+-1,1.5) rectangle (\j+1.5,0.5);
		\fill[gray!20] (\j+-1,2.5) rectangle (\j+2.5,1.5);
		\fill[gray!20] (\j+-1,3) rectangle (\j+3,2.5);
		\fill[gray!20] (\j+-1,0.5) rectangle (\j+0.5,-0.5);
		\fill[gray!20] (\j+-1,-0.5) rectangle (\j+-0.5,-1.5);
	}
	\foreach \j in {-10,-5,...,10,15}
	{\foreach \i in {0,...,3}
		{\draw[thin, black!20!white]  (\j+\i-0.5, 3) -- (\j+\i-0.5, -2);}
		\foreach \i in {-1,...,3}
		{ \draw[thin, black!20!white]  (\j+3, \i-0.5) -- (\j-1, \i-0.5); }}
	\foreach \j in {-10,0,10}
	{\draw[thin, black!60!white] (\j-1, 2.5) -- (\j+2.5, 2.5);
		\draw[thin, black!60!white] (\j+2.5, 2.5) -- (\j+2.5, -2);
		\draw[thin, black!60!white] (\j+1.5, 1.5) -- (\j+1.5, -2);
		\draw[thin, black!60!white]  (\j+0.5, 0.5) -- (\j+0.5, -2);
		\draw[thin, black!60!white]  (\j+1.5, 1.5) -- (\j-1, 1.5);
		\draw[thin, black!60!white] (\j+0.5, 0.5) -- (\j-1, 0.5);
		\draw[thin, black!60!white] (\j-1, -0.5) -- (\j-0.5, -0.5);
		\draw[thin, black!60!white] (\j-0.5, -0.5) -- (\j-0.5, -2);
		\draw[thin, black!60!white] (\j-0.5, -1.5) -- (\j-0.5, -2);}
\end{scope} 
			\foreach \j in {-1,0,1}
			{ \foreach \i in {0,1}
				{\filldraw ({\j*10+1+5*\i}, 2-\j-\i) circle (2pt)
					node[] (a1\i\j) {};
					\filldraw ({\j*10+1+5*\i}, -\j-\i) circle (2pt) 
					node[] (a0\i\j) {};
					\ifnum\j<1        \ifnum \i=1         \draw[densely dashed,-stealth] ($(a1\i\j)$) -- ($(a0\i\j)+(0,0.1)$);      \else \draw[-stealth] ($(a1\i\j)$) -- ($(a0\i\j)+(0,0.1)$); \fi \else \draw[-stealth] ($(a1\i\j)$) -- ($(a0\i\j)+(0,0.1)$); \fi
					\filldraw [color=green!70!black] ({\j*10+2+5*\i}, 2-\j-\i) circle (2pt) node[] (a1'\i\j) {};
					\filldraw [color=green!70!black] ({\j*10+5*\i}, -\j-\i) circle (2pt) node[] (a0'\i\j) {};
					\draw [color=green!70!black, -stealth] ($(a1'\i\j)$) -- ($(a1\i\j)+(0.1,0)$);
					\draw [color=green!70!black, -stealth] ($(a0\i\j)$) -- ($(a0'\i\j)+(0.1,0)$);
				}
				\ifnum\j=-1\relax
				\else
				\draw[red, bend left=20, -stealth] ($(a01\j)+ (-0.1,-0.1)$) to node[midway, below] {{\tiny $Z^{\ell+\theta}$}} ($(a00\j) + (0.1,-0.1)$);
				\fi
			}    
			\draw[red, in=10, out=165, -stealth] ($(a01-1)+ (-0.1,0.1)$) to node[pos=0.55, above] {{\tiny $Z^{\ell+\theta}$}} ($(a00-1) + (0.1,0)$);
			\draw[red, in=-35, out=-175, -stealth] ($(a0'1-1)+ (-0.1,-0.1)$) to node[pos=0.35, below] {{\tiny $Z^{\kappa}$}} ($(a0'0-1) + (0.1,-0.1)$);
			\draw[red, densely dashed, in=-135, out=-95, -stealth] ($(a0'1-1)+ (0.1,-0.1)$) to node[midway, above] {{\tiny $Z^{\xi'_1-\Xi}$}} ($(a0'00) + (-0.1,-0.1)$);
			
			\draw[red, bend right=20, -stealth] ($(a010) + (0.1,-0.1)$)  to node[midway, below] {{\tiny $Z^{\theta}$}} ($(a001)+ (-0.1,-0.1)$);
			\draw[red, bend left=10, -stealth] ($(a01-1) + (0.1,0.1)$)  to node[midway, above] {{\tiny $Z^{\theta}$}} ($(a000)+ (-0.1,0.1)$);
			\node [below] at (a00-1) {{\tiny$a_0$}};
			\node [left] at (a10-1) {{\tiny$a_1$}};
			\node [color=green!70!black,right] at (a1'0-1) {{\tiny$a'_1$}};
			\node [color=green!70!black,left] at (a0'0-1) {{\tiny$a'_0$}};
			\node [] at ($(a00-1)+(.45,1)$) {{\tiny$U^2$}};
			\node [] at ($(a01-1)+(.7,1)$) {{\tiny$UL_Z$}};
			\node [] at ($(a000)+(.3,1)$) {{\tiny$U$}};
			\node [] at ($(a010)+(.5,1)$) {{\tiny$L_Z$}};
			\node [] at ($(a001)+(.3,1)$) {{\tiny$1$}};
			\node [] at ($(a011)+(.3,1)$) {{\tiny$1$}};
			\node [color=green!70!black] at ($(a10-1)+(0.6,-0.3)$) {{\tiny$1$}};
			\node [color=green!70!black] at ($(a11-1)+(0.6,0.3)$) {{\tiny$1$}};
			\node [color=green!70!black] at ($(a100)+(0.6,0.3)$) {{\tiny$1$}};
			\node [color=green!70!black] at ($(a110)+(0.6,0.3)$) {{\tiny$Z^{\omega}$}};
			\node [color=green!70!black] at ($(a101)+(0.8,0.3)$) {{\tiny$U$}};
			\node [color=green!70!black] at ($(a111)+(0.8,0.3)$) {{\tiny$U$}};
			\node [color=green!70!black] at ($(a00-1)+(-0.45,0.3)$) {{\tiny$1$}};
			\node [color=green!70!black] at ($(a01-1)+(-0.45,-0.4)$) {{\tiny$Z^{\omega}$}};
			\node [color=green!70!black] at ($(a000)+(-0.25,-0.3)$) {{\tiny$U$}};
			\node [color=green!70!black] at ($(a010)+(-0.25,0.3)$) {{\tiny$U$}};
			\node [color=green!70!black] at ($(a001)+(-0.25,0.3)$) {{\tiny$U$}};
			\node [color=green!70!black] at ($(a011)+(-0.25,0.3)$) {{\tiny$U$}};
			\draw[blue]     ($(a00-1)+(-5pt,-5pt)$) rectangle ($(a00-1)+(5pt,5pt)$);
			\draw[blue]     ($(a0'1-1)+(-5pt,-5pt)$) rectangle ($(a0'1-1)+(5pt,5pt)$);
			\draw[violet] (a01-1) circle (5pt);
			\draw[violet] (a010) circle (5pt);
			\draw[violet] (a101) circle (5pt);

            \node[axislabel] at (-11.1,-2.3) {$i$};
\node[axislabel] at (-11.5,-2) {$j$};
\node[axislabel] at (-10,-2.3) {$-1$};
\node[axislabel] at (-9,-2.3) {$0$};
\node[axislabel] at (-8,-2.3) {$1$};

\node[axislabel] at (-11.5,2) {$1$};
\node[axislabel] at (-11.5,1) {$0$};
\node[axislabel] at (-11.5,0) {$-1$};
\node[axislabel] at (-11.5,-1) {$-2$};
		\end{tikzpicture}      
		\label{subfig: non_ending_2}
	}
	\subfigure[When $\eta_0>0,\eta_1 <0$.]{
		\centering
		\begin{tikzpicture}[scale=0.5]
			\begin{scope}
				\foreach \j in {-5,5,15}{
					\fill[gray!20] (\j+-1,1.5) rectangle (\j+1.5,0.5);
					\fill[gray!20] (\j+-1,2.5) rectangle (\j+2.5,1.5);
					\fill[gray!20] (\j+-1,3) rectangle (\j+3,2.5);
					\fill[gray!20] (\j+-1,0.5) rectangle (\j+0.5,-0.5);
					\fill[gray!20] (\j+-1,-0.5) rectangle (\j+-0.5,-1.5);
				}
				\foreach \j in {-10,-5,...,10,15}
				{\foreach \i in {0,...,3}
					{\draw[thin, black!20!white]  (\j+\i-0.5, 3) -- (\j+\i-0.5, -2);}
					\foreach \i in {-1,...,3}
					{ \draw[thin, black!20!white]  (\j+3, \i-0.5) -- (\j-1, \i-0.5); }}
				\foreach \j in {-10,0,10}
				{\draw[thin, black!60!white] (\j-1, 2.5) -- (\j+2.5, 2.5);
					\draw[thin, black!60!white] (\j+2.5, 2.5) -- (\j+2.5, -2);
					\draw[thin, black!60!white] (\j+1.5, 1.5) -- (\j+1.5, -2);
					\draw[thin, black!60!white]  (\j+0.5, 0.5) -- (\j+0.5, -2);
					\draw[thin, black!60!white]  (\j+1.5, 1.5) -- (\j-1, 1.5);
					\draw[thin, black!60!white] (\j+0.5, 0.5) -- (\j-1, 0.5);
					\draw[thin, black!60!white] (\j-1, -0.5) -- (\j-0.5, -0.5);
					\draw[thin, black!60!white] (\j-0.5, -0.5) -- (\j-0.5, -2);
					\draw[thin, black!60!white] (\j-0.5, -1.5) -- (\j-0.5, -2);}
			\end{scope} 
			\foreach \j in {-1,0,1}
			{ \foreach \i in {0,1}
				{\filldraw ({\j*10+1+5*\i}, 2-\j-\i) circle (2pt)
					node[] (a1\i\j) {};
					\filldraw ({\j*10+1+5*\i}, -\j-\i) circle (2pt) 
					node[] (a0\i\j) {};
					\ifnum\j<1        \ifnum \i=1         \draw[densely dashed,-stealth] ($(a1\i\j)$) -- ($(a0\i\j)+(0,0.1)$);      \else \draw[-stealth] ($(a1\i\j)$) -- ($(a0\i\j)+(0,0.1)$); \fi \else \draw[-stealth] ($(a1\i\j)$) -- ($(a0\i\j)+(0,0.1)$); \fi
					\filldraw [color=green!70!black] ({\j*10+5*\i}, 2-\j-\i) circle (2pt) node[] (a1'\i\j) {};
					\filldraw [color=green!70!black] ({\j*10+2+5*\i}, -\j-\i) circle (2pt) node[] (a0'\i\j) {};
					\draw [color=green!70!black, -stealth] ($(a1\i\j)$) -- ($(a1'\i\j)+(0.1,0)$);
					\draw [color=green!70!black, -stealth] ($(a0'\i\j)$) -- ($(a0\i\j)+(0.1,0)$);
				}
				\draw[red, bend left=20, -stealth] ($(a01\j)+ (-0.1,-0.1)$) to node[midway, below] {{\tiny $Z^{\ell+\theta}$}} ($(a00\j) + (0.1,-0.1)$);
			}      
			\draw[red, bend right=5, -stealth] ($(a111)+ (-0.1,0.1)$) to node[pos=0.5, above] {{\tiny $Z^{\ell+\theta}$}} ($(a101) + (0.1,0)$);
			\draw[red, bend left=5, -stealth] ($(a1'11)+ (-0.1,0)$) to node[pos=0.3, below] {{\tiny $Z^{\kappa}$}} ($(a1'01) + (0.1,-0.1)$);
			\draw[violet] (a1'11) circle (5pt);
			\draw[red, densely dashed, in=165, out=85, -stealth] ($(a1'11)+ (0.1,0.1)$) to  ($(a1'11)+ (2.5,0.7)$);
			
			\draw[red, bend right=20, -stealth] ($(a010) + (0.1,-0.1)$)  to node[midway, below] {{\tiny $Z^{\theta}$}} ($(a001)+ (-0.1,-0.1)$);     
			\draw[red, bend right=20, -stealth] ($(a01-1) + (0.1,-0.1)$)  to node[midway, below] {{\tiny $Z^{\theta}$}} ($(a000)+ (-0.1,-0.1)$);
			\node [left] at (a00-1) {{\tiny$a_0$}};
			\node [right] at (a10-1) {{\tiny$a_1$}};
			\node [color=green!70!black,left] at (a1'0-1) {{\tiny$a'_1$}};
			\node [color=green!70!black,right] at (a0'0-1) {{\tiny$a'_0$}};
			\node [] at ($(a00-1)+(-.35,1)$) {{\tiny$U^2$}};
			\node [] at ($(a01-1)+(-.65,1)$) {{\tiny$UL_Z$}};
			\node [] at ($(a000)+(-.3,1)$) {{\tiny$U$}};
			\node [] at ($(a010)+(-.5,1)$) {{\tiny$L_Z$}};
			\node [] at ($(a001)+(-.3,1)$) {{\tiny$1$}};
			\node [] at ($(a011)+(.3,1)$) {{\tiny$1$}};
			\node [color=green!70!black] at ($(a10-1)+(-0.4,0.3)$) {{\tiny$1$}};
			\node [color=green!70!black] at ($(a11-1)+(-0.4,0.3)$) {{\tiny$1$}};
			\node [color=green!70!black] at ($(a100)+(-0.4,0.3)$) {{\tiny$1$}};
			\node [color=green!70!black] at ($(a110)+(-0.4,0.3)$) {{\tiny$1$}};
			\node [color=green!70!black] at ($(a101)+(-0.4,0.3)$) {{\tiny$1$}};
			\node [color=green!70!black] at ($(a111)+(-0.4,-0.3)$) {{\tiny$Z^{\omega}$}};
			\node [color=green!70!black] at ($(a00-1)+(0.75,0.3)$) {{\tiny$U$}};
			\node [color=green!70!black] at ($(a01-1)+(0.75,0.3)$) {{\tiny$U$}};
			\node [color=green!70!black] at ($(a000)+(0.75,0.3)$) {{\tiny$U$}};
			\node [color=green!70!black] at ($(a010)+(0.75,0.3)$) {{\tiny$U$}};
			\node [color=green!70!black] at ($(a001)+(0.75,0.3)$) {{\tiny$U$}};
			\node [color=green!70!black] at ($(a011)+(0.75,0.3)$) {{\tiny$U$}};
			\draw[blue] 
			($(a00-1)+(-5pt,-5pt)$) rectangle ($(a00-1)+(5pt,5pt)$);
			\draw[violet] (a01-1) circle (5pt);
			\draw[violet] (a010) circle (5pt);
			\draw[violet] (a101) circle (5pt);

            \node[axislabel] at (-11.1,-2.3) {$i$};
\node[axislabel] at (-11.5,-2) {$j$};
\node[axislabel] at (-10,-2.3) {$-1$};
\node[axislabel] at (-9,-2.3) {$0$};
\node[axislabel] at (-8,-2.3) {$1$};

\node[axislabel] at (-11.5,2) {$1$};
\node[axislabel] at (-11.5,1) {$0$};
\node[axislabel] at (-11.5,0) {$-1$};
\node[axislabel] at (-11.5,-1) {$-2$};
		\end{tikzpicture}      
		\label{subfig: non_ending_3}
	}
	\subfigure[When $\eta_0,\eta_1 <0$.]{
		\centering
		\begin{tikzpicture}[scale=0.5]
			\begin{scope}
				\foreach \j in {-5,5,15}{
					\fill[gray!20] (\j+-1,1.5) rectangle (\j+1.5,0.5);
					\fill[gray!20] (\j+-1,2.5) rectangle (\j+2.5,1.5);
					\fill[gray!20] (\j+-1,3) rectangle (\j+3,2.5);
					\fill[gray!20] (\j+-1,0.5) rectangle (\j+0.5,-0.5);
					\fill[gray!20] (\j+-1,-0.5) rectangle (\j+-0.5,-1.5);
				}
				\foreach \j in {-10,-5,...,10,15}
				{\foreach \i in {0,...,3}
					{\draw[thin, black!20!white]  (\j+\i-0.5, 3) -- (\j+\i-0.5, -2);}
					\foreach \i in {-1,...,3}
					{ \draw[thin, black!20!white]  (\j+3, \i-0.5) -- (\j-1, \i-0.5); }}
				\foreach \j in {-10,0,10}
				{\draw[thin, black!60!white] (\j-1, 2.5) -- (\j+2.5, 2.5);
					\draw[thin, black!60!white] (\j+2.5, 2.5) -- (\j+2.5, -2);
					\draw[thin, black!60!white] (\j+1.5, 1.5) -- (\j+1.5, -2);
					\draw[thin, black!60!white]  (\j+0.5, 0.5) -- (\j+0.5, -2);
					\draw[thin, black!60!white]  (\j+1.5, 1.5) -- (\j-1, 1.5);
					\draw[thin, black!60!white] (\j+0.5, 0.5) -- (\j-1, 0.5);
					\draw[thin, black!60!white] (\j-1, -0.5) -- (\j-0.5, -0.5);
					\draw[thin, black!60!white] (\j-0.5, -0.5) -- (\j-0.5, -2);
					\draw[thin, black!60!white] (\j-0.5, -1.5) -- (\j-0.5, -2);}
			\end{scope} 
			\foreach \j in {-1,0,1}
			{ \foreach \i in {0,1}
				{\filldraw ({\j*10+1+5*\i}, 2-\j-\i) circle (2pt)
					node[] (a1\i\j) {};
					\filldraw ({\j*10+1+5*\i}, -\j-\i) circle (2pt) 
					node[] (a0\i\j) {};
					\ifnum\j<1        \ifnum \i=1         \draw[densely dashed,-stealth] ($(a1\i\j)$) -- ($(a0\i\j)+(0,0.1)$);      \else \draw[-stealth] ($(a1\i\j)$) -- ($(a0\i\j)+(0,0.1)$); \fi \else \draw[-stealth] ($(a1\i\j)$) -- ($(a0\i\j)+(0,0.1)$); \fi
					\filldraw [color=green!70!black] ({\j*10+5*\i}, 2-\j-\i) circle (2pt) node[] (a1'\i\j) {};
					\filldraw [color=green!70!black] ({\j*10+5*\i}, -\j-\i) circle (2pt) node[] (a0'\i\j) {};
					\draw [color=green!70!black, -stealth] ($(a1\i\j)$) -- ($(a1'\i\j)+(0.1,0)$);
					\draw [color=green!70!black, -stealth] ($(a0\i\j)$) -- ($(a0'\i\j)+(0.1,0)$);
				}
				\ifnum\j=-1\relax
				\else
				\draw[red, bend left=20, -stealth] ($(a01\j)+ (-0.1,-0.1)$) to node[midway, below] {{\tiny $Z^{\ell+\theta}$}} ($(a00\j) + (0.1,-0.1)$);
				\fi
			}    
			\draw[red, in=10, out=165, -stealth] ($(a01-1)+ (-0.1,0.1)$) to node[pos=0.55, above] {{\tiny $Z^{\ell+\theta}$}} ($(a00-1) + (0.1,0)$);
			\draw[red, in=-35, out=-175, -stealth] ($(a0'1-1)+ (-0.1,-0.1)$) to node[pos=0.35, below] {{\tiny $Z^{\kappa}$}} ($(a0'0-1) + (0.1,-0.1)$);
			\draw[red, bend right=5, -stealth] ($(a111)+ (-0.1,0.1)$) to node[pos=0.5, above] {{\tiny $Z^{\ell+\theta}$}} ($(a101) + (0.1,0)$);
			\draw[red, bend left=5, -stealth] ($(a1'11)+ (-0.1,0)$) to node[pos=0.3, below] {{\tiny $Z^{\kappa}$}} ($(a1'01) + (0.1,-0.1)$);
			\draw[violet] (a1'11) circle (5pt);
			
			\draw[red, densely dashed, in=165, out=85, -stealth] ($(a1'11)+ (0.1,0.1)$) to  ($(a1'11)+ (2.5,0.7)$);
			\draw[red, densely dashed, in=-135, out=-95, -stealth] ($(a0'1-1)+ (0.1,-0.1)$) to node[midway, above] {{\tiny $Z^{\xi'_1-\Xi}$}} ($(a0'00) + (-0.1,-0.1)$);
			
			\draw[red, bend right=20, -stealth] ($(a010) + (0.1,-0.1)$)  to node[midway, below] {{\tiny $Z^{\theta}$}} ($(a001)+ (-0.1,-0.1)$);
			\draw[red, bend left=10, -stealth] ($(a01-1) + (0.1,0.1)$)  to node[midway, above] {{\tiny $Z^{\theta}$}} ($(a000)+ (-0.1,0.1)$);
			\node [below] at (a00-1) {{\tiny$a_0$}};
			\node [right] at (a10-1) {{\tiny$a_1$}};
			\node [color=green!70!black,left] at (a1'0-1) {{\tiny$a'_1$}};
			\node [color=green!70!black,left] at (a0'0-1) {{\tiny$a'_0$}};
			\node [] at ($(a00-1)+(.45,1)$) {{\tiny$U^2$}};
			\node [] at ($(a01-1)+(.7,1)$) {{\tiny$UL_Z$}};
			\node [] at ($(a000)+(.3,1)$) {{\tiny$U$}};
			\node [] at ($(a010)+(.5,1)$) {{\tiny$L_Z$}};
			\node [] at ($(a001)+(.3,1)$) {{\tiny$1$}};
			\node [] at ($(a011)+(.3,1)$) {{\tiny$1$}};
			\node [color=green!70!black] at ($(a10-1)+(-0.4,0.3)$) {{\tiny$1$}};
			\node [color=green!70!black] at ($(a11-1)+(-0.4,0.3)$) {{\tiny$1$}};
			\node [color=green!70!black] at ($(a100)+(-0.4,0.3)$) {{\tiny$1$}};
			\node [color=green!70!black] at ($(a110)+(-0.4,0.3)$) {{\tiny$1$}};
			\node [color=green!70!black] at ($(a101)+(-0.4,0.3)$) {{\tiny$1$}};
			\node [color=green!70!black] at ($(a111)+(-0.4,-0.3)$) {{\tiny$Z^{\omega}$}};
			\node [color=green!70!black] at ($(a00-1)+(-0.45,0.3)$) {{\tiny$1$}};
			\node [color=green!70!black] at ($(a01-1)+(-0.45,-0.4)$) {{\tiny$Z^{\omega}$}};
			\node [color=green!70!black] at ($(a000)+(-0.25,-0.3)$) {{\tiny$U$}};
			\node [color=green!70!black] at ($(a010)+(-0.25,0.3)$) {{\tiny$U$}};
			\node [color=green!70!black] at ($(a001)+(-0.25,0.3)$) {{\tiny$U$}};
			\node [color=green!70!black] at ($(a011)+(-0.25,0.3)$) {{\tiny$U$}};
			\draw[blue]     ($(a00-1)+(-5pt,-5pt)$) rectangle ($(a00-1)+(5pt,5pt)$);
			\draw[blue]     ($(a0'1-1)+(-5pt,-5pt)$) rectangle ($(a0'1-1)+(5pt,5pt)$);
			\draw[violet] (a01-1) circle (5pt);
			\draw[violet] (a010) circle (5pt);
			\draw[violet] (a101) circle (5pt);

            \node[axislabel] at (-11.1,-2.3) {$i$};
\node[axislabel] at (-11.5,-2) {$j$};
\node[axislabel] at (-10,-2.3) {$-1$};
\node[axislabel] at (-9,-2.3) {$0$};
\node[axislabel] at (-8,-2.3) {$1$};

\node[axislabel] at (-11.5,2) {$1$};
\node[axislabel] at (-11.5,1) {$0$};
\node[axislabel] at (-11.5,0) {$-1$};
\node[axislabel] at (-11.5,-1) {$-2$};
		\end{tikzpicture}      
		\label{subfig: non_ending_4}
	}
\caption{Schematic illustration for the proof of Proposition \ref{prop: non_ending_general_case} when $n=2$. Each solid dot represents a generator of the form $a|x_s$ or $a|x'_s$. The first tensor factor $a \in \cCFK(K)^{\hat\cR}$ is labeled in the leftmost diagram of each case, while the index $s$ of  $x_s\in \scE_{s,\frac{\ell -1}{2}}$ or $x'_s\in \scF_{s,\frac{\ell -1}{2}}$ is given by $s=i-j-\frac{1}{2}$. The dotted arrows point to torsion elements. The element $\alpha_0$ is defined as a linear combination of the boxed generators, while $\alpha_1$ is defined as a linear combination of the circled generators. In the second and fourth pictures, we are illustrating only the case that $\omega-\theta>0$.} 
\label{fig: prop_non_ending}
\end{figure}

Before proceeding to the proof of Proposition \ref{prop: non_ending_general_case}, we first establish a useful lemma.
\begin{lem} \label{lem: horizontal_chain} Equip $\cCFK(K)^{\hat{\cR}}$ with a basis induced by a decomposition into a standard complex and a collection of local system complexes, as in Theorem~\ref{thm:standard_decomp}.    Suppose that $a \in \cCFK(K)^{\hat\cR}$ is a basis element in this decomposition which admits an incoming arrow weighted by $Z^n$ for some $n>0$, and possibly another incoming arrow weighted by $W^\eta$ for some $\eta>0$. With respect to the corresponding basis of $\bX^{\diamond}(P,K,\lambda)_{\bF[Z]}$,
\begin{itemize}
    \item the only incoming arrow to $a | x'_{-\frac{1}{2}}$ is the arrow of $\Phi^{-\mu}$ from $a | x_{\frac{1}{2}}$;
    \item for each $i=1,\dots,n-1$, the only incoming arrows to $a | x'_{i-\frac{1}{2}}$ are the arrow of $\Phi^{\mu}$ from $a|x_{i-\frac{1}{2}}$ and the arrow of $\Phi^{-\mu}$ from $a|x_{i+\frac{1}{2}}$; and
    \item the only incoming arrows to $a | x'_{n-\frac{1}{2}}$ are the length $0$-differential weighted by $1$, the arrow of $\Phi^{\mu}$ from $a|x_{n-\frac{1}{2}}$ and the arrow of $\Phi^{-\mu}$ from $a|x_{n+\frac{1}{2}}$.
\end{itemize}
\end{lem}

\begin{proof}
    Clearly $a | x'_{i-\frac{1}{2}}$ admits no incoming arrows of $\Phi^{\pm K}$, since those have images in $J_{s,t}$ or $M_{s,t}$. Consider length-$0$ differentials, 
    which take the form of
  \[
       \begin{tikzcd}[column sep=0 cm, row sep=0.3 cm]
     \cCFK(K)^{\hat\cR}  \ar[d ]& \otimes& \scF^{\diamond}\ar[dd] \\
     \delta^1_{\hat\cR} \ar[dd]\ar[drr] \\
     & & m_{1|1|0}  \ar[d]  \\
     \cCFK(K)^{\hat\cR} &  \otimes  & \scF^{\diamond} 
	 \end{tikzcd}.
     \]   
     In $\cCFK(K)^{\hat\cR}$, $a$ admits an incoming arrow weighted by $Z^n$ from some generator $b$ and possibly another incoming arrow weighted by $W^{\eta}$  from some generator $c$. 
    Therefore any incoming  length-$0$ differential to $a | x'_{i-\frac{1}{2}}$ would have to originate from either $b|x'_{i-\frac{1}{2}-n}$ or $c|x'_{i-\frac{1}{2}+\eta}$. We now compute using \eqref{eq: scE_scF_actions}. 
    In the former case, when $i=0,\dots,n-1,$ since $i-\frac{1}{2}-n < -1$,  $m_{1|1|0}(Z^n,x'_{i-\frac{1}{2}-n})$  vanishes and when $i=n,$  $m_{1|1|0}(Z^n,x'_{-\frac{1}{2}})=1$; in the latter case, for any $i\geq 0,$ since  $i-\frac{1}{2}+\eta > 0,$ $m_{1|1|0}(W^{\eta},x'_{i-\frac{1}{2}+\eta})$ vanishes.
   
    Therefore, for $i=0,\dots,n-1$,  the generators $a | x'_{i-\frac{1}{2}}$ admit no incoming length-$0$ differentials whereas the generators $a | x'_{n-\frac{1}{2}}$  admits a length-0 differential weighted by $1$.
    It remains to consider the arrow of $\Phi^{\mu}$  originating from $a | x_{i-\frac{1}{2}}$, which is of the form  $\bI | L_\sigma$,  and the arrow of $\Phi^{-\mu}$ originating from $a | x_{i+\frac{1}{2}}$, which is of the form $\bI | L_\tau$. For $i=1,\dots,n,$ by Lemma \ref{lem: C_S_structure_maps}, both arrows point to $a | x'_{i-\frac{1}{2}}$ with nonzero coefficients.
    When $i=0,$
 there is no arrow of the form $\bI | L_\sigma$  pointing to $a | x'_{-\frac{1}{2}}$, since such an arrow would have to originate from $a | x_{-\frac{1}{2}}$, but $L_{\sigma}(x_{-\frac{1}{2}})$ is torsion by Lemma \ref{lem: C_S_structure_maps}. On the other hand, the arrow of the form   $\bI | L_\tau$ pointing to $a | x'_{-\frac{1}{2}}$  is still weighted by a nonnegative power of $Z$. 
\end{proof}

\begin{figure}[hbtp!]
\begin{tikzpicture}
\draw[->, thick, decorate, decoration={snake, amplitude=1.5pt, segment length=8pt}, black!70!white]
(-0.65,0) -- (0.2,0);

\node at (-4,0)
{
\begin{tikzpicture}[scale=0.7]
\begin{scope}
    \draw[very thick, <-, black!20!white]  (6.5, -3.5) -- (6.5, 5) node[above,black!20!white] {{\small $Z$}}; 
    \draw[very thick, ->, black!0!white]  (-1, 0) -- (6, 0);
\end{scope}

\filldraw (5.5, 4) circle (2pt) node[above] {{\tiny $a_1|x'_{-\frac{1}{2}}$}} node (a0) {};
\filldraw (5, 3 ) circle (2pt) node[below] {{\tiny $a_0|x'_{\frac{5}{2}}$}} node (b1) {};
\filldraw (4, 4.5 ) circle (2pt) node[above] {{\tiny $a_0|x_{\frac{5}{2}}$}} node (a1) {};
\filldraw (3, 1 ) circle (2pt) node[below] {{\tiny $a_0|x'_{\frac{3}{2}}$}} node (b2) {};
\filldraw (2, 2.5 ) circle (2pt) node[above left] {{\tiny $a_0|x_{\frac{3}{2}}$}} node (a2) {};
\filldraw (1, -1 ) circle (2pt) node[below] {{\tiny $a_0|x'_{\frac{1}{2}}$}} node (b3) {};
\filldraw (0, 0.5 ) circle (2pt) node[above left] {{\tiny $a_0|x_{\frac{1}{2}}$}} node (a3) {};
\filldraw (-1, -3 ) circle (2pt)  node[right, xshift=3pt, yshift=-5pt] {{\tiny $a_0|x'_{-\frac{1}{2}}$}} node (b4) {};

\draw[->, thick] (a0) -- (b1) node[pos=0.6, right, color=red] {{\tiny $1$}};

\draw[->, thick] (a1) -- (b1) node[pos=0.3, right, color=red] {{\tiny $Z^\theta$}};
\draw[->, thick] (a1) -- (b2) node[pos=0.3, left, color=red] {{\tiny $Z^{\ell+\theta}$}};

\draw[->, thick] (a2) -- (b2) node[pos=0.3, right, color=red] {{\tiny $Z^\theta$}};
\draw[->, thick] (a2) -- (b3) node[pos=0.3, left, color=red] {{\tiny $Z^{\ell+\theta}$}};

\draw[->, thick] (a3) -- (b3) node[pos=0.3, right, color=red] {{\tiny $Z^\theta$}};
\draw[->, thick] (a3) -- (b4) node[pos=0.3, left, color=red] {{\tiny $Z^{\ell+\theta}$}};

\draw[violet] (a0) circle (5pt);
\draw[violet] (a1) circle (5pt);
\draw[violet] (a2) circle (5pt);
\draw[violet] (a3) circle (5pt);
\draw[blue] 
       ($(b4)+(-5pt,-5pt)$) rectangle ($(b4)+(5pt,5pt)$);
       \end{tikzpicture}
};
\node at (4,0)
{
\begin{tikzpicture}[scale=0.7]
\begin{scope}
    \draw[very thick, <-, black!20!white]  (6.5, -3.5) -- (6.5, 5) node[above,black!20!white] {{\small $Z$}}; 
    \draw[very thick, ->, black!0!white]  (-1, 0) -- (6, 0);
\end{scope}

\filldraw (5.5, 4) circle (2pt) node[above] {{\tiny $a_1|x'_{-\frac{1}{2}}$}} node (a0) {};
\filldraw (5, 3 ) circle (2pt) node[below] {{\tiny $a_0|x'_{\frac{5}{2}}$}} node (b1) {};

\filldraw (4, 4.5 ) circle (2pt) node[xshift=-50pt, yshift=13pt] {{\tiny $a_0|(x_{\frac{5}{2}} + x_{\frac{3}{2}} Z^{\ell} + x_{\frac{1}{2}} Z^{2\ell} ) + a_1|x'_{-\frac{1}{2}} Z^{\theta}$}} node (a1) {};
\filldraw (1, -3 ) circle (2pt) node[right, xshift=3pt, yshift=-5pt] {{\tiny $a_0|x'_{-\frac{1}{2}}$}} node (b4) {};

\draw[->, thick] (a1) -- (b4) node[pos=0.5, right, color=red] {{\tiny $Z^{3\ell+\theta}$}};

\filldraw (2, 1 ) circle (2pt) node[left] {{\tiny $a_0|x'_{\frac{3}{2}}$}} node (b2) {};
\node at (1.3, .4 ) {{\tiny $+a_0|x'_{\frac{1}{2}} Z^{\ell}$}};
\filldraw (1, 2.5 ) circle (2pt) node[above ] {{\tiny $a_0|x_{\frac{3}{2}}$}} node (a2) {};
\filldraw (0, -1 ) circle (2pt) node[below] {{\tiny $a_0|x'_{\frac{1}{2}}$}} node (b3) {};
\node at (0.1, -2.1 ) {{\tiny $+a_0|x'_{-\frac{1}{2}} Z^{\ell}$}};
\filldraw (-1, 0.5 ) circle (2pt) node[above] {{\tiny $a_0|x_{\frac{1}{2}}$}} node (a3) {};

\draw[->, thick] (a0) -- (b1) node[pos=0.6, right, color=red] {{\tiny $1$}};

\draw[->, thick] (a2) -- (b2) node[pos=0.3, right, color=red] {{\tiny $Z^\theta$}};

\draw[->, thick] (a3) -- (b3) node[pos=0.3, right, color=red] {{\tiny $Z^\theta$}};
\draw[violet] (a1) circle (5pt);
\draw[blue] 
       ($(b4)+(-5pt,-5pt)$) rectangle ($(b4)+(5pt,5pt)$);  
\end{tikzpicture}};
\end{tikzpicture}

\caption{A schematic illustration of the change of basis in Case \eqref{it: non_ending_general_case_1} of Proposition \ref{prop: non_ending_general_case} when $n=3$. Left: before the change of basis, where arrows of $\Phi^{-\mu}$ are weighted by $Z^{\ell + \theta}$, arrows of  $\Phi^{\mu}$ are weighted by $Z^{\theta}$ and the length-$0$ differential is  weighted by $1$. Right: after the change of basis. The $y$-coordinate is determined by the $\gr_Z$-grading, while the $x$-coordinate has no intrinsic meaning. }\label{fig: Zzigzag}
\end{figure}
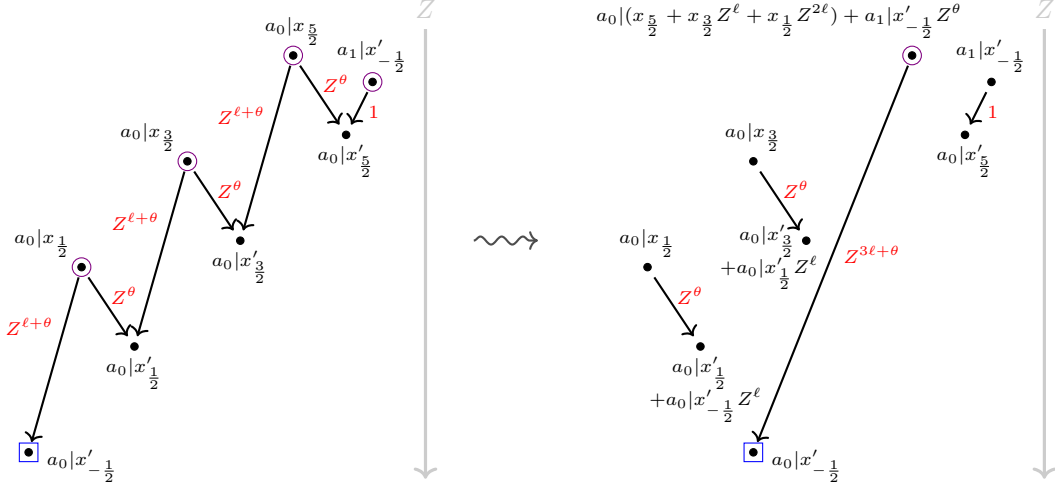
\begin{proof}[Proof of Proposition \ref{prop: non_ending_general_case}]
In the following proof,  we use $\partial^{\free}$ for the differential on $\bX^{\free}(P,K,\lambda)$ and $\partial$ for the differential on $\bX^{\diamond}(P,K,\lambda)$.
We proceed case by case.

\medskip
\noindent\textbf{Case \eqref{it: non_ending_general_case_1}: when $\eta_0, \eta_1>0$.} An example is depicted in Figure \ref{fig: prop_non_ending}\subref{subfig: non_ending_1}. 
Since this case does not involve $\bX^{\Tor}(P,K,\lambda)$,  the proof is comparatively simple.

We set \[\alpha_0 = a_0 | x'_{-\frac{1}{2}},  \qquad \alpha_1 =   a_0 \Big|  \sum_{i=1}^n x_{i - \frac{1}{2}}Z^{\ell (n-i)} +  a_1 | x'_{-\frac{1}{2}} Z^ \theta\] and  make the following claims:

\begin{enumerate}[label=(\alph*), ref=\alph*]
 \item \label{it: alpha0_cycle}$\partial \alpha_0 =0$;
    \item \label{it: Zn_claim}$\partial \alpha_1 =  \alpha_0 Z^{\ell n + \theta}$;
    \item \label{it: shortest_arrow_claim} $\{\alpha_0,\alpha_1\}$ generates a direct summand of the chain complex $\bX^{\diamond}(P,K,\lambda)$. 
\end{enumerate}

Claim \eqref{it: alpha0_cycle} follows from the fact that $\delta^1(a_0)=0$ in $\cX_{\lambda}(K)^{\hat\cK} $, and $m_{0|1|0}(x'_{-\frac{1}{2}})=0$.

Next, to prove Claim \eqref{it: Zn_claim}, we compute the differential of each term in $\alpha_1.$ 
 Since  $\delta^1(a_0)=0$, the only differentials each generator
 $ a_0 | x_{i-\frac{1}{2}}$ admits are $\Phi^{\mu}+\Phi^{-\mu} = \bI |( L_\sigma + L_\tau)$.      
 Recall from Lemma \ref{lem: C_S_structure_maps} that
\begin{equation}\label{eq: L_sigma+L_tau}
(L_\sigma + L_\tau) (x_{i-\frac{1}{2}})=\begin{cases}
    x'_{i-\frac{1}{2}} Z^{ \theta} +  x'_{i-\frac{3}{2}} Z^{\ell + \theta}  \qquad &i>0 \\
     \epsilon \cdot  \tilde{x}'_{i-\frac{1}{2}} Z^{\xi'_1 - \Xi} +  x'_{i-\frac{3}{2}} Z^{\kappa}  \qquad &i\leq 0
\end{cases}
\end{equation}
where $\epsilon\in \{0,1\}$ is nonzero if and only if $\Xi>0$.   Therefore
\begin{align}\label{eq: nonending_case_1_horizontal_diff}
    \partial (a_0 | x_{i-\frac{1}{2}}) &= a_0 | ( x'_{i-\frac{1}{2}} Z^{ \theta} + x'_{i-\frac{3}{2}} Z^{\ell + \theta} )   \qquad \text{for} \quad i \geq 1. 
    \end{align}
    
  Since $\delta^1(a_1)$ has coefficients in $\hat\cR$ (in particular, no components weighted by a non-zero multiple of $\sigma$ or $\tau$) \ and $m_{0|1|0}(x'_{-\frac{1}{2}})=0$,  $ a_1 | x'_{-\frac{1}{2}}$ admits only  length-$0$ differentials, which take the form of
  \[
       \begin{tikzcd}[column sep=0 cm, row sep=0.3 cm]
     \cCFK(K)^{\hat\cR}  \ar[d ]& \otimes& \scF^{\diamond}\ar[dd] \\
     \delta^1_{\hat\cR} \ar[dd]\ar[drr] \\
     & & m_{1|1|0}  \ar[d]  \\
     \cCFK(K)^{\hat\cR} &  \otimes  & \scF^{\diamond} 
	 \end{tikzcd}.
     \]   
    Since  $\delta^1(a_1) = a_0 Z^n$ and $m_{1|1|0}(Z^n,x'_{-\frac{1}{2}}) = x'_{n-\frac{1}{2}}$ we obtain
\begin{align*}    
    \partial  (a_1 | x'_{-\frac{1}{2}}) &= a_0 | x'_{n-\frac{1}{2}}. 
\end{align*} We  compute
\begin{align*}
    \partial \alpha_1 &=    \sum_{i=1}^n  \partial (a_0 | x_{i - \frac{1}{2}}) Z^{\ell (n-i)} + \partial ( a_1 | x'_{-\frac{1}{2}}) Z^ \theta\\
    &=     a_0 \Big|   \sum_{i=1}^n  \left( x'_{i - \frac{3}{2}}Z^{\ell (n-i)}\cdot Z^{\ell + \theta} +     x'_{i - \frac{1}{2}}Z^{\ell (n-i)}\cdot Z^{\theta}\right) +  a_0 |x'_{n-\frac{1}{2}}Z^ \theta\\
    &=  a_0 | x'_{-\frac{1}{2}}Z^{\ell n + \theta},
\end{align*}
which proves Claim \eqref{it: Zn_claim}.

To prove Claim \eqref{it: shortest_arrow_claim}, we perform a change of basis as follows (the changes in the second line apply only when $n>1$):  
\begin{align*}
a_0 | x_{n - \frac{1}{2}} &\mapsto \alpha_1 \\
   a_0 | x'_{i-\frac{1}{2}} &\mapsto a_0 | ( x'_{i-\frac{1}{2}} +   x'_{i-\frac{3}{2}} Z^{\ell } )    \qquad i= 1,\dots, n-1. 
\end{align*}
The remainder of the basis is unchanged. A schematic illustration of this change of basis is depicted in Figure \ref{fig: Zzigzag}. 

By Lemma \ref{lem: horizontal_chain}, in the resulting basis of $\bX^{\diamond}(P,K,\lambda)_{\bF[Z]}$, for $i=1,\dots, n-1$, the only arrow pointing to $a_0 |( x'_{i-\frac{1}{2}} +   x'_{i-\frac{3}{2}} Z^{\ell })$ comes from $a_0 | x_{i-\frac{1}{2}}$.
Since $\partial^2 =0$ and each $a_0 | x'_{i-\frac{1}{2}}$ is non-torsion,  there is no arrow pointing to $a_0 | x_{i-\frac{1}{2}}$ for $i=1,\dots,n-1$. Therefore, $\{a_0 | x_{i-\frac{1}{2}},a_0 | ( x'_{i-\frac{1}{2}} +   x'_{i-\frac{3}{2}} Z^{\ell }):i=1,\dots,n-1 \}$  generates a direct summand in $\bX^{\diamond}(P,K,\lambda)$. 
In the resulting basis,
the only arrow pointing to  $\alpha_0$ comes from $\alpha_1.$
Since $\partial^2 =0$ and  $\alpha_0$ is non-torsion,
it follows that  $\{\alpha_0,\alpha_1\}$ generates a direct summand in $\bX^{\diamond}(P,K,\lambda)$.

From the claims \eqref{it: alpha0_cycle}, \eqref{it: Zn_claim} and \eqref{it: shortest_arrow_claim}, we conclude that  $\Ord(P(K)) \geq \ell n + \theta.$

\medskip
\noindent\textbf{Case \eqref{it: non_ending_general_case_2}: when $\eta_0<0, \eta_1>0$.} An example is depicted in Figure \ref{fig: prop_non_ending}\subref{subfig: non_ending_2}. Since we have a sequence of consecutive arrows in one direction, by Lemma \ref{lem: snake_structure}, we may assume  $n > 1.$ 
\ \\
Compared to the previous case,  $\delta^1(a_0) $ now contains an additional term  $ a'_0 W^{-\eta_0}$. As a result,
each generator $a_0 | x_s$ and $a_0 | x'_s$ may additionally admit a length-$0$ differential of the form $a'_0 | m_{1|1|0}(W^{-\eta_0}, x_s)$ and $a'_0 | m_{1|1|0}(W^{-\eta_0}, x'_s)$, respectively. Since
\[m_{1|1|0}(W^{-\eta_0}, x_{i+\frac{1}{2}})= m_{1|1|0}(W^{-\eta_0}, x'_{i-\frac{1}{2}})=0\] 
for $i>0$ by \eqref{eq: scE_scF_actions},
$\partial (a_0 | x_{i+\frac{1}{2}})$ and $\partial (a_0 | x_{i-\frac{1}{2}})$ are unchanged for $i>0.$ 

When $i=0$,  we have $L_W(x_{\frac{1}{2}}) = x_{-\frac{1}{2}} Z^\omega$ by Lemma \ref{lem: C_S_structure_maps}. Again, by \eqref{eq: scE_scF_actions}, we have
\[m_{1|1|0}(W^{-\eta_0}, x_{\frac{1}{2}}) = m_{1|1|0}(W^{-\eta_0 -1}, L_W(x_{\frac{1}{2}})) = x_{\eta_0 + \frac{1}{2}} Z^\omega.\] 
Therefore
\begin{align} \label{eq: nonending_case2_leftmost}
    \partial (a_0 | x_{\frac{1}{2}}) &=  a_0 | x'_{\frac{1}{2}} Z^{ \theta} + a'_0 | x_{\eta_0 + \frac{1}{2}} Z^\omega + a_0 | x'_{-\frac{1}{2}} Z^{\ell + \theta}  \notag \\
&= a_0 | x'_{\frac{1}{2}} Z^{ \theta} + \left( a'_0|x_{\frac{1}{2}+\eta_0} +  a_0 | x'_{-\frac{1}{2}}Z^{\kappa}\right) \cdot Z^\omega 
\intertext{and}
\partial (a_0 | x'_{-\frac{1}{2}}) &= a'_0 | x'_{\eta_0 -\frac{1}{2}}. \notag
\end{align}
We consider two cases depending on the sign of $\omega - \theta =  \Xi +1$, as follows.
\begin{enumerate}[label=\roman*$)$]
    \item \label{it: eta0<0_eta_1>0_case_1} 
    \textbf{Suppose $\omega - \theta \leq 0.$} In this case, 
    we set 
  \[\alpha_0 = a_0 | x'_{\frac{1}{2}}, \qquad  \alpha_1 =    a_0 \Big|\sum_{i=1}^{n-1}  x_{i + \frac{1}{2}}Z^{\ell (n-1-i)} + a_1 | x'_{-\frac{1}{2}}Z^ \theta.\] 
  This  differs from what is depicted in Figure~\ref{fig: prop_non_ending}\subref{subfig: non_ending_2}. Instead, it is 
  essentially the same setting as in Case \eqref{it: non_ending_general_case_1}. We claim:
  \begin{enumerate}[label=(\alph*), ref=\alph*]
 \item $\partial \alpha_0 =0$;
    \item $\partial \alpha_1 =  \alpha_0 Z^{\ell (n-1) + \theta}$;
    \item \label{it: alpha0alpha1directsummand} $\{\alpha_0,\alpha_1\}$ generates a direct summand in $\bX^{\diamond}(P,K,\lambda)$. 
\end{enumerate} 
We leave the first two claims to the reader and prove
\eqref{it: alpha0alpha1directsummand} by a change of basis, as follows:
\begin{align*}
a_0 | x_{n - \frac{1}{2}} &\mapsto \alpha_1 \\
   a_0 | x'_{i-\frac{1}{2}} &\mapsto a_0 |( x'_{i-\frac{1}{2}} +   x'_{i-\frac{3}{2}} Z^{\ell } )    \qquad i= 2,\dots, n-1 \\
   a'_0|x_{\frac{1}{2}+\eta_0} &\mapsto  a'_0|x_{\frac{1}{2}+\eta_0} + a_0 | x'_{\frac{1}{2}} Z^{ \theta - \omega} +   a_0 | x'_{-\frac{1}{2}}Z^{\kappa}.
\end{align*}
The remainder of the basis is unchanged. In the resulting basis, the only arrow pointing to $\alpha_0$ is from $\alpha_1.$
    We conclude that  $\Ord(P(K)) \geq \ell (n-1) + \theta.$
    \item 
    \label{it: eta0<0_eta_1>0_case_2} 
   \textbf{Suppose $\omega - \theta > 0.$}
Set 
    \[\alpha_0 = a'_0|x_{\frac{1}{2}+\eta_0} +  a_0 | x'_{-\frac{1}{2}}Z^{\kappa}, \qquad \alpha_1 =     a_0 \Big| \sum_{i=1}^n x_{i - \frac{1}{2}}Z^{\ell (n-i)} + a_1 | x'_{-\frac{1}{2}}Z^ \theta\] and we make the following claims:
\begin{enumerate}[label=(\alph*), ref=\alph*]
 \item \label{it: alpha0_cycle_case_2}  $\partial \alpha_0 Z^{\Xi} = 0$;
    \item \label{it: Zn_claim_case_2} $\partial \alpha_1 =  \alpha_0 Z^{\ell (n-1) + \omega}$;
    \item \label{it: shortest_arrow_claim_case_2} $\{ \alpha_0, \alpha_1 \}$ generates a summand in $\bX^{\free}(P,K,\lambda)$. 
\end{enumerate}     

    We first prove Claim \eqref{it: alpha0_cycle_case_2} by considering all possible outgoing arrow from $ a'_0|x_{\frac{1}{2}+\eta_0}$.
    An outgoing arrow of $\Phi^{\pm K}$ requires that
  $a'_0$ be the starting or ending generator of the standard complex (note that $a'_0$ cannot be an ending generator because it is adjacent to a $W$-arrow). Since there are no outgoing $W$ arrows from $a_0'$, an outgoing length-$0$ differential would require that 
  the next arrow weighted by $Z$-powers be outgoing from $a'_0$ in $\cCFK(K)^{\hat\cR}$. In either case,
   by Lemma \ref{lem: snake_structure}, we have $\eta_0 \leq -2.$  Hence $\frac{1}{2}+\eta_0 < -1$, 
   and by \eqref{eq: scE_scJ_actions}, we have \[ m_{1|1|0}(Z,x_{\frac{1}{2}+\eta_0}) = m_{1|1|0}(\sigma,x_{\frac{1}{2}+\eta_0}) = 0. \]
   Therefore
   any length-$0$ differential and $\Phi^{\pm K}$ from $a'_0|x_{\frac{1}{2}+\eta_0}$ must vanish.  It follows that the only outgoing arrows from $a'_0|x_{\frac{1}{2}+\eta_0}$ are of the form $\bI | (L_\sigma + L_\tau)$, which were computed in \eqref{eq: L_sigma+L_tau}.  Suppose $\Xi = 0,$ it is straightforward to verify that $\partial \alpha_0 = 0$ and the claim holds.  If $\Xi > 0$, we compute
   \begin{align*}
       \partial \alpha_0 &= \partial ( a'_0|x_{\frac{1}{2}+\eta_0}) +  \partial (a_0 | x'_{-\frac{1}{2}}) Z^{\kappa}\\
       &=    a'_0|\tilde{x}'_{\frac{1}{2}  +\eta_0} Z^{\xi'_1 - \Xi}  +   a'_0|x'_{-\frac{1}{2}+\eta_0} Z^{\kappa}     +   a'_0|x'_{-\frac{1}{2}+\eta_0} Z^{\kappa}\\
       &=  a'_0|\tilde{x}'_{\frac{1}{2}  +\eta_0} Z^{\xi'_1 - \Xi}.
   \end{align*}
   Since $ a'_0|\tilde{x}'_{\frac{1}{2}+\eta_0} Z^{\xi'_1} = 0$ in $\cS^{\diamond}$,   Claim \eqref{it: alpha0_cycle_case_2} follows. 
   
We leave the verification of Claim \eqref{it: Zn_claim_case_2}  to the reader, which is similar to the previous case.

To prove Claim \eqref{it: shortest_arrow_claim_case_2},  we first observe that $a'_0|x_{\frac{1}{2}+\eta_0}$ admits no other incoming arrows than the one from $a_0|x_{\frac{1}{2}+\eta_0}$. Indeed, any such arrow must be a length-$0$ differential, which requires an incoming arrow to $a'_0$ weighted by $Z^j$ for some $j>0$ in $\cCFK(K)^{\hat \cR}$. Since $\frac{1}{2}+\eta_0-j < -1,$ by \eqref{eq: scE_scF_actions}, $m_{1|1|0}(Z^j, x_{\frac{1}{2}+\eta_0-j})=0$. Therefore any such arrow vanishes.

We perform a change of basis as follows:
\begin{align*}
    a'_0|x_{\frac{1}{2}+\eta_0} &\mapsto \alpha_0  \hspace{7em} a_0 | x_{n - \frac{1}{2}} \mapsto \alpha_1 \\
     a_0 | x'_{\frac{1}{2}} &\mapsto a_0 | x'_{\frac{1}{2}} + \alpha_0 Z^{\omega - \theta}\\
    a_0 | x'_{i-\frac{1}{2}} &\mapsto a_0 | (x'_{i-\frac{1}{2}} +  x'_{i-\frac{3}{2}} Z^{\ell })     \qquad i= 2,\dots, n-1.
\end{align*}
The remainder of the basis is unchanged. 

In the resulting basis,
by the observation earlier together with Lemma \ref{lem: horizontal_chain},  the only arrow pointing to $a_0 | x'_{\frac{1}{2}} + \alpha_0 Z^{\omega - \theta}$ comes from $a_0 | x_{\frac{1}{2}}$. It follows that $\{a_0 | x_{\frac{1}{2}}, a_0 | x'_{\frac{1}{2}} + \alpha_0 Z^{\omega - \theta}\}$ generates a direct summand in $\bX^{\free}(P,K,\lambda)$.
Also, similarly as before, for $i=2,\dots, n-1$, $\{a_0 | x_{i-\frac{1}{2}},a_0 |( x'_{i-\frac{1}{2}} +  x'_{i-\frac{3}{2}} Z^{\ell })\}$ generates a direct summand in $\bX^{\diamond}(P,K,\lambda)$.

In the resulting basis, the only arrow pointing to $\alpha_0$ comes from $\alpha_1.$
It follows that $\{ \alpha_0, \alpha_1 \}$ generates a summand in $\bX^{\free}(P,K,\lambda)$.

 Having proved all three claims, we now apply Lemma \ref{lem: torsion_order_criterion} to compute the torsion order. Letting $A= \bX^{\free}(P,K,\lambda), B= \bX^{\Tor}(P,K,\lambda)$ and $f=\partial$, we conclude that $\Ord(P(K)) \geq \ell(n-1)+ \omega - \Xi = \ell(n-1)+ \theta + 1.$ 
\end{enumerate}

\medskip
\noindent\textbf{Case \eqref{it: non_ending_general_case_3}: when $\eta_0>0, \eta_1<0$.} An example is depicted in Figure \ref{fig: prop_non_ending}\subref{subfig: non_ending_3}. Compared to Case \eqref{it: non_ending_general_case_1}, 
 $\delta^1(a_1) $ now contains an additional term  $ a'_1 W^{-\eta_1}$. As a result,
each generator $a_1 | x_s$ and $a_1 | x'_s$ may additionally admit a length-$0$ differential of the form $a'_1 | m_{1|1|0}(W^{-\eta_1}, x_s)$ and $a'_1 | m_{1|1|0}(W^{-\eta_1}, x'_s)$, respectively. 
When $i > 0,$ by \eqref{eq: scE_scF_actions}, we have
\[ m_{1|1|0}(W^{-\eta_1}, x_{i+\frac{1}{2}}) =  m_{1|1|0}(W^{-\eta_1}, x'_{i-\frac{1}{2}}) = 0. \] 
Therefore $\partial (a_1 | x_{i+\frac{1}{2}})$ and $\partial (a_1 | x'_{i-\frac{1}{2}})$ are unchanged for $i<0.$
When $i=0,$ we compute
\begin{align*}
\partial (a_1 | x_{\frac{1}{2}} )&=   a'_1 | x_{\eta_1 + \frac{1}{2}} Z^{\omega} +  a_1 | x'_{- \frac{1}{2}} Z^{\ell + \theta} \\
\partial (a_1 | x'_{- \frac{1}{2}}) &=      a'_1 | x'_{\eta_1 - \frac{1}{2}} + a_0 | x'_{n- \frac{1}{2}}.
\intertext{We also have}
\partial^{\free} (a'_1 | x_{\eta_1 + \frac{1}{2}}) &=   a'_1 | x'_{\eta_1 - \frac{1}{2}}  Z^{\kappa}.
\end{align*} 
Set   $\alpha_0 = a_0 | x'_{-\frac{1}{2}}$ and
\begin{align*}
     \alpha_1 &= \left(    a_0 \Big| \sum_{i=1}^n    x_{i - \frac{1}{2}} Z^{\ell (n-i)} + a_1 | x'_{-\frac{1}{2}}  Z^ \theta \right) Z^{\max\{\kappa - \theta, 0 \}} +  a'_1|x_{\eta_1+\frac{1}{2}}Z^{\max\{\theta - \kappa, 0 \}} 
\end{align*}
 where the terms in parentheses correspond to $\alpha_1$
 in Case \eqref{it: non_ending_braided_case_1}; the extra term at the end stems from the additional length-$0$ differential, and the power of $Z$ arises from straightforward algebraic reasons.
   We claim:
\begin{enumerate}[label=(\alph*), ref=\alph*]
 \item \label{it: alpha0_cycle_case_3}  $\partial \alpha_0  = 0$;
    \item \label{it: Zn_claim_case_3} $\partial^{\free} \alpha_1 =  \alpha_0 Z^{\ell n + \max\{\kappa, \theta \}}$;
    \item \label{it: shortest_arrow_claim_case_3} $\{ \alpha_0, \alpha_1 \}$ generates a summand in $\bX^{\free}(P,K,\lambda)$. 
\end{enumerate}  
Claim \eqref{it: alpha0_cycle_case_3} is immediate. For Claim \eqref{it: Zn_claim_case_3}, we compute
\begin{align*}
    \partial^{\free} \alpha_1 &= \left( \alpha_0 Z^{\ell n + \theta} + a'_1 | x'_{\eta_1 - \frac{1}{2}}  Z^{\theta}  \right) Z^{\max\{\kappa - \theta, 0 \}} + a'_1 | x'_{\eta_1 - \frac{1}{2}}  Z^{\kappa} \cdot Z^{\max\{\theta - \kappa, 0 \}} \\
    &= \alpha_0 Z^{\ell n + \max\{\kappa, \theta \}}.
\end{align*}
In order to prove Claim \eqref{it: shortest_arrow_claim_case_3}, we divide the argument into the following two cases depending on the sign of $\kappa - \theta$. We demonstrate the case when $\kappa - \theta \geq 0$; the case when $\kappa - \theta < 0$ is similar and left to the reader.

Perform the change of basis as follows:
\begin{align*}
    a'_1 | x_{\eta_1 + \frac{1}{2}} &\mapsto \alpha_1 \\
    a'_1 | x'_{\eta_1 - \frac{1}{2}} &\mapsto a'_1 | x'_{\eta_1 - \frac{1}{2}} + a_0 | x'_{n - \frac{1}{2}}\\
    a_0 | x'_{i-\frac{1}{2}} &\mapsto a_0 | (x'_{i-\frac{1}{2}} +  x'_{i-\frac{3}{2}} Z^{\ell })     \qquad i= 1,\dots, n.
\end{align*}
The remainder of the basis is unchanged. In the resulting basis of $\bX^{\diamond}(P,K,\lambda)$,  the only arrow  pointing to $\alpha_0$ comes from $\alpha_1.$ It follows that $\{ \alpha_0, \alpha_1 \}$ generates a summand in $\bX^{\free}(P,K,\lambda)$.

Having proved all the claims, by Lemma \ref{lem: torsion_order_criterion}, we conclude that $\Ord(P(K)) \geq \ell n + \max\{\kappa,\theta \}$.

\medskip
\noindent\textbf{Case \eqref{it: non_ending_general_case_4}: when $\eta_0, \eta_1<0$ and $n>1$.} An example is depicted in Figure \ref{fig: prop_non_ending}\subref{subfig: non_ending_4}.
\ \\
The proof follows from  a combination of the arguments in Case \eqref{it: non_ending_general_case_2} and \eqref{it: non_ending_general_case_3}.
We sketch out the main steps, leaving the details to the reader. As in Case \eqref{it: non_ending_general_case_2},
there are two cases depending on the sign of $\omega - \theta =  \Xi + 1$  as follows. 
\begin{enumerate}[label=\roman*$)$]
\item  \label{it: eta0<0_eta_1<0_case_1} \textbf{Suppose $\omega - \theta  \leq 0$}. This is the case where the assumption $n>1$ is necessary.  
Set 
\begin{align*}
\alpha_0 &= a_0|x'_{\frac{1}{2}},\\
\alpha_1 &= \left( a_0 \Big| \sum_{i=1}^{n-1}  x_{i + \frac{1}{2}} Z^{\ell (n-1-i)} + a_1|x'_{-\frac{1}{2}} Z^ \theta  \right) Z^{\max \{\kappa - \theta, 0 \} } + a'_1|x_{\eta_1+\frac{1}{2}} Z^{\max \{ \theta - \kappa, 0 \}}.
\end{align*}
Again this is different from what is depicted in Figure \ref{fig: prop_non_ending}\subref{subfig: non_ending_4} by omitting the leftmost portion. We have
\begin{enumerate}[label=(\alph*), ref=\alph*]
\item $\partial \alpha_0 = 0$;
    \item  $\partial^{\free} \alpha_1 =  \alpha_0 Z^{\ell (n-1) + \max\{\kappa, \theta \}}$;
    \item  $\{ \alpha_0, \alpha_1 \}$ generates a summand in $\bX^{\free}(P,K,\lambda)$ 
\end{enumerate}  
and by Lemma \ref{lem: torsion_order_criterion}, we conclude that $\Ord(P(K)) \geq \ell (n-1) + \max\{\kappa,\theta \}.$ 
\item  \label{it: eta0<0_eta_1<0_case_2} \textbf{Suppose $\omega - \theta > 0$.} Set
\begin{align*}
    \alpha_0 &= a'_0|x_{\eta_0 + \frac{1}{2}} +  a_0|x'_{-\frac{1}{2}}Z^\kappa, \\
    \alpha_1 &=  \left( a_0 \Big| \sum_{i=1}^{n}  x_{i - \frac{1}{2}} Z^{\ell (n-i)} + a_1|x'_{-\frac{1}{2}} Z^ \theta \right) Z^{\max \{\kappa-\theta, 0 \}} + a'_1|x_{\eta_1+\frac{1}{2}} Z^{\max \{\theta - \kappa, 0 \}}.
\end{align*}
Then we have
\begin{enumerate}[label=(\alph*), ref=\alph*]
\item $\partial \alpha_0 Z^{\Xi} = 0$;
    \item  $\partial^{\free} \alpha_1 =  \alpha_0 Z^{\ell (n-1) + \omega - \theta + \max\{\kappa, \theta \}}$;
    \item  $\{ \alpha_0, \alpha_1 \}$ generates a summand in $\bX^{\free}(P,K,\lambda)$. 
\end{enumerate}  
By Lemma \ref{lem: torsion_order_criterion}, we conclude that $\Ord(P(K)) \geq \ell (n-1) + \omega - \theta + \max\{\kappa,\theta \} - \Xi = \ell (n-1) + \max\{\kappa,\theta \} + 1.$ 

\end{enumerate}
\end{proof}
For braided L-space satellite patterns, the bounds in Proposition \ref{prop: non_ending_general_case} can be improved. 
The following proposition does not fully recover the bounds in the first half of \cite[Proposition 3.1]{HLPUnknotting} in Case \eqref{it: non_ending_braided_case_3} and \eqref{it: non_ending_braided_case_4}. Nevertheless, it suffices for the proof of the unknotting number bound on the iterated braided L-space satellite patterns stated in Theorem \ref{thm: iterate_braided}.
\begin{prop}\label{prop: non_ending_braided_case}
    Suppose there is a non-ending $Z^n$-arrow  in $\cCFK(K)$ for some positive integer $n$ and $P$ is a braided L-space satellite pattern.  
    \begin{enumerate} 
        \item \label{it: non_ending_braided_case_1} Type $(+,+)$: $\Ord(P(K)) \geq \ell n $;
        \item \label{it: non_ending_braided_case_2} Type $(-,+)$: $\Ord(P(K)) \geq \ell (n-1) + 1$;
         \item \label{it: non_ending_braided_case_3} Type $(+,-)$: $\Ord(P(K)) \geq \ell (n+1) - \omega  $;
         \item \label{it: non_ending_braided_case_4} Type $(-,-)$: $\Ord(P(K)) \geq \ell n -\omega + 1$
    \end{enumerate}
    where $\omega = R_{\frac{\ell}{2}} - R_{\frac{\ell}{2}-1}. $
\end{prop}
\begin{proof}
For braided L-space patterns, $\theta=0, \Xi \geq 0$ and $\kappa = \ell - \omega + \theta = \ell - \omega$ by Definition \ref{def: shorthand} and Lemma \ref{lem: braided_property}.
The bounds  in Case \eqref{it: non_ending_general_case_1} and \eqref{it: non_ending_general_case_3} above directly follow from their counterparts in Proposition \ref{prop: non_ending_general_case}. We focus on the remaining cases.
    Compared to the proof of Proposition \ref{prop: non_ending_general_case},  certain subcases can now be excluded.

      Case \eqref{it: non_ending_general_case_2}: since $\Xi \geq 0,$   only subcase \ref{it: eta0<0_eta_1>0_case_2} is possible. We obtain $\Ord(P(K)) \geq \ell (n-1)  +1.$

Case \eqref{it: non_ending_general_case_4}: similar as before, only subcase \ref{it: eta0<0_eta_1<0_case_2} is possible. In particular, the assumption $n>1$ is no longer necessary.
      We obtain $\Ord(P(K)) \geq \ell (n-1)  +\kappa+1 = \ell n-\omega+1.$
\end{proof}

\subsection{Ending arrows}
Next, suppose that the ending arrow of the standard complex of $\cCFK(K)$ is weighted by $Z^n$ for some positive integer $n$. Note that this implies $\varepsilon(K) \neq 0$. Fixing a framing $\lambda \in \Z,$ we consider the satellite knot $P(K,\lambda).$  

In the case of cables \cite{HLPUnknotting}, using the fact that $u(-P(K))=u(P(K))$ and the fact that the set of cable patterns is closed under mirroring, it suffices to consider only the case that $\varepsilon (K) = 1$. However,  the set of L-space satellite operators is not closed under mirroring, therefore we need to consider the case when $\varepsilon(K)=1$ as well as when $\varepsilon(K)=-1$. 

When $\varepsilon(K) =1$, by definition there is an arrow weighted by $Z^n$ from $a_1$ to $a_0$, where $a_0$ is the ending generator of the standard complex. By the symmetry of knot Floer complex, we also have generators $a_{m-1}$ and $a_m$ in the standard complex, where $a_m$ is the starting generator, and
an arrow weighted by $W^n$ from $a_{m-1}$ to $a_m$. Denote by $a'_1$ the generator that precedes $a_1$ in the standard complex.
As before, we can designate a nonzero integer $\eta$, such that there is an arrow weighted by $W^{|\eta|}$ from $a'_1$ to $a_1$ if $\eta>0$ and from $a_1$ to $a'_1$ if $\eta<0.$
 When $\eta=-n$, it is possible that $a_m = a'_1$ and $a_{m-1} = a_1.$  
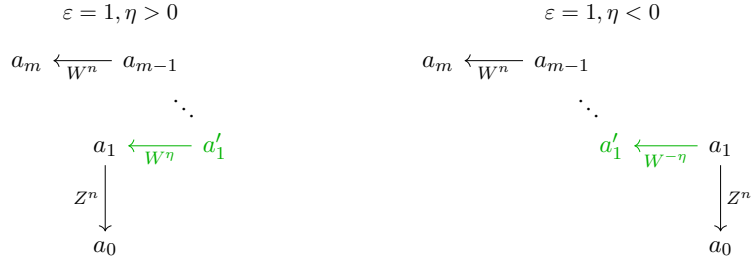
\begin{figure}[hbtp!]\captionsetup{width=\textwidth}
\begin{tikzpicture}[scale=0.85, transform shape]
 \node at (-4,2.8) {{\small$\varepsilon = 1,\eta>0$}};
 \node at (-4.4,2) [inner sep=0pt] {
 \begin{tikzcd}
    a_m & a_{m-1} \ar[l, "W^n"]
\end{tikzcd}
 };
 \node[rotate=-45] at (-3,1.3) {$\cdots$};
    \node at (-3.5,0) [inner sep=0pt] {
\begin{tikzcd}[column sep=1 cm, row sep=1 cm]
a_1 \ar[d, "Z^n"']   &  {\color{green!70!black} a'_1} \ar[l, color = green!70!black,  "W^{\eta}"]
\\
a_0 &
\end{tikzcd}
};
  \node at (3.5,2.8) {{\small$\varepsilon = 1,\eta<0$}};
 \node at (2,2) [inner sep=0pt] {
 \begin{tikzcd}
    a_m & a_{m-1} \ar[l, "W^n"]
\end{tikzcd}
 };
 \node[rotate=-45] at (3.3,1.3) {$\cdots$};
    \node at (4.6,0) [inner sep=0pt] {
\begin{tikzcd}[column sep=1 cm, row sep=1 cm]
{\color{green!70!black} a'_1}  
& a_1 \ar[d, "Z^n"] \ar[l, color = green!70!black,   "W^{-\eta}"] 
\\
&a_0 
\end{tikzcd}
};
\end{tikzpicture}
\end{figure}

When $\varepsilon(K) = -1$, there is an arrow weighted by $Z^n$ from $a_0$ to $a_1$, where $a_0$ is the ending generator of the standard complex (and by the symmetry an arrow weighted by $W^n$ from $a_m$ to $a_{m-1}$).  Define  the nonzero integer $\eta$ such that there is an arrow weighted by $W^{|\eta|}$ from $a'_1$ to $a_1$ if $\eta>0$ and from $a_1$ to $a'_1$ if $\eta<0.$ When $\eta=n$, it is possible that $a_m = a'_1$ and $a_{m-1} = a_1.$ 

\begin{figure}[hbtp!]\captionsetup{width=\textwidth}
\begin{tikzpicture}[scale=0.85, transform shape]
 \node at (-3.5,3) {{\small$\varepsilon = -1,\eta>0$}};
 \node at (-4.5,1.5) [inner sep=0pt] {
\begin{tikzcd}[column sep=1 cm, row sep=1 cm]
 a_0 \ar[d, "Z^n"]   &  
\\
   a_1  & {\color{green!70!black} a'_1} \ar[l, color = green!70!black,  "W^{\eta}"']
\end{tikzcd}
 };
 \node[rotate=-45] at (-3.3,0) {$\cdots$};
    \node at (-2,-0.5) [inner sep=0pt] { \begin{tikzcd}
     a_{m-1}  & a_{m} \ar[l, "W^n"']
\end{tikzcd}
};
  \node at (3.8,3) {{\small$\varepsilon = -1,\eta<0$}};
 \node at (3,1.5) [inner sep=0pt] {
 \begin{tikzcd}[column sep=1 cm, row sep=1 cm]
& a_0 \ar[d, "Z^n"]    
\\
{\color{green!70!black} a'_1}   &a_1 \ar[l, color = green!70!black,   "W^{-\eta}"'] 
\end{tikzcd}
 };
 \node[rotate=-45] at (2.5,0) {$\cdots$};
    \node at (4.2,-0.5) [inner sep=0pt] {\begin{tikzcd}
    a_{m-1}  & a_{m} \ar[l, "W^n"']
\end{tikzcd}
};
\end{tikzpicture}
\end{figure} 

We are now ready to state our result regarding $\Ord(P(K, \lambda))$ in the case of ending arrows.
\begin{prop}\label{prop: ending_arrow_general_case}
       Suppose the ending arrow of the standard complex in $\cCFK(K)$ is weighted by $Z^n$ for some $n>0$, and $\eta$ as defined above. Let $P$ be an L-space satellite pattern and $\lambda \in \Z$. 
\begin{itemize}
    \item 
       Suppose $\varepsilon (K) = 1.$
       \medskip
    \begin{enumerate}
        \item \label{it: ending_general_epsilon>0_case_1}
        If $\eta>0$, then 
        \begin{align*}
            \Ord(P(K, \lambda)) \geq \begin{cases}
            \ell (n -1) + \theta   \quad &\text{if} \quad  \lambda - 2\tau(K) < 0 \\
             \ell n  + \theta   \quad &\text{if} \quad  \lambda - 2\tau(K) \geq 0;
            \end{cases}
        \end{align*}
        \item \label{it: ending_general_epsilon>0_case_2} If $\eta<0$, 
        \begin{align*}
            \Ord(P(K, \lambda)) \geq \begin{cases}
            \ell (n-1) + \max\{\kappa, \theta \}   \quad &\text{if} \quad  \lambda - 2\tau(K) <0 \quad \text{and} \quad n>1\\
             \ell n  + \max\{\kappa, \theta \}   \quad &\text{if} \quad  \lambda - 2\tau(K) \geq 0.
            \end{cases}
            \end{align*}
    \end{enumerate}
    \item Suppose $\varepsilon (K) = -1.$
    \medskip
    \begin{enumerate}[start=3]
        \item \label{it: ending_general_epsilon<0_eta>0} If $\eta>0$, then
        \[
            \Ord(P(K, \lambda)) \geq          
             \begin{cases}
            \ell n + \max\{\kappa, \theta \}   \quad &\text{if} \quad  \lambda - 2\tau(K) < 0 \\
             \ell n  + \theta   \quad &\text{if} \quad  \lambda - 2\tau(K) \geq 0;
            \end{cases}
            \]
        \item \label{it: ending_general_epsilon<0_eta<0} If $\eta<0$, then
        \[
            \Ord(P(K, \lambda)) \geq    
            \begin{cases}
            \ell (n-1) + \max\{\kappa, \theta \}   \quad &\text{if} \quad  \lambda - 2\tau(K) < 0 \\
             \ell (n-1)  + \theta   \quad &\text{if} \quad  \lambda - 2\tau(K) \geq 0.
            \end{cases} 
            \]
    \end{enumerate}
    \end{itemize}
\end{prop}
\begin{rem}\label{rem:exclude-T23} The case where $\veps(K)=1$, $\eta<0$ does not completely cover the case that $n=1$. The only case where we cannot apply either Proposition~\ref{prop: ending_arrow_general_case} or Proposition~\ref{prop: non_ending_general_case} is where the longest arrow in $\cCFK(K)^{\hat{\cR}}$ has length 1 (i.e. $n=1$) and $\cCFK(K)^{\hat{\cR}}$ has no local system complexes, i.e. $\cCFK(K)^{\hat{\cR}}$ coincides with $\cCFK(T_{2,3})^{\hat{\cR}}$. This case is studied in Theorem~\ref{thm: tor_ord_T_23}. 
\end{rem}
The proof follows the same strategy as before. We construct a pair of free generators $\alpha_0$ and $\alpha_1$ in $\bX^{\diamond}(P,K,\lambda)_{\bF[Z]}$ together with an arrow of the desired length from $\alpha_1$ to $\alpha_0$, then apply Lemma \ref{lem: torsion_order_criterion} to bound the torsion order.
  
The generators  $\alpha_0$ and $\alpha_1$
will be constructed in the following portion of the complex:
\begin{equation} \label{eq: ending_EFJM}
\begin{tikzcd}[column sep=1.3cm, row sep=0.8cm]
E^{\diamond}_{*,\frac{\ell-3}{2}}
  \ar[r,"\Phi^{\pm\mu}"{description}]
  \ar[d,"\Phi^{K}"{description}]
&
F^{\diamond}_{*,\frac{\ell-3}{2}}
  \ar[d,"\Phi^{K}"{description}]\\
J^{\diamond}_{*,\frac{\ell-3}{2}}
  \ar[r,"\Phi^{\pm\mu}"{description}]
&
M^{\diamond}_{*,\frac{\ell-3}{2}}\\
E^{\diamond}_{*,\frac{\ell-1}{2}}
  \ar[r,"\Phi^{\pm\mu}"{description}]
  \ar[u,"\Phi^{-K}"{description}]
&
F^{\diamond}_{*,\frac{\ell-1}{2}}
  \ar[u,"\Phi^{-K}"{description}]
\end{tikzcd}
\end{equation}
where each of these complexes is obtained by tensoring $\cX_{\lambda}(K)^{\hat{\cK}}$ with the corresponding  $\scE^{\diamond}_{*,*},\scF^{\diamond}_{*,*}, \scJ^{\diamond}_{*,*}$ and $\scM^{\diamond}_{*,*}$, respectively.

We now extend the shorthand from Definition~\ref{def: xs-shorthand}.
Recall that
  \begin{equation}\label{eq: scE-s-def2}
      \scE^{\diamond}_{s,\frac{\ell-3}{2}} = \begin{cases}
          \cC^{\diamond}_{\frac{\ell}{2}-2} \quad &s<0 \\
          \cC^{\diamond}_{\frac{\ell}{2}-1} \quad &s>0
      \end{cases},
  \end{equation} while  each $\scJ^{\diamond}_{s,\frac{\ell-3}{2}}$ is a copy of $\cC^{\diamond}_{\frac{\ell}{2}-1}$, and each $\scF^{\diamond}_{s,\frac{\ell-3}{2}}$ or  $\scM^{\diamond}_{s,\frac{\ell-3}{2}}$ a copy of $\cS^{\diamond}$. 
 \begin{define} \label{def: ys-shorthand}
  We denote the free generator  $\xs^t_0 \in \scE^{\diamond}_{s,\frac{\ell-3}{2}}$ by $y_s$, where $t = \frac{\ell}{2}-2$ or $\frac{\ell}{2}-1$ depending on $s$ as in Equation~\eqref{eq: scE-s-def2}. Similarly, denote the  generator  $\xs^t_1 \in \scE^{\diamond}_{s,\frac{\ell-3}{2}}$ by $\tilde{y}_s$ if it exists.
  
  Denote the free generator  $\xs'_0 \in \scF^{\diamond}_{s,\frac{\ell-3}{2}} = \cS^{\diamond}$ by $y'_s$ and the generator $\xs'_1 \in \scF^{\diamond}_{s,\frac{\ell-3}{2}} = \cS^{\diamond}$ by $\tilde{y}'_s$ if it exists.
  
  Denote the free generator  $\xs^{\ell/2-1}_0 \in \scJ^{\diamond}_{s,\frac{\ell-3}{2}}=\cC^{\diamond}_{\frac{\ell}{2}-1}$ by $w_s$ and
  the generator $\xs^{\ell/2-1}_1 \in \scJ^{\diamond}_{s,\frac{\ell-3}{2}} = \cC^{\diamond}_{\frac{\ell}{2}-1}$ by $\tilde{w}_s$ if it exists. 
  
  Denote the free generator  $\xs'_0 \in \scM^{\diamond}_{s,\frac{\ell-3}{2}} = \cS^{\diamond}$ by $w'_s$ and the generator $\xs'_1 \in \scM^{\diamond}_{s,\frac{\ell-3}{2}} = \cS^{\diamond}$ by $\tilde{w}'_s$ if it exists. 
 \end{define} 

The following diagram describes the maps $f^{\pm \mu}$ and $f^{\pm K}$ restricted to the free generators in this portion of the complex.
It is obtained by combining  \eqref{eq: scE_scJ_actions} and the proof of Lemma \ref{lem: C_S_structure_maps}. The maps $f^{\pm \mu}$ are depicted roughly horizontally, the maps $f^{K}$ point vertically downwards and the maps $f^{-K}$ point vertically upwards. 
\begin{equation}\label{eq: f_mu_and_K_EFJM}
\begin{tikzcd}[labels=description,column sep=0.9cm, row sep=1cm] \scE^{\diamond} \ar[r, "f^{\pm \mu}"]\ar[d, "f^{\pm K}"] &  \scF^{\diamond}
	 	\ar[d, "f^{\pm K}"]
	 	\\ 
	 	\scJ^{\diamond}  \ar[r, "f^{\pm \mu}"] & \scM^{\diamond}
	 \end{tikzcd}
	 \hspace{0.1cm}
	 \supset
	 \hspace{-.1cm}\begin{tikzcd}[column sep=1.1 cm, row sep=0cm]
	 	&[-1.1 cm]
	 	& 
	 	&[-0.2 cm] {\color{gray!80} \cdots}
	 	& 
        &
        {\color{gray!80} \tilde{y}'_\frac{1}{2}}
	 	&[-1cm] 
	 	\\[0.1cm]
        \cdots
	 	&[-1.1 cm]
	 	 y'_{-\frac{3}{2}}
	 	& y_{-\frac{1}{2}}
        \ar[d,  color=gray!80, dashed, "\sigma|L_Z"]
        \ar[ur, color=gray!80, dashed, "L_\sigma"'] 
	 	\ar[l,  "L_\tau"'] 	
	 	&[-0.2 cm] y'_{-\frac{1}{2}}
         \ar[dd, bend left, out=35,in=135,   "\sigma|1"]
	 	& y_{\frac{1}{2}}
         \ar[ur, color=gray!80, dashed, pos=0.2,"Z^{\xi'_1 - \Xi}"'{xshift=-2pt,yshift=5pt}] 
        \ar[l, "Z^{\kappa}"]
	 	 \ar[dd, bend left, pos=0.4, "\sigma|1"]
        &
        y'_\frac{1}{2}
         \ar[dd, bend left, pos=0.4, "\sigma|1"]
	 	&[-1cm] \cdots
	 	\\[0.5cm]
	 	&
	 	& {\color{gray!80} \cdots}
	 	& {\color{gray!80} \tilde{w}'_{-\frac{1}{2}}}
	 	& 
        &
        {\color{gray!80} \tilde{w}'_\frac{1}{2}}
	 	&[-1.5cm] 
        \\[0.1cm]
	 	\cdots
	 	&
	 	 w'_{-\frac{3}{2}}
	 	& w_{-\frac{1}{2}}
        \ar[ur, color=gray!80, dashed, pos=0.2, "Z^{\xi'_1 - \Xi}"'{xshift=-2pt,yshift=5pt}] 
	 	\ar[l,  "Z^\kappa"'] 	
	 	& w'_{-\frac{1}{2}}
	 	& w_{\frac{1}{2}}
        \ar[ur, color=gray!80, dashed,pos=0.2,  "Z^{\xi'_1 - \Xi}"'{xshift=-2pt,yshift=5pt}] 
        \ar[l, "Z^{\kappa}"]
        &
        w'_\frac{1}{2}
	 	&[-1.5cm] \cdots
        \\[1cm]
        \cdots
	 	&
	 	 x'_{-\frac{3}{2}}
          \ar[u,  "\tau|1"]
	 	& x_{-\frac{1}{2}}
        \ar[dr, color=gray!80, dashed, pos=0.8,"Z^{\xi'_1 - \Xi}"'{xshift=-2pt,yshift=5pt}] 
        \ar[u,  "\tau|1"]
	 	\ar[l,  "Z^{\kappa}"']  	
	 	& x'_{-\frac{1}{2}}
         \ar[u,  "\tau|1"]
	 	& x_{\frac{1}{2}}
         \ar[r,  "Z^{\theta}"']
        \ar[l, "Z^{\ell + \theta}"]
	 	 \ar[u,  "\tau|Z^ \omega "]
        &
        x'_\frac{1}{2}
	 	& \cdots
        \\[0.1cm]
	 	&
	 	& 
	 	& {\color{gray!80} \tilde{x}'_{-\frac{1}{2}}}
	 	& 
        &
	 	& 
	 \end{tikzcd}
\end{equation}
The gray dashed arrows point to $\bF[Z]$-torsion elements which only exist under appropriate conditions.

When tensoring with $\cX_{\lambda}(K)^{\hat \cK}$, elements in the middle row  are tensored with the single generator $p$ in the idempotent $1$ part of $\cX_{\lambda}(K)^{\hat \cK}$ ; elements in the top and bottom rows are tensored with generators in the idempotent $0$ part of $\cX_{\lambda}(K)^{\hat \cK}$. See Figure \ref{fig: shifting_f_pm_K} for an example when $\lambda - 2\tau(K)=-1$. The arrows of $\Phi^{K}$ (replacing $f^{K}$) still point vertically downwards whereas the end points of
the arrows of $\Phi^{- K}$ (replacing $f^{-K}$) are shifted by $\lambda - 2 \tau(K)$ units horizontally.
More precisely, we have
\[\Phi^{\pm K}=
	 \hspace{-.3cm}
     \begin{tikzcd}[column sep=0cm, row sep=0.5 cm]
     \cCFK(K)^{\hat \cR}  \ar[d ]& \boxtimes& \scE^{\diamond}\oplus \scF^{\diamond}\ar[dd]\\
     \delta^1_{\sigma,\tau} \ar[dd]\ar[drr]& &&\\
     & & f^{\pm K} \ar[d] \\
     \langle p \rangle \otimes {\bF[U,T,T^{-1}]} &  \boxtimes  & \scJ^{\diamond}\oplus \scM^{\diamond}  
	 \end{tikzcd}\]
and 
\[\delta^1_{\sigma,\tau}(a_0) =  p \otimes T^{\lambda - 2\tau(K)}\tau,\qquad  \delta^1_{\sigma,\tau}(a_m) =  p \otimes \sigma. \]
Therefore, using \eqref{eq: scE_scJ_actions} for the description of $f^{\pm K}: \scE^{\diamond}\to \scJ^{\diamond}$ (and a similar one from $\scF^{\diamond}$ to $\scM^{\diamond}$, see \cite{CZZApp}*{Figures~5.1}), we compute
\begin{align}\label{eq: Phi-K_arrow}
       \Phi^{-K} (a_0|x_i) &= \begin{cases}
              p|w_{i+\lambda - 2\tau(K)} \quad &i<0\\
             p|w_{i+\lambda - 2\tau(K)} Z^ \omega  \quad &i=\frac{1}{2}\\
            0 \quad &i>1
        \end{cases}\hspace{3em}
        \Phi^{-K} (a_0|x'_i) = \begin{cases}
              p|w'_{i+\lambda - 2\tau(K)} \quad &i<0\\
            0 \quad &i>0
        \end{cases},
\intertext{and}\label{eq: PhiK_arrow}
         \Phi^{K} (a_m|y_i) &= \begin{cases}
              0 \hspace{3 em} &i<-1\\
            \epsilon \cdot p|\tilde{w}_{-\frac{1}{2}} Z^{q}  \quad &i=-\frac{1}{2}\\
             p|w_{i} \quad &i>0
        \end{cases}
        \hspace{5em}
        \Phi^{K} (a_m|y'_i) = \begin{cases}
              0 \quad &i<-1\\
             p|w'_{i} \quad &i>-1
        \end{cases},
\end{align}
where $q=\xi^{\ell/2-1}+1-R_{\frac{\ell}{2}-1}+R_{\frac{\ell}{2}-2}$ and $\epsilon \in \{0,1\}$ is nonzero if and only if $R_{\frac{\ell}{2}-1}-R_{\frac{\ell}{2}-2}>1$.
This follows since
$\Phi^{K}(a_m|y_{-\frac{1}{2}})= p| L_Z(y_{-\frac{1}{2}}) = p|L_Z(\xs^{\frac{\ell}{2}-2}_0)$ 
by the preceding discussion, and $L_Z(\xs^{\frac{\ell}{2}-2}_0)$ is obtained by an argument analogous to that used in the proof of Lemma~\ref{lem: C_S_structure_maps}, together with Equation~\eqref{eq: gradingxs}.  Therefore, we have the following: 

 \begin{lem} \label{lem: middlerow_chain}
Under the basis of $\bX^{\diamond}(P,K,\lambda)$ given by Equation~\eqref{eq: CFK_PK_Z},  the only incoming arrows to $p|w_{i+\frac{1}{2}}$ are the arrows of $\Phi^K$ if $i \geq 0$, and $\Phi^{-K}$ if $i \leq \lambda - 2 \tau(K)$; the only incoming arrows to $p|w'_{i-\frac{1}{2}}$ are the arrow of $\Phi^{-\mu}$ from   $p|w_{i+\frac{1}{2}}$,
the arrow of $\Phi^K$ if $i \geq 0$, and the arrow of $\Phi^{-K}$ if $i \leq \lambda - 2 \tau(K)$.
 \end{lem}
\begin{proof}
Since $\delta^1(p)=0,$ the only incoming arrows to $p|w_{i+\frac{1}{2}}$ are the arrows of $\Phi^{\pm K}$ and  the only incoming arrows to $p|w'_{i-\frac{1}{2}}$ are the arrows of $\Phi^{\pm K}$ and  $\Phi^{\pm \mu}$. 
    The statements regarding $\Phi^{\pm K}$ follow from Equations~\eqref{eq: Phi-K_arrow} and \eqref{eq: PhiK_arrow}. The statement regarding $\Phi^{\pm\mu}$ is obtained by combining   \cite{CZZApp}*{Figures~5.2} and Lemma \ref{lem: C_S_structure_maps}.
\end{proof} 
The length-$0$ differentials, although not depicted in Figure \ref{fig: shifting_f_pm_K},  can also be fully understood using the top and bottom row of Equation~\eqref{eq: scE_scF_actions}.  We only focus on elements of the form $a_m | y_s $ and $a_m | y'_{s-1} $ for $s>0$, namely those mapping to a free generator under $\Phi^K$.
It is straightforward to verify that 
an incoming or outgoing length-$0$ differential exists for such an element  
if and only if $\varepsilon(K) = -1$ and $s=\frac{1}{2}$. 
 In this case,  there is an outgoing length-$0$ differential from $a_{m}|y'_{-\frac{1}{2}}$ to $a_{m-1}|y'_{-\frac{1}{2}-n}$ weighted by $1$  and  an outgoing length-$0$ differential from $a_{m}|y_{\frac{1}{2}}$ to $a_{m-1}|y_{\frac{1}{2}-n}$ weighted by $Z^{R_{\ell/2-1}-R_{\ell/2-2}}$. The coefficient $Z^{R_{\ell/2-1}-R_{\ell/2-2}}$ is that of $L_W(y_{\frac{1}{2}})=L_W(\xs_0^{\ell/2-1})$, which can be computed using a similar argument as in the proof of Lemma \ref{lem: C_S_structure_maps}.
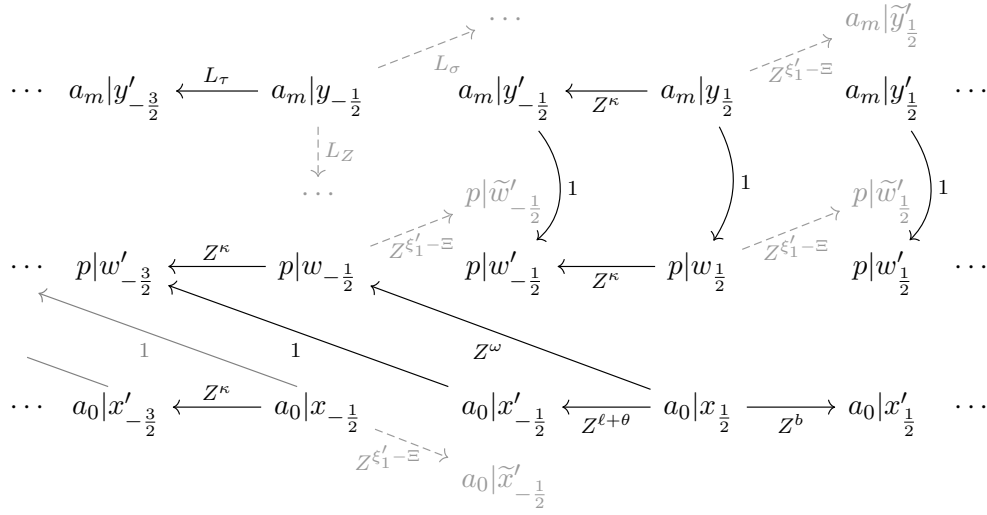
\begin{figure}[h]
\begin{equation*}
\begin{tikzcd}[column sep=1.1 cm, row sep=0cm]
&[-1.1 cm]
	 	& 
	 	&[-0.2 cm] {\color{gray!80} \cdots}
	 	& 
        &
        {\color{gray!80} a_m|\tilde{y}'_\frac{1}{2}}
	 	&[-1cm] 
	 	\\[0.1cm]
	 	\cdots
	 	&[-1.1 cm]
	 	 a_m |y'_{-\frac{3}{2}}
	 	& a_m |y_{-\frac{1}{2}}
         \ar[d,  color=gray!80, dashed, "L_Z"]
        \ar[ur, color=gray!80, dashed, "L_{\sigma}"'] 
	 	\ar[l,  "L_\tau"'] 	
	 	&[-0.2 cm] a_m |y'_{-\frac{1}{2}}
         \ar[dd, bend left, out=40,in=130,  "1"]
	 	& a_m |y_{\frac{1}{2}}
        \ar[ur, color=gray!80, dashed, pos=0.2,"Z^{\xi'_1 - \Xi}"'{xshift=-2pt,yshift=5pt}] 
        \ar[l, "Z^{\kappa}"]
	 	 \ar[dd, bend left, pos=0.48,  "1"]
        &
        a_m |y'_\frac{1}{2}
         \ar[dd, bend left, out=35,in=135,  "1"]
	 	&[-1cm] \cdots
	 	\\[0.5cm]
	 	&
	 	& {\color{gray!80} \cdots}
	 	& {\color{gray!80} p|\tilde{w}'_{-\frac{1}{2}}}
	 	& 
        &
        {\color{gray!80} p|\tilde{w}'_\frac{1}{2}}
	 	&[-1.5cm] 
        \\[0.1cm]
	 	\cdots
	 	&
	 	 p|w'_{-\frac{3}{2}}
	 	& p|w_{-\frac{1}{2}}
        \ar[ur, color=gray!80, dashed, pos=0.2,"Z^{\xi'_1 - \Xi}"'{xshift=-2pt,yshift=5pt}] 
	 	\ar[l,  "Z^\kappa"'] 	
	 	& p|w'_{-\frac{1}{2}}
	 	& p|w_{\frac{1}{2}}
        \ar[ur, color=gray!80, dashed, pos=0.2,"Z^{\xi'_1 - \Xi}"'{xshift=-2pt,yshift=5pt}] 
        \ar[l, "Z^{\kappa}"]
        &
        p|w'_\frac{1}{2}
	 	&[-1.5cm] \cdots
        \\[1cm]
        \cdots
	 	&
	 	 a_0 |x'_{-\frac{3}{2}}
	 	& a_0 |x_{-\frac{1}{2}}
        \ar[dr, color=gray!80, dashed, pos=0.8,"Z^{\xi'_1 - \Xi}"'{xshift=-2pt,yshift=5pt}] 
	 	\ar[l,  "Z^{\kappa}"'] 
	 	& a_0 |x'_{-\frac{1}{2}}
         \ar[llu,  "1"]
        \ar[llu, " 1", gray, transform canvas={shift={(-2cm,0cm)}},shorten >=0.3cm]
        \ar[llu,  -, gray, transform canvas={shift={(-4.5 cm,0cm)}},shorten >=2.8 cm]
	 	& a_0 |x_{\frac{1}{2}}
         \ar[r,  "Z^{ b}"']
        \ar[l, "Z^{\ell  + \theta}"]
	 	 \ar[llu,  "Z^ \omega "]
        &
        a_0 |x'_\frac{1}{2}
	 	& \cdots
        \\[0.1cm]
	 	&
	 	& 
	 	& {\color{gray!80} a_0|\tilde{x}'_{-\frac{1}{2}}}
	 	& 
        &
	 	& 
	 \end{tikzcd}
    \end{equation*} 
    \caption{ When $\lambda - 2\tau(K)=-1$. Top row: $\bigoplus_s E^{\diamond}_{s,\frac{\ell-3}{2}} \oplus F^{\diamond}_{s,\frac{\ell-3}{2}}$. Middle row: $\bigoplus_s J^{\diamond}_{s,\frac{\ell-3}{2}} \oplus M^{\diamond}_{s,\frac{\ell-3}{2}}$. Bottom row: $\bigoplus_s E^{\diamond}_{s,\frac{\ell-1}{2}} \oplus F^{\diamond}_{s,\frac{\ell-1}{2}}$. Arrows from the top to the middle row represent $\Phi^{K}$ and  arrows from the bottom to the middle row represent $\Phi^{-K}$. }
    \label{fig: shifting_f_pm_K} 
\end{figure}

\begin{figure}[hbtp!]
\captionsetup{width=\textwidth}
\subfigure[When $\varepsilon(K)=1, \eta>0$ and $\lambda-2\tau(K) = -1.$]{
\begin{tikzpicture}[scale=0.5]   
\begin{scope}
	\foreach \k in {0,8}{
	\foreach \j in {-5,5,15}{
		\fill[gray!20] (\j+-1,1.5+\k) rectangle (\j+1.5,0.5+\k);
		\fill[gray!20] (\j+-1,2.5+\k) rectangle (\j+2.5,1.5+\k);
		\fill[gray!20] (\j+-1,3+\k) rectangle (\j+3,2.5+\k);
		\fill[gray!20] (\j+-1,0.5+\k) rectangle (\j+0.5,-0.5+\k);
		\fill[gray!20] (\j+-1,-0.5+\k) rectangle (\j+-0.5,-1.5+\k);
	}
	\foreach \j in {-10,-5,...,10,15}
	{\foreach \i in {0,...,3}
		{\draw[thin, black!20!white]  (\j+\i-0.5, 3+\k) -- (\j+\i-0.5, -2+\k);}
		\foreach \i in {-1,...,3}
		{ \draw[thin, black!20!white]  (\j+3, \i-0.5+\k) -- (\j-1, \i-0.5+\k); }}
	\foreach \j in {-10,0,10}
	{\draw[thin, black!60!white] (\j-1, 2.5+\k) -- (\j+2.5, 2.5+\k);
		\draw[thin, black!60!white] (\j+2.5, 2.5+\k) -- (\j+2.5, -2+\k);
		\draw[thin, black!60!white] (\j+1.5, 1.5+\k) -- (\j+1.5, -2+\k);
		\draw[thin, black!60!white]  (\j+0.5, 0.5+\k) -- (\j+0.5, -2+\k);
		\draw[thin, black!60!white]  (\j+1.5, 1.5+\k) -- (\j-1, 1.5+\k);
		\draw[thin, black!60!white] (\j+0.5, 0.5+\k) -- (\j-1, 0.5+\k);
		\draw[thin, black!60!white] (\j-1, -0.5+\k) -- (\j-0.5, -0.5+\k);
		\draw[thin, black!60!white] (\j-0.5, -0.5+\k) -- (\j-0.5, -2+\k);
		\draw[thin, black!60!white] (\j-0.5, -1.5+\k) -- (\j-0.5, -2+\k);}}
\end{scope}  

 \foreach \k in {8}{
 \foreach \j in {-1,0,1}
     { \foreach \i in {0,1}
     {
     {\filldraw ({\j*10+1+5*\i}, 4.5) circle (2pt) node[] (o\i\j) {};}
     \filldraw ({\j*10+1+5*\i+2}, -\j-\i+\k+1) circle (2pt)
     node[] (am-1\i\j) {};
     \filldraw ({\j*10+1+5*\i}, -\j-\i+\k+1) circle (2pt) 
     node[] (am\i\j) {};
     \draw[-stealth] ($(am-1\i\j)$) -- ($(am\i\j)+(0.1,0)$);
     \ifnum \j>-1 
     \draw[color=brown, -stealth]   ($(am\i\j)$) to node[midway, right, xshift=-2pt]{{\tiny $1$}} ($(o\i\j)+(0,0.1)$);
     \fi
     }
     \ifnum\j>-1            \draw[red, bend left=20, -stealth] ($(am1\j)+ (-0.1,-0.1)$) to node[midway, above] {{\tiny $Z^\kappa$}} ($(am0\j) + (0.1,-0.1)$);
     \else
     \draw[red, bend left=20, -stealth] ($(am1\j)+ (-0.1,-0.1)$) to node[midway, below] {{\tiny $L_\sigma$}} ($(am0\j) + (0.1,-0.1)$); \fi
      \draw[red, bend left=10, -stealth] ($(o1\j)+ (-0.1,-0.1)$) to node[pos=0.7, above] {{\tiny $Z^\kappa$}} ($(o0\j) + (0.1,-0.1)$);
	   }}   
       \draw[color=brown,densely dashed, -stealth]   ($(am1-1)$) to node[midway,  right, xshift=-2pt]{{\tiny $L_Z$}} ($(o1-1)+(0,0.1)$);
                \draw[red, bend right=5, densely dashed, -stealth] ($(o1-1)+ (0.2,0)$) to node[pos=0.5, above] {{\tiny $Z^{\xi'_1 -\Xi}$}} ($(o00) + (-0.1,-0.1)$);
                \draw[red, bend right=5, densely dashed, -stealth] ($(o10)+ (0.2,0)$) to node[pos=0.5, above] {{\tiny $Z^{\xi'_1 -\Xi}$}} ($(o01) + (-0.1,-0.1)$);
                \draw[red, bend right=15, densely dashed, -stealth] ($(am1-1)+ (0.2,0)$) to node[pos=0.5, below] {{\tiny $L_\tau$}} ($(am00) + (-0.1,-0.1)$);
                \draw[red, bend right=15, densely dashed, -stealth] ($(am10)+ (0.2,0)$) to node[pos=0.5, below] {{\tiny $Z^{\xi'_1 -\Xi}$}} ($(am01) + (-0.1,-0.1)$);
      \foreach \j in {-1,0,1}
     { \foreach \i in {0,1}
     {\filldraw ({\j*10+1+5*\i}, 2-\j-\i) circle (2pt)
     node[] (a1\i\j) {};
     \filldraw ({\j*10+1+5*\i}, -\j-\i) circle (2pt) 
     node[] (a0\i\j) {};
\ifnum\j<1        \ifnum \i=1         \draw[densely dashed,-stealth] ($(a1\i\j)$) -- ($(a0\i\j)+(0,0.1)$);      \else \draw[-stealth] ($(a1\i\j)$) -- ($(a0\i\j)+(0,0.1)$); \fi \else \draw[-stealth] ($(a1\i\j)$) -- ($(a0\i\j)+(0,0.1)$); \fi
     \filldraw [color=green!70!black] ({\j*10+2+5*\i}, 2-\j-\i) circle (2pt) node[] (a1'\i\j) {};
     \draw [color=green!70!black, -stealth] ($(a1'\i\j)$) -- ($(a1\i\j)+(0.1,0)$);
     }
     \draw[red, bend left=20, -stealth] ($(a01\j)+ (-0.1,-0.1)$) to node[midway, below] {{\tiny $Z^{\ell+\theta}$}} ($(a00\j) + (0.1,-0.1)$);
	   }    
        \draw[color=brown, bend left=40, -stealth]   ($(a00-1)$) to node[midway, left, xshift=2pt]{{\tiny $1$}} ($(o0-1)+(0,-0.1)$);
         \draw[color=brown, bend left=20, -stealth]   ($(a01-1)$) to node[midway, left, xshift=2pt]{{\tiny $Z^ \omega $}} ($(o1-1)+(0,-0.1)$);
     \draw[red, bend right=20, -stealth] ($(a010) + (0.1,-0.1)$)  to node[midway, below] {{\tiny $Z^ \theta$}} ($(a001)+ (-0.1,-0.1)$);
     \draw[red, bend right=20, -stealth] ($(a01-1) + (0.1,-0.1)$)  to node[midway, below] {{\tiny $Z^ \theta$}} ($(a000)+ (-0.1,-0.1)$);
    \node [left] at (a00-1) {{\tiny$a_0$}};
    \node [above] at (a10-1) {{\tiny$a_1$}};
    \node [color=green!70!black,right] at (a1'0-1) {{\tiny$a'_1$}};
    \node [left] at (o0-1) {{\tiny$p|w'_{-\frac{3}{2}}$}};
     \node [above left] at (o1-1) {{\tiny$p|w_{-\frac{1}{2}}$}};
\node [left] at (am0-1) {{\tiny$a_m$}};
\node [below] at (am-10-1) {{\tiny$a_{m-1}$}};
    
     \node [] at ($(a00-1)+(.4,1)$) {{\tiny$U^2$}};
     \node [] at ($(a01-1)+(.7,1)$) {{\tiny$UL_Z$}};
     \node [] at ($(a000)+(.3,1)$) {{\tiny$U$}};
     \node [] at ($(a010)+(.5,1)$) {{\tiny$L_Z$}};
     \node [] at ($(a001)+(.3,1)$) {{\tiny$1$}};
     \node [] at ($(a011)+(.3,1)$) {{\tiny$1$}};
     \node [color=green!70!black] at ($(a10-1)+(0.6,-0.3)$) {{\tiny$1$}};
      \node [color=green!70!black] at ($(a11-1)+(0.6,0.3)$) {{\tiny$1$}};
      \node [color=green!70!black] at ($(a100)+(0.6,0.3)$) {{\tiny$1$}};
      \node [color=green!70!black] at ($(a110)+(0.6,0.3)$) {{\tiny$Z^{\omega}$}};
      \node [color=green!70!black] at ($(a101)+(0.8,0.3)$) {{\tiny$U$}};
      \node [color=green!70!black] at ($(a111)+(0.8,0.3)$) {{\tiny$U$}};
     \node [] at ($(am0-1)+(1,.3)$) {{\tiny$U$}};
     \node [] at ($(am1-1)+(1,.3)$) {{\tiny$UL_W$}};
     \node [] at ($(am00)+(0.75,.3)$) {{\tiny$U^2$}};
     \node [] at ($(am10)+(0.75,.3)$) {{\tiny$U$}};
     \node [] at ($(am01)+(0.75,.3)$) {{\tiny$U^2$}};
     \node [] at ($(am11)+(0.75,.3)$) {{\tiny$U$}};   
     \draw[blue] 
     ($(a00-1)+(-5pt,-5pt)$) rectangle ($(a00-1)+(5pt,5pt)$);
     \draw[blue] 
     ($(o1-1)+(-5pt,-5pt)$) rectangle ($(o1-1)+(5pt,5pt)$);
    \draw[violet] (a01-1) circle (5pt);
    \draw[violet] (a010) circle (5pt);
    \draw[violet] (a101) circle (5pt);
     \draw[yellow!50!gray, very thick, dashed]
        (18,5.5)-- (17,5.5)--(17,8)-- (15,8)-- (15,5.5)-- (12,5.5)-- (12,9)-- (10,9) -- (10,5.5) -- (7,5.5)-- (7,9)-- (5,9)-- (5,5.5)-- (2,5.5) -- (2,10) -- (0,10) -- (0,3.6) -- (18,3.6);
\end{tikzpicture}   
\label{subfig: prop_ending_epsilon>0_eta>0}
}
\subfigure[When $\varepsilon(K)=1, \eta<0$ and $\lambda-2\tau(K) = -1.$]{
\begin{tikzpicture}[scale=0.5]
\begin{scope}
	\foreach \k in {0,8}{
		\foreach \j in {-5,5,15}{
			\fill[gray!20] (\j+-1,1.5+\k) rectangle (\j+1.5,0.5+\k);
			\fill[gray!20] (\j+-1,2.5+\k) rectangle (\j+2.5,1.5+\k);
			\fill[gray!20] (\j+-1,3+\k) rectangle (\j+3,2.5+\k);
			\fill[gray!20] (\j+-1,0.5+\k) rectangle (\j+0.5,-0.5+\k);
			\fill[gray!20] (\j+-1,-0.5+\k) rectangle (\j+-0.5,-1.5+\k);
		}
		\foreach \j in {-10,-5,...,10,15}
		{\foreach \i in {0,...,3}
			{\draw[thin, black!20!white]  (\j+\i-0.5, 3+\k) -- (\j+\i-0.5, -2+\k);}
			\foreach \i in {-1,...,3}
			{ \draw[thin, black!20!white]  (\j+3, \i-0.5+\k) -- (\j-1, \i-0.5+\k); }}
		\foreach \j in {-10,0,10}
		{\draw[thin, black!60!white] (\j-1, 2.5+\k) -- (\j+2.5, 2.5+\k);
			\draw[thin, black!60!white] (\j+2.5, 2.5+\k) -- (\j+2.5, -2+\k);
			\draw[thin, black!60!white] (\j+1.5, 1.5+\k) -- (\j+1.5, -2+\k);
			\draw[thin, black!60!white]  (\j+0.5, 0.5+\k) -- (\j+0.5, -2+\k);
			\draw[thin, black!60!white]  (\j+1.5, 1.5+\k) -- (\j-1, 1.5+\k);
			\draw[thin, black!60!white] (\j+0.5, 0.5+\k) -- (\j-1, 0.5+\k);
			\draw[thin, black!60!white] (\j-1, -0.5+\k) -- (\j-0.5, -0.5+\k);
			\draw[thin, black!60!white] (\j-0.5, -0.5+\k) -- (\j-0.5, -2+\k);
			\draw[thin, black!60!white] (\j-0.5, -1.5+\k) -- (\j-0.5, -2+\k);}}
\end{scope}       
 \foreach \k in {8}{
 \foreach \j in {-1,0,1}
     { \foreach \i in {0,1}
     {
     {\filldraw ({\j*10+1+5*\i}, 4.5) circle (2pt) node[] (o\i\j) {};}
     \filldraw ({\j*10+1+5*\i+2}, -\j-\i+\k+1) circle (2pt)
     node[] (am-1\i\j) {};
     \filldraw ({\j*10+1+5*\i}, -\j-\i+\k+1) circle (2pt) 
     node[] (am\i\j) {};
     \draw[-stealth] ($(am-1\i\j)$) -- ($(am\i\j)+(0.1,0)$);
     \ifnum\j>-1
     \draw[color=brown, -stealth]   ($(am\i\j)$) to node[midway, right, xshift=-2pt]{{\tiny $1$}} ($(o\i\j)+(0,0.1)$);
     \fi
     }
     \ifnum\j=1            \draw[red, bend left=20, -stealth] ($(am1\j)+ (-0.1,-0.1)$) to node[midway, above] {{\tiny $Z^\kappa$}} ($(am0\j) + (0.1,-0.1)$);
     \else
     \draw[red, bend left=20, -stealth] ($(am1\j)+ (-0.1,-0.1)$) to node[midway, below] {{\tiny $Z^\kappa$}} ($(am0\j) + (0.1,-0.1)$); \fi
      \draw[red, bend left=10, -stealth] ($(o1\j)+ (-0.1,-0.1)$) to node[midway, above] {{\tiny $Z^\kappa$}} ($(o0\j) + (0.1,-0.1)$);
	   }}   
                \draw[red, bend right=5, densely dashed, -stealth] ($(o1-1)+ (0.2,0)$) to node[pos=0.5, above] {{\tiny $Z^{\xi'_1 -\Xi}$}} ($(o00) + (-0.1,-0.1)$);
                \draw[red, bend right=5, densely dashed, -stealth] ($(o10)+ (0.2,0)$) to node[pos=0.5, above] {{\tiny $Z^{\xi'_1 -\Xi}$}} ($(o01) + (-0.1,-0.1)$);
                   \draw[red, bend right=15, densely dashed, -stealth] ($(am1-1)+ (0.2,0)$) to node[pos=0.5, below] {{\tiny $Z^{\xi'_1 -\Xi}$}} ($(am00) + (-0.1,-0.1)$);
                \draw[red, bend right=15, densely dashed, -stealth] ($(am10)+ (0.2,0)$) to node[pos=0.5, below] {{\tiny $Z^{\xi'_1 -\Xi}$}} ($(am01) + (-0.1,-0.1)$);
      \foreach \j in {-1,0,1}
     { \foreach \i in {0,1}
     {\filldraw ({\j*10+1+5*\i}, 2-\j-\i) circle (2pt)
     node[] (a1\i\j) {};
     \filldraw ({\j*10+1+5*\i}, -\j-\i) circle (2pt) 
     node[] (a0\i\j) {};
\ifnum\j<1        \ifnum \i=1         \draw[densely dashed,-stealth] ($(a1\i\j)$) -- ($(a0\i\j)+(0,0.1)$);      \else \draw[-stealth] ($(a1\i\j)$) -- ($(a0\i\j)+(0,0.1)$); \fi \else \draw[-stealth] ($(a1\i\j)$) -- ($(a0\i\j)+(0,0.1)$); \fi
     \filldraw [color=green!70!black] ({\j*10+5*\i}, 2-\j-\i) circle (2pt) node[] (a1'\i\j) {};
     \draw [color=green!70!black, -stealth] ($(a1\i\j)$) -- ($(a1'\i\j)+(0.1,0)$);
     }
      \draw[color=brown,densely dashed, -stealth]   ($(am1-1)$) to node[midway,  right, xshift=-2pt]{{\tiny $L_Z$}} ($(o1-1)+(0,0.1)$);
     \draw[red, bend left=20, -stealth] ($(a01\j)+ (-0.1,-0.1)$) to node[midway, below] {{\tiny $Z^{\ell+\theta}$}} ($(a00\j) + (0.1,-0.1)$);
	   }    
        \draw[color=brown, bend right=60, -stealth]   ($(a00-1)$) to node[midway, right, xshift=-2pt]{{\tiny $1$}} ($(o0-1)+(0,-0.1)$);
         \draw[color=brown, bend right=20, -stealth]   ($(a01-1)$) to node[midway, right, xshift=-2pt]{{\tiny $Z^ \omega $}} ($(o1-1)+(0,-0.1)$);
     \draw[red, bend right=20, -stealth] ($(a010) + (0.1,-0.1)$)  to node[midway, below] {{\tiny $Z^ \theta$}} ($(a001)+ (-0.1,-0.1)$);
     \draw[red, bend right=20, -stealth] ($(a01-1) + (0.1,-0.1)$)  to node[midway, below] {{\tiny $Z^ \theta$}} ($(a000)+ (-0.1,-0.1)$);
    \node [left] at (a00-1) {{\tiny$a_0$}};
    \node [right] at ($(a10-1)+(-.1,.2)$) {{\tiny$a_1$}};
    \node [color=green!70!black,left] at (a1'0-1) {{\tiny$a'_1$}};
    \node [left] at (o0-1) {{\tiny$p|w'_{-\frac{3}{2}}$}};
     \node [above left] at (o1-1) {{\tiny$p|w_{-\frac{1}{2}}$}};
\node [left] at (am0-1) {{\tiny$a_m$}};
\node [below] at (am-10-1) {{\tiny$a_{m-1}$}};
    
     \node [] at ($(a00-1)+(-.4,1)$) {{\tiny$U^2$}};
     \node [] at ($(a01-1)+(-.7,1)$) {{\tiny$UL_Z$}};
     \node [] at ($(a000)+(.3,1)$) {{\tiny$U$}};
     \node [] at ($(a010)+(.5,1)$) {{\tiny$L_Z$}};
     \node [] at ($(a001)+(.3,1)$) {{\tiny$1$}};
     \node [] at ($(a011)+(.3,1)$) {{\tiny$1$}};
    \node [color=green!70!black] at ($(a10-1)+(-0.4,0.3)$) {{\tiny$1$}};
      \node [color=green!70!black] at ($(a11-1)+(-0.4,0.3)$) {{\tiny$1$}};
      \node [color=green!70!black] at ($(a100)+(-0.4,0.3)$) {{\tiny$1$}};
      \node [color=green!70!black] at ($(a110)+(-0.4,0.3)$) {{\tiny$1$}};
      \node [color=green!70!black] at ($(a101)+(-0.4,0.3)$) {{\tiny$1$}};
      \node [color=green!70!black] at ($(a111)+(-0.4,-0.3)$) {{\tiny$Z^{\omega}$}};
     \node [] at ($(am0-1)+(1,.3)$) {{\tiny$U$}};
     \node [] at ($(am1-1)+(1,.3)$) {{\tiny$UL_W$}};
     \node [] at ($(am00)+(0.75,.3)$) {{\tiny$U^2$}};
     \node [] at ($(am10)+(0.75,.3)$) {{\tiny$U$}};
     \node [] at ($(am01)+(0.75,.3)$) {{\tiny$U^2$}};
     \node [] at ($(am11)+(0.75,.3)$) {{\tiny$U$}};   
     \draw[blue] 
     ($(a00-1)+(-5pt,-5pt)$) rectangle ($(a00-1)+(5pt,5pt)$);
     \draw[blue] 
       ($(o1-1)+(-5pt,-5pt)$) rectangle ($(o1-1)+(5pt,5pt)$);
    \draw[violet] (a01-1) circle (5pt);
    \draw[violet] (a010) circle (5pt);
    \draw[violet] (a101) circle (5pt);
    \draw[violet] (a1'11) circle (5pt);
    \draw[red, densely dashed, in=165, out=85, -stealth] ($(a1'11)+ (0.1,0.1)$) to  ($(a1'11)+ (2.5,0.7)$);
    \draw[red, bend right=5, -stealth] ($(a111)+ (-0.1,0.1)$) to node[pos=0.5, above] {{\tiny $Z^{\ell+\theta}$}} ($(a101) + (0.1,0)$);
       \draw[red, bend left=5, -stealth] ($(a1'11)+ (-0.1,0)$) to node[pos=0.3, below] {{\tiny $Z^{\kappa}$}} ($(a1'01) + (0.1,-0.1)$);
        \draw[yellow!50!gray, very thick, dashed]
        (18,5.5)-- (17,5.5)--(17,8)-- (15,8)-- (15,5.5)-- (12,5.5)-- (12,9)-- (10,9) -- (10,5.5) -- (7,5.5)-- (7,9)-- (5,9)-- (5,5.5)-- (2,5.5) -- (2,10) -- (0,10) -- (0,3.6) -- (18,3.6);
\end{tikzpicture}   

\label{subfig: prop_ending_epsilon>0_eta<0}
}    
    \caption{
    A schematic illustration  of the proof of Proposition \ref{prop: ending_arrow_general_case}  when $n=2$ and $\omega-\theta>0$, drawn in the form of the diagram in Equation~\eqref{eq: ending_EFJM}.
    In each diagram, the solid dots in the middle row represent the free generators, alternating between those in
    $J^{\diamond}_{s,\frac{\ell-3}{2}}$ and $M^{\diamond}_{s,\frac{\ell-3}{2}}$, from left to right in  increasing order of $s$. We label the first two to provide a reference point. The horizontal alignment of the three rows is chosen so that the arrows of $\Phi^{K}$ and $\Phi^{-K}$ point vertically downwards and upwards, respectively. The region enclosed by the dashed line indicates the portion that can be truncated. As before, dotted arrows point to torsion elements, and
    the element $\alpha_0$ (resp.~$\alpha_1$) is defined as a linear combination of the boxed (resp.~circled) generators.
    }
    \label{fig: prop_ending_epsilon>0}
\end{figure}

We divide the proof of Proposition \ref{prop: ending_arrow_general_case} into two parts, depending on the sign of $\varepsilon(K).$
\begin{proof}[Proof of Proposition \ref{prop: ending_arrow_general_case} in the case $\varepsilon(K) = 1$] 
 We introduce a truncation to simplify the arguments. Let
 \[
Y:=\bigoplus_{s\in \bZ_{\geq 0}} a_m|\big(\scE_{s+\frac{1}{2},\frac{\ell-3}{2}} \oplus \scF_{s-\frac{1}{2},\frac{\ell-3}{2}}\big) \oplus p|\big(\scJ_{s+\frac{1}{2},\frac{\ell-3}{2}} \oplus \scM_{s-\frac{1}{2},\frac{\ell-3}{2}}\big).
\]
We now check that $Y$ is a subcomplex. Firstly, the length-0 differentials vanish on $Y$ since $a_m$ is a cycle in $\cCFK(K)^{\hat{\cR}}$ and similarly $\delta^1(p)=0$. We now check that $Y$ is preserved by $\Phi^{\pm \mu}$ and $\Phi^{\pm K}$. Since $f^{\mu}(\scE_{s+\frac{1}{2},\frac{\ell-3}{2}}) \subset \scF_{s+\frac{1}{2},\frac{\ell-3}{2}}$ and $f^{-\mu}(\scE_{s+\frac{1}{2},\frac{\ell-3}{2}}) \subset \scF_{s-\frac{1}{2},\frac{\ell-3}{2}}$, we have $\Phi^{\pm \mu}(Y) \subset Y.$ Since $\delta^1(a_m) =  p \otimes \sigma$, only $\Phi^{K}$ is nonzero on $Y$. By  Figure~\cite{CZZApp}*{Figures~5.1}, 
$f^{K}(\scE_{s+\frac{1}{2},\frac{\ell-3}{2}}) \subset \scJ_{s+\frac{1}{2},\frac{\ell-3}{2}}$ and $f^{K}(\scF_{s+\frac{1}{2},\frac{\ell-3}{2}}) \subset \scM_{s+\frac{1}{2},\frac{\ell-3}{2}}$,
 we have $\Phi^{K}(Y) \subset Y.$
 
 Since  by  Figure~\cite{CZZApp}*{Figures~5.1} (See also Equation~\eqref{eq: scE_scJ_actions}), the maps 
      \[    
\begin{tikzcd}
\Phi^K:&[-3em]
a_m|\scE^\diamond_{s+\frac{1}{2},\frac{\ell-3}{2}}
\arrow[r]
&[-1em]
p|\scJ^\diamond_{s+\frac{1}{2},\frac{\ell-3}{2}}
\\[-1em]
&
\cC^{\diamond}_{\frac{\ell}{2}-1}
\arrow[u,equal]
&
\cC^{\diamond}_{\frac{\ell}{2}-1}
\arrow[u,equal]
\end{tikzcd} \quad \text{and} \quad
\begin{tikzcd}
\Phi^K:&[-3em]
a_m|\scF^\diamond_{s-\frac{1}{2},\frac{\ell-3}{2}}
\arrow[r]
&[-1em]
p|\scM^\diamond_{s-\frac{1}{2},\frac{\ell-3}{2}}
\\[-1em]
&
\cS^{\diamond}
\arrow[u,equal]
&
\cS^{\diamond}
\arrow[u,equal]
\end{tikzcd}
\]
are both the identity map for $s\in \bZ_{\geq 0}$,  the subcomplex 
$Y$
is acyclic. 
Therefore,  we may replace $\bX^{\diamond}(P,K,\lambda)$ by the chain homotopy equivalent complex $\bX^{\diamond}(P,K,\lambda)/Y$,   when desired. 

Recall that $\bX^{\diamond}(P,K,\lambda)$ decomposes as mapping cone of its free part and torsion part. This induces 
 similar  decompositions
\[
Y= \Cone \big( Y^{\free} \xrightarrow{\partial} Y^{\Tor} \big) \]
and 
\[
\bX^{\diamond}(P,K,\lambda)/Y = \Cone \big( \bX^{\free}(P,K,\lambda)/Y^{\free} \xrightarrow{\partial} \bX^{\Tor}(P,K,\lambda)/Y^{\Tor} \big). \]
  In $\bX^{\diamond}(P,K,\lambda)/Y$, among the generators $p|w_{i+\frac{1}{2}}$ and $p|w'_{i-\frac{1}{2}}$, only those with $i<0$ are present.

In the following proof, we use $\partial$ for the differential on $\bX^{\diamond}(P,K,\lambda)$, $\partial^{\free}$ for the differential on $\bX^{\free}(P,K,\lambda)$ 
and $\bar{\partial}$ for the differential on $\bX^{\diamond}(P,K,\lambda)/Y$. We proceed case by case.

\medskip
\noindent\textbf{Case \eqref{it: ending_general_epsilon>0_case_1}: when $\varepsilon(K)=1$ and $\eta>0$.} 
There is an arrow weighted by $W^{\eta}$ from  $a'_1 $ to $ a_1$, followed by an arrow weighted by $Z^{n}$ from  $a_1 $ to $ a_0$ in $\cCFK(K)$. By Lemma \ref{lem: snake_structure}, we have $n>1$.
\begin{enumerate}[label=$($1.\arabic*$)$, ref=$($1.\arabic*$)$]
    \item \label{it: epsilon>0_eta>0_lambda-2tau<0} 
    \textbf{When $\lambda - 2\tau(K) < 0$.}
    See Figure~\ref{fig: prop_ending_epsilon>0}\subref{subfig: prop_ending_epsilon>0_eta>0} for an example when $\lambda - 2\tau(K) = -1$ and $\omega-\theta>0$.
   If $\lambda - 2\tau(K) < -1,$ the end points of 
    the arrows of $\Phi^{-K}$ are shifted to the left,  while the rest of the diagram remains unchanged.
Similar to Case \eqref{it: non_ending_braided_case_2} and \eqref{it: non_ending_general_case_4} in the non-ending arrow case, we consider two cases depending on the sign of $\omega - \theta = \Xi+1$ as follows. 
\begin{enumerate}[label=\roman*$)$, ref=\roman*$)$, widest=iii]
    \item \label{it: epsilon>0_eta>0_lambda<0_case_1} \textbf{Suppose 
    $  \omega - \theta \leq 0$.} 
   Set \[\alpha_0 = a_0 | x'_{\frac{1}{2}}, \qquad \alpha_1 = a_0 \Big| \sum_{i=1}^{n-1}  x_{i + \frac{1}{2}} Z^{\ell (n-1-i)} + a_1 | x'_{-\frac{1}{2}} Z^ \theta \]
  and we claim:
  \begin{enumerate}[label=(\alph*), ref=\alph*]
 \item \label{it: alpha0_cycle_epsilon>0_eta>0_lambda<0_case_1}$\partial \alpha_0 =0$;
    \item \label{it: Zn_claim_epsilon>0_eta>0_lambda<0_case_1}$\partial \alpha_1 =  \alpha_0 Z^{\ell (n-1) + \theta}$;
    \item \label{it: shortest_arrow_claim_epsilon>0_eta>0_lambda<0_case_1} $\{\alpha_0,\alpha_1\}$ generates a direct summand of the chain complex $\bX^{\diamond}(P,K,\lambda)$. 
\end{enumerate} 

Since $\Phi^{-K}$ vanishes for both $\alpha_0$ and $\alpha_1$ by Equation~\eqref{eq: Phi-K_arrow},    
the claims \eqref{it: alpha0_cycle_epsilon>0_eta>0_lambda<0_case_1} and \eqref{it: Zn_claim_epsilon>0_eta>0_lambda<0_case_1} are straightforward to verify. 
To prove Claim \eqref{it: shortest_arrow_claim_epsilon>0_eta>0_lambda<0_case_1}, observing that 
\begin{align*}
\partial(a_0|x_{\frac{1}{2}}) &=  p|w_{\frac{1}{2}+\lambda-2\tau(K)} Z^{\omega} + a_0|x'_{-\frac{1}{2}} Z^{\ell + \theta} + a_0|x'_{\frac{1}{2}} Z^{\theta}\\
&= \big( p|w_{\frac{1}{2}+\lambda-2\tau(K)}  + a_0|(x'_{-\frac{1}{2}} Z^{\kappa}   +  x'_{\frac{1}{2}} Z^{\theta-\omega}) \big) Z^{\omega},
\end{align*}
we perform a change of basis as follows:
\begin{align*}
a_0 | x_{n - \frac{1}{2}} &\mapsto \alpha_1 \\
p|w_{\frac{1}{2}+\lambda-2\tau(K)} &\mapsto p|w_{\frac{1}{2}+\lambda-2\tau(K)}  +   a_0 |( x'_{-\frac{1}{2}}Z^{\kappa} +   x'_{\frac{1}{2}} Z^{ \theta - \omega})\\
   a_0 | x'_{i-\frac{1}{2}} &\mapsto a_0 | (x'_{i-\frac{1}{2}} +  x'_{i-\frac{3}{2}} Z^{\ell })     \qquad i= 2,\dots, n-1.
\end{align*}
The remainder of the basis is unchanged. In the resulting basis, by Lemma \ref{lem: horizontal_chain}, the only arrow pointing to $\alpha_0$ is from $\alpha_1$. Claim \eqref{it: shortest_arrow_claim_epsilon>0_eta>0_lambda<0_case_1} follows. We conclude that
  $\Ord(P(K,\lambda)) \geq \ell(n-1) + \theta$.
   \item \label{it: epsilon>0_eta>0_lambda<0_case_2} \textbf{ Suppose 
    $   \omega - \theta > 0$.}
    Set 
    \[\alpha_0 = p|w_{\frac{1}{2}+\lambda - 2\tau(K)} +  a_0 | x'_{-\frac{1}{2}} Z^\kappa, \qquad   \alpha_1 =    a_0 \Big| \sum_{i=1}^{n}  x_{i - \frac{1}{2}} Z^{\ell (n-i)} + a_1 | x'_{-\frac{1}{2}} Z^ \theta\]
    and we claim:
      \begin{enumerate}[label=(\alph*), ref=\alph*]
 \item \label{it: alpha0_cycle_epsilon>0_eta>0_lambda<0_case_2}$\partial \alpha_0 Z^{\Xi}=0$;
    \item \label{it: Zn_claim_epsilon>0_eta>0_lambda<0_case_2}$\partial \alpha_1 =  \alpha_0 Z^{\ell (n-1) + \omega}$;
    \item \label{it: shortest_arrow_claim_epsilon>0_eta>0_lambda<0_case_2} $\{\alpha_0,\alpha_1\}$ generates a direct summand of the chain complex $\bX^{\free}(P,K,\lambda)$. 
\end{enumerate} 
Start with Claim \eqref{it: alpha0_cycle_epsilon>0_eta>0_lambda<0_case_2}. 
Since $\delta^1(p)=0,$  the only outgoing arrows from $p|w_{i+\frac{1}{2}}$  are the arrows of $\Phi^{\mu} + \Phi^{-\mu} = \bI | (L_\sigma+L_\tau).$ By Lemma \ref{lem: C_S_structure_maps},
\[
(L_\sigma+L_\tau)(w_{i+\frac{1}{2}}) = \epsilon \cdot \tilde{w}'_{i+\frac{1}{2}} Z^{\xi'-\Xi}  + w'_{i-\frac{1}{2}} Z^\kappa \qquad i
\in \bZ
\] 
where $\epsilon \in \{0,1\}$ is nonzero if and only if $\Xi>0.$ Therefore 
\begin{equation}
\partial(p|w_{i+\frac{1}{2}}) = \epsilon \cdot p|\tilde{w}'_{i+\frac{1}{2}} Z^{\xi'-\Xi}  + p|w'_{i-\frac{1}{2}} Z^\kappa \qquad i
\in \bZ
\end{equation}
and we compute
\begin{align*}
\partial \alpha_0 &= \partial(p|w_{\frac{1}{2}+\lambda - 2\tau(K)}) +  \partial(a_0 | x'_{-\frac{1}{2}}) Z^\kappa \\
&=  \epsilon \cdot p|\tilde{w}'_{\frac{1}{2} + \lambda - 2\tau(K)} Z^{\xi'-\Xi}  + p|w'_{-\frac{1}{2} + \lambda - 2\tau(K)} Z^\kappa + p|w'_{-\frac{1}{2} + \lambda - 2\tau(K)} Z^\kappa\\
&= \epsilon \cdot p|\tilde{w}'_{\frac{1}{2} + \lambda - 2\tau(K)} Z^{\xi'-\Xi}.
\end{align*}
Since $\tilde{w}'_{\frac{1}{2} + \lambda - 2\tau(K)} Z^{\xi'} =0$ in $\cS^{\diamond}$, Claim \eqref{it: alpha0_cycle_epsilon>0_eta>0_lambda<0_case_2} follows.

To prove Claim \eqref{it: Zn_claim_epsilon>0_eta>0_lambda<0_case_2}, we compute
\begin{align*}
    \partial \alpha_1 &=    \sum_{i=1}^n  \partial (a_0 | x_{i - \frac{1}{2}}) Z^{\ell (n-i)} + \partial ( a_1 | x'_{-\frac{1}{2}}) Z^ \theta    \\
    &= \Big( \partial (a_0 | x_{\frac{1}{2}}) Z^{\ell (n-1)} +  \sum_{i=2}^n  \partial (a_0 | x_{i - \frac{1}{2}}) Z^{\ell (n-i)} \Big) + \partial ( a_1 | x'_{-\frac{1}{2}}) Z^ \theta   \\
    &= \big( a_0|x'_{\frac{1}{2}} Z^{\theta} + p|w_{\frac{1}{2}+\lambda-2\tau(K)} Z^{\omega} + a_0|x'_{-\frac{1}{2}} Z^{\ell + \theta} \big) Z^{\ell (n-1)} + a_0 | x'_{\frac{1}{2}}Z^{\ell (n-1) + \theta}\\
     &= p|w_{\frac{1}{2}+\lambda-2\tau(K)} Z^{\ell (n-1)+\omega} + a_0|x'_{-\frac{1}{2}} Z^{\ell n + \theta}  \\
     &=  \alpha_0 Z^{\ell (n-1)+\omega}.
\end{align*}
To prove Claim \eqref{it: shortest_arrow_claim_epsilon>0_eta>0_lambda<0_case_2}, perform the change of basis 
\begin{align*}
p|w_{\frac{1}{2}+\lambda-2\tau(K)} &\mapsto \alpha_0 \hspace{3em}
a_0 | x_{n - \frac{1}{2}} \mapsto \alpha_1 \\
   a_0 | x'_{i-\frac{1}{2}} &\mapsto a_0 | (x'_{i-\frac{1}{2}} +  x'_{i-\frac{3}{2}} Z^{\ell })     \qquad i= 2,\dots, n-1\\
   a_0 | x'_{\frac{1}{2}} &\mapsto  a_0 | x'_{\frac{1}{2}}  + p|w_{\frac{1}{2}+\lambda-2\tau(K)} Z^{\omega-\theta } +    a_0 | x'_{-\frac{1}{2}}Z^{\ell}.
\end{align*}
    The remainder of the basis is unchanged.
By Lemma \ref{lem: horizontal_chain},  $\{ a_0 | x_{i-\frac{1}{2}},  a_0 |(x'_{i-\frac{1}{2}} + x'_{i-\frac{3}{2}} Z^{\ell })\}$ for $i=2,\dots,n-1$ generates a direct summand of the chain complex $\bX^{\diamond}(P,K,\lambda)$.  

 Similarly, $\{a_0|x_{\frac{1}{2}}, a_0|x'_{\frac{1}{2}}  + p|w_{\frac{1}{2}+\lambda-2\tau(K)} Z^{\omega -\theta} + a_0|x'_{-\frac{1}{2}} Z^{\ell}\}$  generates a direct summand of the chain complex $\bX^{\diamond}(P,K,\lambda)$. 
Since $\frac{1}{2}+\lambda - 2\tau(K)<0$,  by Lemma \ref{lem: middlerow_chain}, $p|w_{\frac{1}{2} + \lambda - 2\tau(K)}$ is not in the image of $\Phi^K$. It follows that
in the resulting basis,
     the only arrow pointing to  $\alpha_0$ comes from $\alpha_1$. Hence Claim \eqref{it: shortest_arrow_claim_epsilon>0_eta>0_lambda<0_case_2} follows from $\partial^2=0$ and the fact that $\alpha_0$ is free. 

Having proved all three claims, we now apply Lemma \ref{lem: torsion_order_criterion} to compute the torsion order. Letting $A= \bX^{\free}(P,K,\lambda), B= \bX^{\Tor}(P,K,\lambda)$ and $f=\partial$,
we conclude that $\Ord(P(K,\lambda)) \geq   \ell (n-1) + \omega -\Xi = \ell (n - 1) + \theta + 1.$
\item[ii$')$]\makeatletter\protect\def\@currentlabel{ii$')$}\makeatother \label{it: epsilon>0_eta>0_lambda<0_specialcase}
\textbf{Suppose $\omega-\theta>0$ and $\lambda -2\tau(K)=-1.$} In this special case of \ref{it: epsilon>0_eta>0_lambda<0_case_2},  the bound can be improved by working over $\bX^{\diamond}(P,K,\lambda)/Y$. Following the same construction as in subcase \ref{it: epsilon>0_eta>0_lambda<0_case_2}, since $p|\scM_{-\frac{1}{2},\frac{\ell-3}{2}} $ lies in $Y$,  we have
    \begin{enumerate}[label=(\alph*{$'$}),ref=\alph*{$'$}]
 \item \label{it: alpha0_cycle_epsilon>0_eta>0_lambda<0_specialcase}$\bar{\partial} \alpha_0=0$;
    \item  \label{it: Zn_claim_epsilon>0_eta>0_lambda<0_specialcase}$\partial \alpha_1 =  \alpha_0 Z^{\ell (n-1) + \omega}$ (and hence the same holds for $\bar \d \a_1$);
    \item \label{it: shortest_arrow_claim_epsilon>0_eta>0_lambda<0_specialcase} $\{\alpha_0,\alpha_1\}$ generates a direct summand of the chain complex $\bX^{\diamond}(P,K,\lambda)/Y$. 
    \end{enumerate}
  We conclude that $\Ord(P(K,\lambda)) \geq   \ell (n-1) + \omega.$
     \end{enumerate}

\item \label{it: epsilon>0_eta>0_lambda-2tau>-1} \textbf{When $\lambda - 2\tau(K) \geq 0$.}
We work over $\bX^{\diamond}(P,K,\lambda)/Y$. 
Set \begin{align*}
\alpha_0 = a_0 | x'_{-\frac{1}{2}}, \qquad \alpha_1 =       a_0 \Big| \sum_{i=1}^n x_{i - \frac{1}{2}} Z^{\ell (n-i)}  + a_1 | x'_{-\frac{1}{2}} Z^ \theta.\end{align*}
Since $\frac{1}{2}+\lambda-2\tau(K)>0,$ in $\bX^{\diamond}(P,K,\lambda)/Y$,  we have \begin{align*}
\Phi^{-K}(a_0|x'_{-\frac{1}{2}})  = \Phi^{-K}(a_0|x_{\frac{1}{2}}) =  0. \end{align*}
As a result,   each term  appearing in $ \alpha_0 $ and $ \alpha_1$ is mapped by $\bar\partial$ into $\bigoplus_s E_{s,\frac{\ell-1}{2}}\oplus F_{s,\frac{\ell-1}{2}}.$
We omit the proofs of the following claims, as they are essentially the same as that of Case \eqref{it: non_ending_general_case_1} in the proof of Proposition \ref{prop: non_ending_general_case}.
      \begin{enumerate}[label=(\alph*), ref=\alph*]
 \item \label{it: alpha0_cycle_epsilon>0_eta>0_lambda>0}$\bar\partial \alpha_0 =0$;
    \item \label{it: Zn_claim_epsilon>0_eta>0_lambda>0}$\bar\partial \alpha_1 =  \alpha_0 Z^{\ell n + \theta}$;
    \item \label{it: shortest_arrow_claim_epsilon>0_eta>0_lambda>0} $\{\alpha_0,\alpha_1\}$ generates a direct summand of the chain complex  $\bX^{\diamond}(P,K,\lambda)/Y$. 
\end{enumerate} 
We conclude that $\Ord(P(K,\lambda)) \geq \ell n + \theta.$ 
\end{enumerate}

\medskip
\noindent\textbf{Case \eqref{it: ending_general_epsilon>0_case_2}: when $\varepsilon(K)=1$ and $\eta<0$.}   We have  $\delta^1 (a_1) = a'_1 \otimes W^{-\eta} + a_0 \otimes Z^{n}$. See Figure~\ref{fig: prop_ending_epsilon>0}\subref{subfig: prop_ending_epsilon>0_eta<0} for an example when $\lambda - 2\tau(K) = -1$ and $\omega-\theta>0$.

The proof follows by combining the arguments from the previous case with those of
Case \eqref{it: non_ending_general_case_3} in the proof of Proposition \ref{prop: non_ending_general_case};  
we sketch only the main steps.
\begin{enumerate}[label=(2.\arabic*), ref=(2.\arabic*)]
    \item \label{it: epsilon>0_eta<0_lambda-2tau<0} \textbf{When $\lambda - 2\tau(K) < 0$.}
    Consider two cases depending on the sign of $\omega - \theta = \Xi+1$ as follows.
    \begin{enumerate}[label=\roman*$)$, ref=\roman*$)$]
        \item \label{it: epsilon>0_eta<0_lambda<0_case_1} \textbf{Suppose 
    $  \omega - \theta \leq 0$.} This is the case where the assumption $n>1$ is necessary.  
Set 
\begin{align*}
\alpha_0 &= a_0|x'_{\frac{1}{2}},\\
\alpha_1 &= \left( a_0 \Big| \sum_{i=1}^{n-1}  x_{i + \frac{1}{2}} Z^{\ell (n-1-i)} + a_1|x'_{-\frac{1}{2}} Z^ \theta  \right) Z^{\max \{\kappa - \theta, 0 \} } + a'_1|x_{\eta+\frac{1}{2}} Z^{\max \{ \theta - \kappa, 0 \}}
\end{align*}
  and we have:
  \begin{enumerate}[label=(\alph*), ref=\alph*]
 \item \label{it: alpha0_cycle_epsilon>0_eta<0_lambda<0_case_1}$\partial \alpha_0 =0$;
    \item \label{it: Zn_claim_epsilon>0_eta<0_lambda<0_case_1}$\partial \alpha_1 =  \alpha_0 Z^{\ell (n-1) + \max\{\kappa,\theta \}}$;
    \item \label{it: shortest_arrow_claim_epsilon>0_eta<0_lambda<0_case_1} $\{\alpha_0,\alpha_1\}$ generates a direct summand of the chain complex $\bX^{\diamond}(P,K,\lambda)$. 
\end{enumerate} 
       We conclude that
        $\Ord(P(K,\lambda))  \geq \ell (n-1) + \max\{\kappa,\theta \}.$ 
\item \label{it: epsilon>0_eta<0_lambda<0_case_2} \textbf{Suppose 
    $  \omega - \theta > 0$.}  
Set
       \begin{align*}
    \alpha_0 &= p|w_{\frac{1}{2}+\lambda-2\tau(K)} +  a_0|x'_{-\frac{1}{2}}Z^\kappa, \\
    \alpha_1 &=  \left( a_0 \Big| \sum_{i=1}^{n}  x_{i - \frac{1}{2}} Z^{\ell (n-i)} + a_1|x'_{-\frac{1}{2}} Z^ \theta \right) Z^{\max \{\kappa-\theta, 0 \}} + a'_1|x_{\eta+\frac{1}{2}} Z^{\max \{\theta - \kappa, 0 \}}.
\end{align*}
Then we have
\begin{enumerate}[label=(\alph*), ref=\alph*]
\item $\partial \alpha_0 Z^{\Xi} = 0$;
    \item  $\partial^{\free} \alpha_1 =  \alpha_0 Z^{\ell (n-1) + \omega - \theta + \max\{\kappa, \theta \}}$;
    \item  $\{ \alpha_0, \alpha_1 \}$ generates a summand in $\bX^{\free}(P,K,\lambda)$. 
\end{enumerate}  
 It follows from Lemma \ref{lem: torsion_order_criterion} that
        $\Ord(P(K,\lambda))  \geq \ell (n-1) + \omega - \theta + \max\{\kappa, \theta \} - \Xi = \ell (n-1)  + \max\{\kappa, \theta \} + 1.$ 
\item[ii$')$]\makeatletter\protect\def\@currentlabel{ii$')$}\makeatother \label{it: epsilon>0_eta<0_lambda<0_specialcase} \textbf{Suppose 
    $  \omega - \theta > 0$ and $\lambda -2 \tau(K) =-1$.}  This is a special case of \ref{it: epsilon>0_eta<0_lambda<0_case_2}.
   Following the same construction, and working over $\bX^{\diamond}(P,K,\lambda)/Y$, we have
\begin{enumerate}[label=(\alph*{$'$}),ref=\alph*{$'$}]
\item $\bar\partial \alpha_0 = 0$;
    \item  $\bar{\partial}^{\free} \alpha_1 =  \alpha_0 Z^{\ell (n-1) + \omega - \theta + \max\{\kappa, \theta \}}$, where $\bar{\partial}^{\free}$ is the differential on $\bX^{\free}(P,K,\lambda)/Y^{\free}$;
    \item  $\{ \alpha_0, \alpha_1 \}$ generates a summand in $\bX^{\free}(P,K,\lambda)/Y^{\free}$. 
\end{enumerate}  
In Lemma \ref{lem: torsion_order_criterion}, by setting $A=\bX^{\free}(P,K,\lambda)/Y^{\free}$ and $B=\bX^{\Tor}(P,K,\lambda)/Y^{\Tor}$,
we obtain that
        $\Ord(P(K,\lambda))  \geq \ell (n-1) + \omega - \theta + \max\{\kappa, \theta \}.$ 
    \end{enumerate}
\item \label{it: epsilon>0_eta<0_lambda-2tau>-1} \textbf{When $\lambda - 2\tau(K) \geq 0$.} Working over $\bX^{\diamond}(P,K,\lambda)/Y$, the proof is essentially the same as that of Case \eqref{it: non_ending_general_case_3} in the proof of Proposition \ref{prop: non_ending_general_case}.
Set   
\begin{align*}
\alpha_0 &= a_0 | x'_{-\frac{1}{2}},  \\
     \alpha_1 &= \left(    a_0 \Big| \sum_{i=1}^n    x_{i - \frac{1}{2}} Z^{\ell (n-i)} + a_1 | x'_{-\frac{1}{2}}  Z^ \theta \right) Z^{\max\{\kappa - \theta, 0 \}} +  a'_1|x_{\eta+\frac{1}{2}}Z^{\max\{\theta - \kappa, 0 \}}. 
\end{align*}
Since $\frac{1}{2}+\lambda-2\tau(K)>0,$ $\Phi^{-K}$ vanishes for each term in $\alpha_0$ and $\alpha_1$ in $\bX^{\diamond}(P,K,\lambda)/Y$. We obtain
\begin{enumerate}[label=(\alph*), ref=\alph*]
 \item $\bar\partial \alpha_0 =0$;
    \item $\partial^{\free} \alpha_1 =  \alpha_0 Z^{\ell n + \max\{\kappa, \theta \}}$;
    \item $\{\alpha_0,\alpha_1\}$ generates a direct summand of the chain complex $\bX^{\free}(P,K,\lambda)/Y^{\free}$,
\end{enumerate} 
In Lemma \ref{lem: torsion_order_criterion}, by setting $A=\bX^{\free}(P,K,\lambda)/Y^{\free}$ and $B=\bX^{\Tor}(P,K,\lambda)/Y^{\Tor}$, we conclude that  $\Ord(P(K,\lambda))  \geq \ell n + \max\{\kappa, \theta \}.$ 
\end{enumerate}
\end{proof}
This concludes the case when $\varepsilon(K) = 1.$
In the following, we consider  the case when $\varepsilon(K) = -1.$

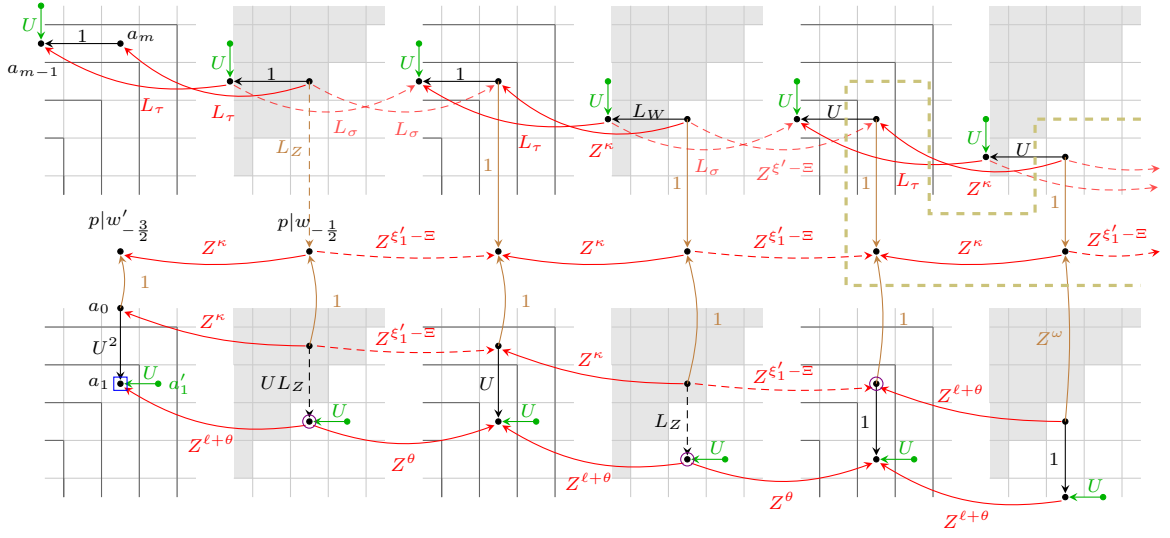
\begin{figure}[hbtp!]
	\captionsetup{width=\textwidth}
	\subfigure[When $\varepsilon(K)=-1, \eta>0$ and $\lambda-2\tau(K) = 1.$]
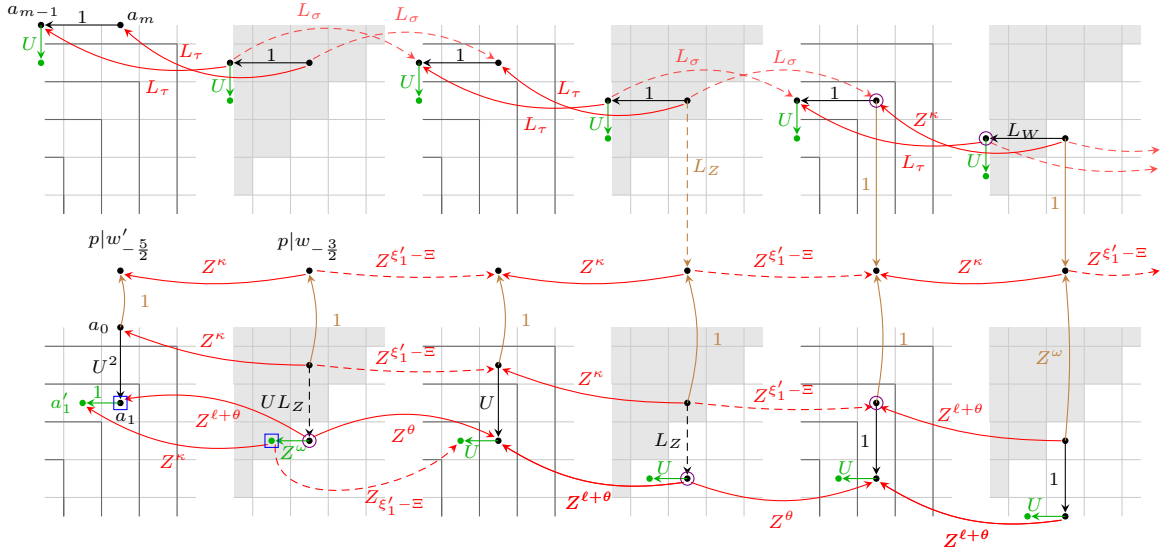
{
		\begin{tikzpicture}[scale=0.5]
			\begin{scope}
				\foreach \k in {0,8}{
					\foreach \j in {-5,5,15}{
						\fill[gray!20] (\j+-1,1.5+\k) rectangle (\j+1.5,0.5+\k);
						\fill[gray!20] (\j+-1,2.5+\k) rectangle (\j+2.5,1.5+\k);
						\fill[gray!20] (\j+-1,3+\k) rectangle (\j+3,2.5+\k);
						\fill[gray!20] (\j+-1,0.5+\k) rectangle (\j+0.5,-0.5+\k);
						\fill[gray!20] (\j+-1,-0.5+\k) rectangle (\j+-0.5,-1.5+\k);
					}
					\foreach \j in {-10,-5,...,10,15}
					{\foreach \i in {0,...,3}
						{\draw[thin, black!20!white]  (\j+\i-0.5, 3+\k) -- (\j+\i-0.5, -2+\k);}
						\foreach \i in {-1,...,3}
						{ \draw[thin, black!20!white]  (\j+3, \i-0.5+\k) -- (\j-1, \i-0.5+\k); }}
					\foreach \j in {-10,0,10}
					{\draw[thin, black!60!white] (\j-1, 2.5+\k) -- (\j+2.5, 2.5+\k);
						\draw[thin, black!60!white] (\j+2.5, 2.5+\k) -- (\j+2.5, -2+\k);
						\draw[thin, black!60!white] (\j+1.5, 1.5+\k) -- (\j+1.5, -2+\k);
						\draw[thin, black!60!white]  (\j+0.5, 0.5+\k) -- (\j+0.5, -2+\k);
						\draw[thin, black!60!white]  (\j+1.5, 1.5+\k) -- (\j-1, 1.5+\k);
						\draw[thin, black!60!white] (\j+0.5, 0.5+\k) -- (\j-1, 0.5+\k);
						\draw[thin, black!60!white] (\j-1, -0.5+\k) -- (\j-0.5, -0.5+\k);
						\draw[thin, black!60!white] (\j-0.5, -0.5+\k) -- (\j-0.5, -2+\k);
						\draw[thin, black!60!white] (\j-0.5, -1.5+\k) -- (\j-0.5, -2+\k);}}
			\end{scope}       
			\foreach \k in {8}{
				\foreach \j in {-1,0,1}
				{ \foreach \i in {0,1}
					{
						{\filldraw ({\j*10+1+5*\i}, 4.5) circle (2pt) node[] (o\i\j) {};}
						\filldraw [color=green!70!black]({\j*10+1+5*\i-2.1}, -\j-\i+\k+2) circle (2pt)
						node[] (am-1'\i\j) {};
						\filldraw ({\j*10+1+5*\i-2.1}, -\j-\i+\k+1) circle (2pt)
						node[] (am-1\i\j) {};
						\filldraw ({\j*10+1+5*\i}, -\j-\i+\k+1) circle (2pt) 
						node[] (am\i\j) {};
						\draw[-stealth] ($(am\i\j)$) -- ($(am-1\i\j)+(0.1,0)$);
						\ifnum \j>-1
						\draw[color=brown, -stealth]   ($(am\i\j)$) to node[midway, left, xshift=2pt]{{\tiny $1$}} ($(o\i\j)+(0,0.1)$);\fi
						\draw [color=green!70!black, -stealth] ($(am-1'\i\j)$) -- ($(am-1\i\j)+(0,0.1)$);
					}
					\draw[red, bend left=20, -stealth] ($(am-11\j)+ (-0.1,-0.1)$) to node[pos=0.38, below] {{\tiny $L_\tau$}} ($(am-10\j) + (0.1,-0.1)$);
					\ifnum\j>-1
					\draw[red, bend left=30, -stealth] ($(am1\j)+ (-0.1,-0.1)$) to node[pos=0.42, below] {{\tiny $Z^\kappa$}} ($(am0\j) + (0.1,-0.1)$);
					\else
					\draw[red, bend left=30, -stealth] ($(am1\j)+ (-0.1,-0.1)$) to node[pos=0.42, below] {{\tiny $L_\tau$}} ($(am0\j) + (0.1,-0.1)$); 
					\fi
					\draw[red, bend left=10, -stealth] ($(o1\j)+ (-0.1,-0.1)$) to node[midway, above] {{\tiny $Z^\kappa$}} ($(o0\j) + (0.1,-0.1)$);
			}}   
			\draw[red, opacity=0.7, bend right=30, densely dashed, -stealth] ($(am1-1)+ (0.1,-0.1)$) to node[pos=0.52, below] {{\tiny $L_\sigma$}} ($(am00) + (-0.1,-0.1)$);
			\draw[red, opacity=0.7, bend right=30, densely dashed, -stealth] ($(am-11-1)+ (0.1,-0.1)$) to node[pos=0.62, below] {{\tiny $L_\sigma$}} ($(am-100) + (-0.1,-0.1)$);
			\draw[red, opacity=0.7, bend right=30, densely dashed, -stealth] ($(am10)+ (0.1,-0.1)$) to node[pos=0.52, below] {{\tiny $Z^{\xi'-\Xi}$}} ($(am01) + (-0.1,-0.1)$);
			\draw[red, opacity=0.7, bend right=30, densely dashed, -stealth] ($(am-110)+ (0.1,-0.1)$) to node[pos=0.53, below] {{\tiny $L_\sigma$}} ($(am-101) + (-0.1,-0.1)$);
			\draw[red, opacity=0.7, bend right=10, densely dashed, -stealth] ($(am11)+ (0.1,-0.1)$) to  ($(am11)+ (2.5,-0.3)$);
			\draw[red, opacity=0.7, bend right=15, densely dashed, -stealth] ($(am-111)+ (0.1,-0.1)$) to  ($(am-111) + (4.55,-0.8)$);
			
			\draw[red, bend right=5, densely dashed, -stealth] ($(o1-1)+ (0.2,0)$) to node[pos=0.5, above] {{\tiny $Z^{\xi'_1 -\Xi}$}} ($(o00) + (-0.1,-0.1)$);
			\draw[red, bend right=5, densely dashed, -stealth] ($(o10)+ (0.2,0)$) to node[pos=0.5, above] {{\tiny $Z^{\xi'_1 -\Xi}$}} ($(o01) + (-0.1,-0.1)$);
			\foreach \j in {-1,0,1}
			{ \foreach \i in {0,1}
				{\filldraw ({\j*10+1+5*\i}, 2-\j-\i) circle (2pt)
					node[] (a1\i\j) {};
					\filldraw ({\j*10+1+5*\i}, -\j-\i) circle (2pt) 
					node[] (a0\i\j) {};
					\ifnum\j<1       \ifnum \i=1         \draw[densely dashed,-stealth] ($(a1\i\j)$) -- ($(a0\i\j)+(0,0.1)$);      \else \draw[-stealth] ($(a1\i\j)$) -- ($(a0\i\j)+(0,0.1)$); \fi \else \draw[-stealth] ($(a1\i\j)$) -- ($(a0\i\j)+(0,0.1)$); \fi
					\filldraw [color=green!70!black] ({\j*10+2+5*\i}, -\j-\i) circle (2pt) node[] (a0'\i\j) {};
					\draw [color=green!70!black, -stealth] ($(a0'\i\j)$) -- ($(a0\i\j)+(0.1,0)$);
				}
				\draw[red, bend left=20, -stealth] ($(a01\j)+ (-0.1,-0.1)$) to node[midway, below] {{\tiny $Z^{\ell+\theta}$}} ($(a00\j) + (0.1,-0.1)$);
			}    
			\foreach \j in {-1,0,1}
			{ \foreach \i in {0,1}
				{
					\ifnum\j<1
					\draw[color=brown, bend right=15,  -stealth]   ($(a1\i\j)$) to node[midway, right, xshift=2pt]{{\tiny $1$}} ($(o\i\j)+(0,-0.1)$); \fi}}
			\draw[color=brown,densely dashed, -stealth]   ($(am1-1)$) to node[pos=0.4, left, xshift=2pt]{{\tiny $L_Z$}} ($(o1-1)+(0,0.1)$);
			\draw[color=red,densely dashed, bend right=10, -stealth]   ($(o11)+(0.1,0)$) to node[pos=0.6,  above, xshift=-2pt]{{\tiny $Z^{\xi'_1 -\Xi}$}} ($(o11)+(2.5,0)$);
			\draw[color=brown, bend right=10,  -stealth]   ($(a101)$) to node[midway, right, xshift=2pt]{{\tiny $1$}} ($(o01)+(0,-0.1)$);
			\draw[color=brown, bend right=5, -stealth]   ($(a111)$) to node[midway, left, xshift=2pt]{{\tiny $Z^ \omega $}} ($(o11)+(0,-0.1)$);
			\draw[red, bend right=20, -stealth] ($(a010) + (0.1,-0.1)$)  to node[midway, below] {{\tiny $Z^ \theta$}} ($(a001)+ (-0.1,-0.1)$);
			\draw[red, bend right=20, -stealth] ($(a01-1) + (0.1,-0.1)$)  to node[midway, below] {{\tiny $Z^ \theta$}} ($(a000)+ (-0.1,-0.1)$);
			
			\draw[red, bend left=10, -stealth] ($(a111) + (-0.1,0)$)  to node[pos=0.55, above] {{\tiny $Z^{\ell+\theta}$}} ($(a101)+ (0.1,-0.1)$);
			\draw[red, bend right=5, densely dashed, -stealth] ($(a110)+ (0.2,0)$) to node[pos=0.5, above,yshift=-1pt] {{\tiny $Z^{\xi'_1 -\Xi}$}} ($(a101) + (-0.1,-0.1)$);
			\draw[red, bend right=5, densely dashed, -stealth] ($(a11-1)+ (0.2,0)$) to node[pos=0.5, above,yshift=-1pt] {{\tiny $Z^{\xi'_1 -\Xi}$}} ($(a100) + (-0.1,-0.1)$);
			\draw[red, bend left=10, -stealth] ($(a11-1)+ (-0.1,0)$) to node[midway, above] {{\tiny $Z^\kappa$}} ($(a10-1) + (0.1,-0.1)$);
			\draw[red, bend left=10, -stealth] ($(a110)+ (-0.1,0)$) to node[midway, above] {{\tiny $Z^\kappa$}} ($(a100) + (0.1,-0.1)$);
			
			\node [left] at (a00-1) {{\tiny$a_1$}};
			\node [left] at (a10-1) {{\tiny$a_0$}};
			\node [color=green!70!black,right] at (a0'0-1) {{\tiny$a'_1$}};
			\node [above] at (o0-1) {{\tiny$p|w'_{-\frac{3}{2}}$}};
			\node [above] at (o1-1) {{\tiny$p|w_{-\frac{1}{2}}$}};
			\node [right] at ($(am0-1)+(-0.1,0.2)$) {{\tiny$a_m$}};
			\node [below] at ($(am-10-1)+(-0.2,-0.3)$) {{\tiny$a_{m-1}$}};
			
			\node [] at ($(a00-1)+(-.4,1)$) {{\tiny$U^2$}};
			\node [] at ($(a01-1)+(-.7,1)$) {{\tiny$UL_Z$}};
			\node [] at ($(a000)+(-.3,1)$) {{\tiny$U$}};
			\node [] at ($(a010)+(-.5,1)$) {{\tiny$L_Z$}};
			\node [] at ($(a001)+(-.3,1)$) {{\tiny$1$}};
			\node [] at ($(a011)+(-.3,1)$) {{\tiny$1$}};
			\foreach \j in {-1,0,1}
			{ \foreach \i in {0,1}{
					\node [color=green!70!black] at ($(am-1\i\j)+(-0.3,0.5)$) {{\tiny$U$}};
					\node [color=green!70!black] at ($(a0\i\j)+(0.8,0.3)$) {{\tiny$U$}};
			}}
			\node [] at ($(am0-1)+(-1,.2)$) {{\tiny$1$}};
			\node [] at ($(am1-1)+(-1,.2)$) {{\tiny$1$}};
			\node [] at ($(am00)+(-1,.2)$) {{\tiny$1$}};
			\node [] at ($(am10)+(-1,.2)$) {{\tiny$L_W$}};
			\node [] at ($(am01)+(-1.1,.2)$) {{\tiny$U$}};
			\node [] at ($(am11)+(-1.1,.2)$) {{\tiny$U$}};  
			\draw[blue] 
			($(a00-1)+(-5pt,-5pt)$) rectangle ($(a00-1)+(5pt,5pt)$);
			\draw[violet] (a01-1) circle (5pt);
			\draw[violet] (a010) circle (5pt);
			\draw[violet] (a101) circle (5pt);
			\draw[yellow!50!gray, very thick, dashed]
			(18,8) -- (15.2,8) -- (15.2,5.5) -- (12.4,5.5) -- (12.4,9) -- (10.2,9) -- (10.2,3.6) -- (18,3.6);
		\end{tikzpicture}    \label{subfig: prop_ending_epsilon<0_eta>0}
	}
	\subfigure[When $\varepsilon(K)=-1, \eta<0$ and $\lambda-2\tau(K) = 0.$]{
		\begin{tikzpicture}[scale=0.5]
		\begin{scope}
			\foreach \k in {0,8}{
				\foreach \j in {-5,5,15}{
					\fill[gray!20] (\j+-1,1.5+\k) rectangle (\j+1.5,0.5+\k);
					\fill[gray!20] (\j+-1,2.5+\k) rectangle (\j+2.5,1.5+\k);
					\fill[gray!20] (\j+-1,3+\k) rectangle (\j+3,2.5+\k);
					\fill[gray!20] (\j+-1,0.5+\k) rectangle (\j+0.5,-0.5+\k);
					\fill[gray!20] (\j+-1,-0.5+\k) rectangle (\j+-0.5,-1.5+\k);
				}
				\foreach \j in {-10,-5,...,10,15}
				{\foreach \i in {0,...,3}
					{\draw[thin, black!20!white]  (\j+\i-0.5, 3+\k) -- (\j+\i-0.5, -2+\k);}
					\foreach \i in {-1,...,3}
					{ \draw[thin, black!20!white]  (\j+3, \i-0.5+\k) -- (\j-1, \i-0.5+\k); }}
				\foreach \j in {-10,0,10}
				{\draw[thin, black!60!white] (\j-1, 2.5+\k) -- (\j+2.5, 2.5+\k);
					\draw[thin, black!60!white] (\j+2.5, 2.5+\k) -- (\j+2.5, -2+\k);
					\draw[thin, black!60!white] (\j+1.5, 1.5+\k) -- (\j+1.5, -2+\k);
					\draw[thin, black!60!white]  (\j+0.5, 0.5+\k) -- (\j+0.5, -2+\k);
					\draw[thin, black!60!white]  (\j+1.5, 1.5+\k) -- (\j-1, 1.5+\k);
					\draw[thin, black!60!white] (\j+0.5, 0.5+\k) -- (\j-1, 0.5+\k);
					\draw[thin, black!60!white] (\j-1, -0.5+\k) -- (\j-0.5, -0.5+\k);
					\draw[thin, black!60!white] (\j-0.5, -0.5+\k) -- (\j-0.5, -2+\k);
					\draw[thin, black!60!white] (\j-0.5, -1.5+\k) -- (\j-0.5, -2+\k);}}
		\end{scope}    
			\foreach \k in {8}{
				\foreach \j in {-1,0,1}
				{ \foreach \i in {0,1}
					{
						{\filldraw ({\j*10+1+5*\i}, 4.5) circle (2pt) node[] (o\i\j) {};}
						\filldraw [color=green!70!black]({\j*10+1+5*\i-2.1}, -\j-\i+\k+1) circle (2pt)
						node[] (am-1'\i\j) {};
						\filldraw ({\j*10+1+5*\i-2.1}, -\j-\i+\k+2) circle (2pt)
						node[] (am-1\i\j) {};
						\filldraw ({\j*10+1+5*\i}, -\j-\i+\k+2) circle (2pt) 
						node[] (am\i\j) {};
						\draw[-stealth] ($(am\i\j)$) -- ($(am-1\i\j)+(0.1,0)$);
						\ifnum \j>0
						\draw[color=brown, -stealth]   ($(am\i\j)$) to node[midway, left, xshift=2pt]{{\tiny $1$}} ($(o\i\j)+(0,0.1)$);\fi
						\draw [color=green!70!black, -stealth] ($(am-1\i\j)$) -- ($(am-1'\i\j)+(0,0.1)$);
					}
					\draw[red, bend left=20, -stealth] ($(am-11\j)+ (-0.1,-0.1)$) to node[pos=0.35, below] {{\tiny $L_\tau$}} ($(am-10\j) + (0.1,-0.1)$);
					\ifnum \j>0
					\draw[red, bend left=30, -stealth] ($(am1\j)+ (-0.1,-0.1)$) to node[pos=0.72, above] {{\tiny $Z^{\kappa}$}} ($(am0\j) + (0.1,-0.1)$);
					\else
					\draw[red, bend left=30, -stealth] ($(am1\j)+ (-0.1,-0.1)$) to node[pos=0.6, above] {{\tiny $L_\tau$}} ($(am0\j) + (0.1,-0.1)$); 
					\fi 
					\draw[red, bend left=10, -stealth] ($(o1\j)+ (-0.1,-0.1)$) to node[midway, above] {{\tiny $Z^\kappa$}} ($(o0\j) + (0.1,-0.1)$);
			}}   
			\draw[red, opacity=0.7, bend left=30, densely dashed, -stealth] ($(am1-1)+ (0.1,0.1)$) to node[pos=0.52, above] {{\tiny $L_\sigma$}} ($(am00) + (-0.1,0.1)$);
			\draw[red, opacity=0.7, bend left=35, densely dashed, -stealth] ($(am-11-1)+ (0.1,0.1)$) to node[pos=0.42, above] {{\tiny $L_\sigma$}} ($(am-100) + (-0.1,0.1)$);
			\draw[red, opacity=0.7, bend left=30, densely dashed, -stealth] ($(am10)+ (0.1,0.1)$) to node[pos=0.52, above,yshift=-0.07cm] {{\tiny $L_\sigma$}} ($(am01) + (-0.1,0.1)$);
			\draw[red, opacity=0.7, bend left=30, densely dashed, -stealth] ($(am-110)+ (0.1,0.1)$) to node[pos=0.42, above,yshift=-0.05cm] {{\tiny $L_\sigma$}} ($(am-101) + (-0.1,0.1)$);
			\draw[red, opacity=0.7, bend right=10, densely dashed, -stealth] ($(am11)+ (0.1,-0.1)$) to  ($(am11)+ (2.5,-0.3)$);
			\draw[red, opacity=0.7, bend right=15, densely dashed, -stealth] ($(am-111)+ (0.1,-0.1)$) to  ($(am-111) + (4.55,-0.8)$);
			
			\draw[red, bend right=5, densely dashed, -stealth] ($(o1-1)+ (0.2,0)$) to node[pos=0.5, above] {{\tiny $Z^{\xi'_1 -\Xi}$}} ($(o00) + (-0.1,-0.1)$);
			\draw[red, bend right=5, densely dashed, -stealth] ($(o10)+ (0.2,0)$) to node[pos=0.5, above] {{\tiny $Z^{\xi'_1 -\Xi}$}} ($(o01) + (-0.1,-0.1)$);
			\foreach \j in {-1,0,1}
			{ \foreach \i in {0,1}
				{\filldraw ({\j*10+1+5*\i}, 2-\j-\i) circle (2pt)
					node[] (a1\i\j) {};
					\filldraw ({\j*10+1+5*\i}, -\j-\i) circle (2pt) 
					node[] (a0\i\j) {};
					\ifnum\j<1        \ifnum \i=1         \draw[densely dashed,-stealth] ($(a1\i\j)$) -- ($(a0\i\j)+(0,0.1)$);      \else \draw[-stealth] ($(a1\i\j)$) -- ($(a0\i\j)+(0,0.1)$); \fi \else \draw[-stealth] ($(a1\i\j)$) -- ($(a0\i\j)+(0,0.1)$); \fi
					\filldraw [color=green!70!black] ({\j*10+5*\i}, -\j-\i) circle (2pt) node[] (a0'\i\j) {};
					\draw [color=green!70!black, -stealth] ($(a0\i\j)$) -- ($(a0'\i\j)+(0.1,0)$);
					{ \ifnum \j>-1
						\draw[red, bend left=20, -stealth] ($(a01\j)+ (-0.1,-0.1)$) to node[midway, below] {{\tiny $Z^{\ell+\theta}$}} ($(a00\j) + (0.1,-0.1)$);\fi}
				}
			}    
			\foreach \j in {-1,0,1}
			{ \foreach \i in {0,1}
				{
					\ifnum\j<1
					\draw[color=brown, bend right=15,  -stealth]   ($(a1\i\j)$) to node[midway, right, xshift=2pt]{{\tiny $1$}} ($(o\i\j)+(0,-0.1)$); \fi}}
			\draw[color=brown,densely dashed, -stealth]   ($(am10)$) to node[pos=0.4, right, xshift=-2pt]{{\tiny $L_Z$}} ($(o10)+(0,0.1)$);
			\draw[color=red,densely dashed, bend right=10, -stealth]   ($(o11)+(0.1,0)$) to node[pos=0.6,  above, xshift=-2pt]{{\tiny $Z^{\xi'_1 -\Xi}$}} ($(o11)+(2.5,0)$);
			\draw[color=brown, bend right=10,  -stealth]   ($(a101)$) to node[midway, right, xshift=2pt]{{\tiny $1$}} ($(o01)+(0,-0.1)$);
			\draw[color=brown, bend right=5, -stealth]   ($(a111)$) to node[midway, left, xshift=2pt]{{\tiny $Z^ \omega $}} ($(o11)+(0,-0.1)$);
			\draw[red, bend right=20, -stealth] ($(a010) + (0.1,-0.1)$)  to node[midway, below] {{\tiny $Z^ \theta$}} ($(a001)+ (-0.1,-0.1)$);
			\draw[red, bend left=25, -stealth] ($(a01-1) + (0.1,0.1)$)  to node[midway, below] {{\tiny $Z^ \theta$}} ($(a000)+ (-0.1,0.1)$);
			\draw[red, bend right=20, -stealth] ($(a01-1)+ (-0.1,0.1)$) to node[midway, below] {{\tiny $Z^{\ell+\theta}$}} ($(a00-1) + (0.1,0.1)$);
			\draw[red, bend left=20, -stealth] ($(a0'1-1)+ (-0.1,-0.1)$) to node[midway, below] {{\tiny $Z^{\kappa}$}} ($(a0'0-1) + (0.1,-0.1)$);
			
			\draw[red, bend left=10, -stealth] ($(a111) + (-0.1,0)$)  to node[pos=0.55, above] {{\tiny $Z^{\ell+\theta}$}} ($(a101)+ (0.1,-0.1)$);
			\draw[red, bend right=5, densely dashed, -stealth] ($(a110)+ (0.2,0)$) to node[pos=0.5, above,yshift=-1pt] {{\tiny $Z^{\xi'_1 -\Xi}$}} ($(a101) + (-0.1,-0.1)$);
			\draw[red, bend right=5, densely dashed, -stealth] ($(a11-1)+ (0.2,0)$) to node[pos=0.5, above,yshift=-1pt] {{\tiny $Z^{\xi'_1 -\Xi}$}} ($(a100) + (-0.1,-0.1)$);
			\draw[red, bend left=10, -stealth] ($(a11-1)+ (-0.1,0)$) to node[midway, above] {{\tiny $Z^\kappa$}} ($(a10-1) + (0.1,-0.1)$);
			\draw[red, bend left=10, -stealth] ($(a110)+ (-0.1,0)$) to node[midway, above] {{\tiny $Z^\kappa$}} ($(a100) + (0.1,-0.1)$);
			
			\draw[red, densely dashed, in=-135, out=-95, -stealth] ($(a0'1-1)+ (0.1,-0.1)$) to node[pos=0.7, below] {{\tiny $Z_{\xi'_1-\Xi}$}} ($(a0'00) + (-0.1,-0.1)$);
			
			\node [below] at ($(a00-1)+(0.15,0)$) {{\tiny$a_1$}};
			\node [left] at (a10-1) {{\tiny$a_0$}};
			\node [color=green!70!black,left] at (a0'0-1) {{\tiny$a'_1$}};
			\node [above] at (o0-1) {{\tiny$p|w'_{-\frac{5}{2}}$}};
			\node [above] at (o1-1) {{\tiny$p|w_{-\frac{3}{2}}$}};
			\node [right] at ($(am0-1)+(-0.1,0.2)$) {{\tiny$a_m$}};
			\node [above] at ($(am-10-1)+(-0.2,-0.1)$) {{\tiny$a_{m-1}$}};
			
			\node [] at ($(a00-1)+(-.4,1)$) {{\tiny$U^2$}};
			\node [] at ($(a01-1)+(-.7,1)$) {{\tiny$UL_Z$}};
			\node [] at ($(a000)+(-.3,1)$) {{\tiny$U$}};
			\node [] at ($(a010)+(-.5,1)$) {{\tiny$L_Z$}};
			\node [] at ($(a001)+(-.3,1)$) {{\tiny$1$}};
			\node [] at ($(a011)+(-.3,1)$) {{\tiny$1$}};
			\node [color=green!70!black] at ($(a00-1)+(-0.6,0.3)$) {{\tiny$1$}};
			\node [color=green!70!black] at ($(a01-1)+(-0.4,-0.3)$) {{\tiny$Z^{\omega}$}};
			\node [color=green!70!black] at ($(a000)+(-0.7,-0.3)$) {{\tiny$U$}};
			\node [color=green!70!black] at ($(a010)+(-0.6,0.3)$) {{\tiny$U$}};
			\node [color=green!70!black] at ($(a001)+(-0.8,0.3)$) {{\tiny$U$}};
			\node [color=green!70!black] at ($(a011)+(-0.8,0.3)$) {{\tiny$U$}};
			\foreach \j in {-1,0,1}
			{
				\node [color=green!70!black] at ($(am-10\j)+(-0.3,-0.5)$) {{\tiny$U$}};
				\node [color=green!70!black] at ($(am-11\j)+(-0.3,-0.6)$) {{\tiny$U$}};
			}
			\node [] at ($(am0-1)+(-1,.2)$) {{\tiny$1$}};
			\node [] at ($(am1-1)+(-1,.2)$) {{\tiny$1$}};
			\node [] at ($(am00)+(-1,.2)$) {{\tiny$1$}};
			\node [] at ($(am10)+(-1,.2)$) {{\tiny$1$}};
			\node [] at ($(am01)+(-1.1,.2)$) {{\tiny$1$}};
			\node [] at ($(am11)+(-1.1,.2)$) {{\tiny$L_W$}};  
			\draw[blue] 
			($(a00-1)+(-5pt,-5pt)$) rectangle ($(a00-1)+(5pt,5pt)$);
			\draw[blue] 
			($(a0'1-1)+(-5pt,-5pt)$) rectangle ($(a0'1-1)+(5pt,5pt)$);
			\draw[violet] (am-111) circle (5pt);
			\draw[violet] (am01) circle (5pt);
			\draw[violet] (a01-1) circle (5pt);
			\draw[violet] (a010) circle (5pt);
			\draw[violet] (a101) circle (5pt);
		\end{tikzpicture}    \label{subfig: prop_ending_epsilon<0_eta<0}
	}
	\caption{
		A linear sum of the boxed (resp.~circled) elements represents 
		$\alpha_0$ (resp.~$\alpha_1$)  defined in the proof of Proposition \ref{prop: ending_arrow_general_case}. 
	}
	\label{fig: prop_ending_epsilon<0}
\end{figure}

\begin{proof}[Proof of Proposition \ref{prop: ending_arrow_general_case} in the case $\varepsilon(K) = -1$, when $\varepsilon(K)=-1$. ] 
Recall that in this case the ending generator of the standard complex, $a_0$, has an outgoing arrow weighted by $Z^n$ to $a_1$.  The generator $a_1$ in turn admits either an incoming or an outgoing arrow weighted by $W^{|\eta|}$ for some nonzero integer $\eta$.
 See Figure \ref{fig: prop_ending_epsilon<0} for some examples. Similar to the proof in the case $\varepsilon(K) = 1$,  we introduce a truncation by identifying an acyclic subcomplex, which we again denote by $Y$ by an abuse of notation. We set
\[
Y:=\bigoplus_{s\in \bZ_{\geq 0}} a_m|\big(\scE^\diamond_{s+\frac{3}{2},\frac{\ell-3}{2}} \oplus \scF^\diamond_{s+\frac{1}{2},\frac{\ell-3}{2}}\big) \oplus p|\big(\scJ^\diamond_{s+\frac{3}{2},\frac{\ell-3}{2}} \oplus \scM^\diamond_{s+\frac{1}{2},\frac{\ell-3}{2}}\big).
\]
By Equation~\eqref{eq: scE_scF_actions}, the length-$0$ differential vanishes on
$a_m|\big(\scE_{s+\frac{3}{2},\frac{\ell-3}{2}} \oplus \scF_{s+\frac{1}{2},\frac{\ell-3}{2}}\big)$
for any $s \in \bZ_{\geq 0}$ and therefore $Y$
is a subcomplex of $\bX^{\diamond}(P,K,\lambda)$. Since $\Phi^K$ restricts to the canonical isomorphism between  
$a_m|\big(\scE_{s+\frac{3}{2},\frac{\ell-3}{2}} \oplus \scF_{s+\frac{1}{2},\frac{\ell-3}{2}}\big)$ and $p|\big(\scJ^\diamond_{s+\frac{3}{2},\frac{\ell-3}{2}} \oplus \scM^\diamond_{s+\frac{1}{2},\frac{\ell-3}{2}}\big) $ when $s \in \bZ_{\geq 0}$, $Y$ is acyclic. Therefore   $\bX^{\diamond}(P,K,\lambda)/Y$ is chain homotopy equivalent to   $\bX^{\diamond}(P,K,\lambda)$. There is a decomposition
\[
Y= \Cone \big( Y^{\free} \xrightarrow{\partial} Y^{\Tor} \big). \]
In $\bX^{\diamond}(P,K,\lambda)/Y$, among the generators $p|w_{i+\frac{1}{2}}$ and $p|w'_{i-\frac{1}{2}}$, only those with $i<1$ are present.

In the following proof, we use $\partial$ for the differential on $\bX^{\diamond}(P,K,\lambda)$, $\partial^{\free}$ for the differential on $\bX^{\free}(P,K,\lambda)$  and $\bar\partial$ for the differential on $\bX^{\diamond}(P,K,\lambda)/Y$.

\noindent\textbf{Case \eqref{it: ending_general_epsilon<0_eta>0}: when $\varepsilon(K)=-1$ and $\eta>0$.} 
 There are two incoming arrows to $a_1,$ one weighted by $W^\eta$ from $a'_1$  and another weighted by $Z^n$ from $a_0$. 

Set 
\[\alpha_0 = a_1 | x'_{-\frac{1}{2}}.\]  
Since $\delta^1(a_1)=0$, we have $\partial \alpha_0=0.$
We then consider three cases according to the value of $\lambda - 2\tau(K)$ and define $\alpha_1$ in each case.
\begin{enumerate}[label=$($3.\arabic*$)$, ref=$($3.\arabic*$)$] \item 
 \label{it: epsilon<0_eta>0_lambda-2tau<0} 
\textbf{When $\lambda - 2\tau(K) < 0$.}
Set \[
\alpha_1 =   \left( a_1\Big|\sum^{n}_{i=1} x_{i-\frac{1}{2}}  Z^{\ell(n-i)} + a_0 | x'_{-\frac{1}{2}} Z^ \theta \right) Z^{\max\{\kappa-\theta,0\}} + p|w_{\lambda-2\tau(K)+\frac{1}{2}} Z^{\max\{\theta - \kappa,0\}}
\]
and we claim
  \begin{enumerate}[label=(\alph*), ref=\alph*]
    \item \label{it: Zn_claim_epsilon<0_eta>0_lambda<0}$\partial^{\free} \alpha_1 =  \alpha_0 Z^{\ell n + \max\{\kappa,\theta\}}$;
    \item \label{it: shortest_epsilon<0_eta>0_lambda<0} $\{\alpha_0,\alpha_1\}$ generates a direct summand of the chain complex $\bX^{\free}(P,K,\lambda)$. 
\end{enumerate} 
Claim \eqref{it: Zn_claim_epsilon<0_eta>0_lambda<0} is straightforward to verify. To prove Claim \eqref{it: shortest_epsilon<0_eta>0_lambda<0}, there are two cases depending on the sign of $\kappa-\theta.$ We demonstrate the case when $\kappa>\theta.$ The other case is similar and left to the reader. Perform the change of basis as follows:
\begin{align*}
 p|w_{\lambda-2\tau(K)+\frac{1}{2}} &\mapsto \alpha_1 \\
 p|w'_{\lambda-2\tau(K)-\frac{1}{2}}  &\mapsto  p|w'_{\lambda-2\tau(K)-\frac{1}{2}} + a_1 | x'_{n-\frac{1}{2}} \\
  a_1 | x'_{n+\frac{1}{2}}  &\mapsto  a_1 | x'_{n+\frac{1}{2}} + a_0 | x'_{-\frac{1}{2}}Z^{\ell + \theta} +   p|w_{\lambda-2\tau(K)+\frac{1}{2}}Z^{\omega} \\
   a_1 | x'_{i-\frac{1}{2}} &\mapsto a_1 |( x'_{i-\frac{1}{2}} +  x'_{i-\frac{3}{2}} Z^{\ell })     \qquad i= 1,\dots, n. 
\end{align*}
The remainder of the basis is unchanged. Since $\lambda - 2\tau(K) -\frac{1}{2}<-1$, $p|w'_{\lambda-2\tau(K)-\frac{1}{2}}$ is not in the image of $\Phi^K$. Therefore $\{a_0 | x'_{-\frac{1}{2}}, p|w'_{\lambda-2\tau(K)-\frac{1}{2}} + a_1 | x'_{n-\frac{1}{2}}\}$ generates a direct summand of $\bX^{\diamond}(P,K,\lambda)$. Similarly, $\{a_0 | x'_{\frac{1}{2}},a_1 | x'_{n+\frac{1}{2}} + a_0 | x'_{-\frac{1}{2}}Z^{\ell + \theta} +   p|w_{\lambda-2\tau(K)+\frac{1}{2}}Z^{\omega}\}$ generates a direct summand of $\bX^{\free}(P,K,\lambda)$. And as before, $\{a_1 | x_{i-\frac{1}{2}}, a_1 |( x'_{i-\frac{1}{2}} +  x'_{i-\frac{3}{2}} Z^{\ell })\}$ for $i=1,\dots, n$ generates a direct summand of $\bX^{\diamond}(P,K,\lambda)$. In the resulting basis, the only arrow pointing to $\alpha_0$ comes from $\alpha_1$. Claim \eqref{it: shortest_epsilon<0_eta>0_lambda<0} therefore follows. 

We conclude that $\Ord(P(K,\lambda))\geq \ell n + \max\{\kappa,\theta\}.$

\item  \label{it: epsilon<0_eta>0_lambda-2tau=0} \textbf{When $\lambda - 2\tau(K) = 0$.}
Set
\begin{align*}
\alpha_1 =   \left( a_1\Big|\sum^{n}_{i=1} x_{i-\frac{1}{2}}  Z^{\ell(n-i)} + a_0 | x'_{-\frac{1}{2}} Z^{\theta} \right) Z^{\max\{\kappa'-\theta,0\}} \\
+ a_m | y'_{-\frac{1}{2}} Z^{\max\{\kappa',\theta\}} + a_{m-1}|y_{-n+\frac{1}{2}} Z^{\max\{\theta - \kappa',0\}}
\end{align*}
where $\kappa'= R_{\frac{\ell}{2}-2} - g_3(P) + \frac{\ell}{2}.$  See Figure \ref{fig: prop_ending_epsilon<0}\subref{subfig: prop_ending_epsilon<0_eta<0} for a schematic illustration of $\alpha_1$ in a similar case. We claim
  \begin{enumerate}[label=(\alph*), ref=\alph*]
    \item \label{it: Zn_claim_epsilon<0_eta>0_lambda=0}$\partial^{\free} \alpha_1 =  \alpha_0 Z^{\ell n + \max\{\kappa',\theta\}}$;
    \item \label{it: shortest_epsilon<0_eta>0_lambda=0} $\{\alpha_0,\alpha_1\}$ generates a direct summand of the chain complex $\bX^{\free}(P,K,\lambda)$. 
\end{enumerate} 
We first prove Claim \eqref{it: Zn_claim_epsilon<0_eta>0_lambda=0}.
Since $\delta^1(a_{m-1})=0,$ the only outgoing arrows from $a_{m-1}|y_s$ are the arrows of $\Phi^{\mu}+\Phi^{-\mu}= \bI | (L_\sigma + L_\tau)$. By a similar argument as in the proof of Lemma \ref{lem: C_S_structure_maps}, we obtain that for $s<0,$ 
\[
L_{\tau} (y_s) = L_{\tau} (\xs^{\frac{\ell}{2}-2}_0) = y'_{s-1}\otimes Z^{\kappa'}
\]
and $L_{\sigma} (y_s)$ is torsion, possibly zero. Therefore
\[
\partial^{\free} (a_{m-1}|y_{-n+\frac{1}{2}}) = a_{m-1}|y'_{-n-\frac{1}{2}} Z^{\kappa'}.
\]
Since $\delta^1(y_m) = y_{m-1} \otimes W^n + p \otimes \sigma$ and $m_{1|1|0}(W^n,y'_s)=y'_{s-n}$ for $s<0$ by Equation~\eqref{eq: scE_scF_actions}, we have
\[
\partial (a_{m}|y'_{-\frac{1}{2}}) = a_{m-1}|y'_{-n-\frac{1}{2}} + p|w'_{-\frac{1}{2}}.
\]
We compute
\begin{align*}
    \partial^{\free} \alpha_1 &= \partial \left( a_1\Big|\sum^{n}_{i=1} x_{i-\frac{1}{2}}  Z^{\ell(n-i)} + a_0 | x'_{-\frac{1}{2}} Z^{\theta} \right) Z^{\max\{\kappa'-\theta,0\}} \\
&+ \big( a_{m-1}|y'_{-n-\frac{1}{2}} + p|w'_{-\frac{1}{2}}\big) Z^{\max\{\kappa',\theta\}} + a_{m-1}|y'_{-n-\frac{1}{2}} Z^{\kappa'} \cdot Z^{\max\{\theta - \kappa',0\}}\\
&=\big(\alpha_0 Z^{\ell n + \theta} +  p|w'_{-\frac{1}{2}} Z^{\theta} \big) Z^{\max\{\kappa'-\theta,0\}} + p|w'_{-\frac{1}{2}} Z^{\max\{\kappa',\theta\}}\\
&=\alpha_0 Z^{\ell n + \max\{\kappa',\theta\}},
\end{align*}
proving Claim \eqref{it: Zn_claim_epsilon<0_eta>0_lambda=0}. To prove of Claim \eqref{it: shortest_epsilon<0_eta>0_lambda=0}, 
there are two cases depending on the sign of $\kappa'-\theta$. We demonstrate the case when $\kappa' \geq \theta$ and leave the other case to the reader.
First, observe that the only incoming arrows to $a_{m-1}|y'_{-n-\frac{1}{2}}$ are the length-$0$ differential from $a_{m}|y'_{\frac{1}{2}}$ and the arrow of $\Phi^{-\mu}$ from $a_{m-1}|y_{-n+\frac{1}{2}}$. Indeed, the only other incoming arrows are the arrow of $\Phi^{\mu}$ and the length-$0$ differentials. 
Following from the above discussion,  $\Phi^{\mu} (a_{m-1}|y_{-n-\frac{1}{2}})$ is torsion;  the length-$0$ differential induced by the incoming arrow weighted by $Z^{\eta}$ is zero, since $m_{1|1|0}(Z^{\eta},y'_s)=0$ for $s<-1$ by Equation~\eqref{eq: scE_scF_actions}.
Perform the following change of basis:
\begin{align*}
 a_{m-1}|y_{-n+\frac{1}{2}} &\mapsto \alpha_1 \\
 a_{m-1}|y'_{-n-\frac{1}{2}} &\mapsto a_{m-1}|y'_{-n-\frac{1}{2}} + p|w'_{-\frac{1}{2}} \\
 p|w'_{-\frac{1}{2}}  &\mapsto  p|w'_{-\frac{1}{2}} + a_1 | x'_{n-\frac{1}{2}} \\
   a_1 | x'_{i-\frac{1}{2}} &\mapsto a_1 |( x'_{i-\frac{1}{2}} +  x'_{i-\frac{3}{2}} Z^{\ell })     \qquad i= 1,\dots, n. 
\end{align*}
The remainder of the basis is unchanged. Then $\{a_{m}|y'_{\frac{1}{2}}, a_{m-1}|y'_{-n-\frac{1}{2}} + p|w'_{-\frac{1}{2}}\}$ and $\{ a_0 | x'_{-\frac{1}{2}}, p|w'_{-\frac{1}{2}} + a_1 | x'_{n-\frac{1}{2}}\}$ each generates a direct summand of $\bX^{\diamond}(P,K,\lambda)$. In the resulting basis, the only arrow pointing to $\alpha_0$ comes from $\alpha_1$ and Claim  \eqref{it: shortest_epsilon<0_eta>0_lambda=0} follows.

We conclude that $\Ord(P(K,\lambda))\geq \ell n + \max\{\kappa',\theta\} \geq \ell n + \theta.$
\item \label{it: epsilon<0_eta>0_lambda-2tau>0} \textbf{When $\lambda - 2\tau(K) > 0$.} Working over $\bX^{\diamond}(P,K,\lambda)/Y$, 
  we set 
  \[\alpha_1 =        a_1\Big|\sum^{n}_{i=1} x_{i-\frac{1}{2}} Z^{\ell(n-i)} + a_0 | x'_{-\frac{1}{2}}  Z^{\theta}\]
 and claim
  \begin{enumerate}[label=(\alph*), ref=\alph*]
    \item $\bar\partial \alpha_1 =  \alpha_0 Z^{\ell n + \theta}$;
    \item  $\{\alpha_0,\alpha_1\}$ generates a direct summand of the chain complex $\bX^{\diamond}(P,K,\lambda)/Y$. 
\end{enumerate} 
Since $\Phi^{-K}(a_0 | x'_{-\frac{1}{2}}) = 0$ in $\bX^{\diamond}(P,K,\lambda)/Y$, the proof is essentially the same as that of Case \eqref{it: non_ending_general_case_1} in the proof of Proposition \ref{prop: non_ending_general_case} and therefore we omit it here.
  We conclude that
$\Ord(P(K,\lambda))\geq \ell n + \theta.$
\end{enumerate}

\medskip
\noindent\textbf{Case \eqref{it: ending_general_epsilon<0_eta<0}: when $\varepsilon(K)=-1$ and $\eta<0$.} 
 There is an arrow weighted by $Z^n$ from $a_0$ to $a_1$ followed by an arrow weighted by $W^{-\eta}$ from $a_1$ to $a'_1$. By Lemma \ref{lem: snake_structure} we have $n>1$.

The proof follows by combining  arguments from the previous case and Case \eqref{it: non_ending_general_case_2} in the proof of Proposition \ref{prop: non_ending_general_case}. We sketch out the mains steps, leaving the details for the reader to fill out.
\begin{enumerate}[label=$($4.\arabic*$)$, ref=$($4.\arabic*$)$] 
 \item \label{it: epsilon<0_eta<0_lambda-2tau<0} \textbf{When $\lambda - 2\tau(K) <0$.}
We consider two cases depending on the sign of $\omega - \theta =  \Xi +1$, as follows.
\begin{enumerate} [label=\roman*$)$]
    \item \label{it: epsilon<0_eta<0_lambda<0_case_1} \textbf{Suppose $\omega - \theta  \leq 0$.}  Set 
    \begin{align*}
        \alpha_0 &= a_1|x'_{\frac{1}{2}},\\
        \alpha_1 &=   \left(  a_1 \Big|\sum_{i=1}^{n-1} x_{i + \frac{1}{2}}Z^{\ell (n-1-i)} + a_0 | x'_{-\frac{1}{2}} Z^ \theta \right) Z^{\max\{\kappa-\theta,0\}} + p|w_{\lambda-2\tau(K)+\frac{1}{2}} Z^{\max\{\theta - \kappa,0\}}
    \end{align*}
    and we have 
     \begin{enumerate}[label=(\alph*), ref=\alph*]
      \item $\partial \alpha_0 =  0$;
    \item $\partial^{\free} \alpha_1 =  \alpha_0 Z^{\ell (n-1) + \max\{\kappa,\theta\}}$;
    \item  $\{\alpha_0,\alpha_1\}$ generates a direct summand of the chain complex $\bX^{\free}(P,K,\lambda)$. 
\end{enumerate} 
We conclude that $\Ord(P(K,\lambda))\geq \ell (n-1) + \max\{\kappa,\theta\}.$
 \item \label{it: epsilon<0_eta<0_lambda<0_case_2} \textbf{Suppose $\omega - \theta  > 0$.} Set
 \begin{align*}
        \alpha_0 &= a'_1|x_{\eta+\frac{1}{2}} +  a_1|x'_{-\frac{1}{2}}Z^\kappa,\\
        \alpha_1 &=   \left( a_1\Big|\sum^{n}_{i=1} x_{i-\frac{1}{2}}  Z^{\ell(n-i)} + a_0 | x'_{-\frac{1}{2}} Z^ \theta \right) Z^{\max\{\kappa-\theta,0\}} + p|w_{\lambda-2\tau(K)+\frac{1}{2}} Z^{\max\{\theta - \kappa,0\}}
    \end{align*}
 and we have 
     \begin{enumerate}[label=(\alph*), ref=\alph*]
      \item $\partial \alpha_0 Z^{\Xi}=  0$;
    \item $\partial^{\free} \alpha_1 =  \alpha_0 Z^{\ell (n-1) + \omega -\theta + \max\{\kappa,\theta\}}$;
    \item  $\{\alpha_0,\alpha_1\}$ generates a direct summand of the chain complex $\bX^{\free}(P,K,\lambda)$. 
\end{enumerate} 
By Lemma \ref{lem: torsion_order_criterion}, we conclude that $\Ord(P(K,\lambda))\geq \ell (n-1) + \omega -\theta + \max\{\kappa,\theta\} -\Xi = \ell (n-1) + \max\{\kappa,\theta\} + 1.$
\end{enumerate}

\item \label{it: epsilon<0_eta<0_lambda-2tau=0} \textbf{When $\lambda - 2\tau(K) =0$.}
\begin{enumerate} [label=\roman*$)$]
    \item \label{it: epsilon<0_eta<0_lambda=0_case_1} \textbf{Suppose $\omega -\theta \leq 0$.} 
  \begin{align*}
        \alpha_0 &= a_1|x'_{\frac{1}{2}},\\
        \alpha_1 &=   \left(  a_1 \Big|\sum_{i=1}^{n-1} x_{i + \frac{1}{2}}Z^{\ell (n-1-i)} + a_0 | x'_{-\frac{1}{2}} Z^{\theta} \right) Z^{\max\{\kappa'-\theta,0\}} \\
&+ a_m | y'_{-\frac{1}{2}} Z^{\max\{\kappa',\theta\}} + a_{m-1}|y_{-n+\frac{1}{2}} Z^{\max\{\theta - \kappa',0\}}
    \end{align*}
    where $\kappa'= R_{\frac{\ell}{2}-2} - g_3(P) + \frac{\ell}{2}.$ We claim
  \begin{enumerate}[label=(\alph*), ref=\alph*]
  \item $\partial \alpha_0 = 0$
    \item $\partial^{\free} \alpha_1 =  \alpha_0 Z^{\ell (n-1) + \max\{\kappa',\theta\}}$;
    \item  $\{\alpha_0,\alpha_1\}$ generates a direct summand of the chain complex $\bX^{\free}(P,K,\lambda)$. 
\end{enumerate} 
We conclude that $\Ord(P(K,\lambda))\geq \ell (n-1)  + \max\{\kappa',\theta\} \geq \ell (n-1) + \theta.$
    \item \label{it: epsilon<0_eta<0_lambda=0_case_2} \textbf{Suppose $\omega -\theta > 0$. }
  \begin{align*}
        \alpha_0 &= a'_1|x_{\eta+\frac{1}{2}} +  a_1|x'_{-\frac{1}{2}}Z^\kappa,\\
        \alpha_1 &=   \left( a_1\Big|\sum^{n}_{i=1} x_{i-\frac{1}{2}}  Z^{\ell(n-i)} + a_0 | x'_{-\frac{1}{2}} Z^{\theta} \right) Z^{\max\{\kappa'-\theta,0\}} \\
&+ a_m | y'_{-\frac{1}{2}} Z^{\max\{\kappa',\theta\}} + a_{m-1}|y_{-n+\frac{1}{2}} Z^{\max\{\theta - \kappa',0\}}
    \end{align*}
    and we have
  \begin{enumerate}[label=(\alph*), ref=\alph*]
    \item $\partial \alpha_0 Z^{\Xi}=  0$;
    \item $\partial^{\free} \alpha_1 =  \alpha_0 Z^{\ell (n-1) + \omega -\theta + \max\{\kappa',\theta\}}$;
    \item  $\{\alpha_0,\alpha_1\}$ generates a direct summand of the chain complex $\bX^{\free}(P,K,\lambda)$. 
\end{enumerate} 
By Lemma \ref{lem: torsion_order_criterion}, we conclude that $\Ord(P(K,\lambda))\geq \ell (n-1)  + \omega -\theta + \max\{\kappa',\theta\} -\Xi = \ell (n-1) + \max\{\kappa',\theta\} + 1 \geq \ell (n-1) + \theta + 1.$
\end{enumerate}

\item \label{it: epsilon<0_eta<0_lambda-2tau>0} \textbf{When $\lambda - 2\tau(K) >0$.}
   \begin{enumerate} [label=\roman*$)$]
    \item \label{it: epsilon<0_eta<0_lambda>0_case_1} \textbf{Suppose $\omega -\theta \leq 0$.} Working over $\bX^{\diamond}(P,K,\lambda)/Y$, we set
 \begin{align*}
        \alpha_0 = a_1|x'_{\frac{1}{2}},\qquad
        \alpha_1 =      a_1 \Big| \sum_{i=1}^{n-1}  x_{i + \frac{1}{2}}Z^{\ell (n-1-i)} + a_0 | x'_{-\frac{1}{2}} Z^{\theta} 
    \end{align*}
and claim 
 \begin{enumerate}[label=(\alph*), ref=\alph*]
    \item $\partial \alpha_0 =  0$;
    \item $\bar \partial \alpha_1 =  \alpha_0 Z^{\ell (n-1)   + \theta}$;
    \item  $\{\alpha_0,\alpha_1\}$ generates a direct summand of the chain complex $\bX^{\diamond}(P,K,\lambda)/Y$. 
\end{enumerate} 
We conclude that $\Ord(P(K,\lambda))\geq \ell (n-1)  + \theta.$
 \item \label{it: epsilon<0_eta<0_lambda>0_case_2} \textbf{Suppose $\omega -\theta > 0$.}
 Set
 \begin{align*}
        \alpha_0 = a'_1|x_{\eta+\frac{1}{2}} +  a_1|x'_{-\frac{1}{2}}Z^\kappa,\qquad
        \alpha_1 =    a_1\Big|\sum^{n}_{i=1} x_{i-\frac{1}{2}}  Z^{\ell(n-i)} + a_0 | x'_{-\frac{1}{2}} Z^{\theta}
    \end{align*}
and we have
 \begin{enumerate}[label=(\alph*), ref=\alph*]
    \item $\partial \alpha_0 Z^{\Xi}=  0$;
    \item $\bar \partial \alpha_1 =  \alpha_0 Z^{\ell (n-1) + \omega}$;
    \item  $\{\alpha_0,\alpha_1\}$ generates a direct summand of the chain complex $\bX^{\free}(P,K,\lambda)/Y^{\free}$. 
\end{enumerate} 
We conclude that $\Ord(P(K,\lambda))\geq \ell (n-1)  + \omega - \Xi =\ell (n-1) + \theta + 1.$
    \end{enumerate}
    \end{enumerate}
\end{proof}

   In previous results, such as Proposition~\ref{prop: ending_arrow_general_case} we made some assumptions which excluded the case that $K=T_{2,3}$, so we handle this now. (Compare Remark~\ref{rem:exclude-T23}). See Definition \ref{def: shorthand} for the definitions of $\omega$ and $\kappa$.
\begin{thm}\label{thm: tor_ord_T_23}
Suppose $P$ is  L-space satellite pattern. For any $\lambda \in \bZ,$ we have
    \[\Ord(P(T_{2,3},\lambda)) \geq \begin{cases}
        \min\{ \max\{\ell,\omega\}, \max\{\theta,\kappa\} + 1\} &\ell >0 \\
        \omega &\ell = 0.
    \end{cases} \] 
\end{thm}
\begin{proof}
The proof of Case \ref{it: epsilon>0_eta<0_lambda-2tau<0} in Proposition \ref{prop: ending_arrow_general_case} applies verbatim except when $\lambda <2 \tau(T_{2,3}) = 2$ and $\omega - \theta <0$. In this situation, the construction in the subcase \ref{it: epsilon>0_eta<0_lambda<0_case_1} is no longer available, so we instead adopt the construction from the subcase \ref{it: epsilon>0_eta<0_lambda<0_case_2}. Explicitly, we set
       \begin{align*}
    \alpha_0 &= p|w_{\frac{1}{2}+\lambda-2\tau(K)} +  a_0|x'_{-\frac{1}{2}}Z^\kappa, \\
    \alpha_1 &=   \left(  a_0 | x_{ \frac{1}{2}}  + a_1|x'_{-\frac{1}{2}} Z^ \theta \right) Z^{\max \{\kappa-\theta, 0 \}} + a'_1|x_{-\frac{1}{2}} Z^{\max \{\theta - \kappa, 0 \}}.
\end{align*}
Then we have
\begin{enumerate}[label=(\alph*), ref=\alph*]
\item \label{it: T_23_alpha0_cycle} $\partial \alpha_0  = 0$;
    \item  $\partial^{\free} \alpha_1 =  \alpha_0 Z^{ \omega - \theta + \max\{\kappa, \theta \}}$;
    \item  $\{ \alpha_0, \alpha_1 \}$ generates a summand in $\bX^{\free}(P,K,\lambda)$. 
\end{enumerate}  
Compared to the claims in the subcase \ref{it: epsilon>0_eta<0_lambda<0_case_2} in Case \ref{it: epsilon>0_eta<0_lambda-2tau<0}, since $\Xi<0,$ $\alpha_0$ is now a cycle.  
We obtain
        $\Ord(P(K,\lambda))  \geq  \omega - \theta + \max\{\kappa, \theta \}  =  \max\{\kappa+ \omega - \theta, \omega  \} =\max\{\ell, \omega  \}.$ 
        
Compare the above bound with the ones obtained in the remaining subcases: in subcase \ref{it: epsilon>0_eta<0_lambda<0_case_2} in Case \ref{it: epsilon>0_eta<0_lambda-2tau<0}, the bound is $\max\{\theta,\kappa\} + 1$ and in Case \ref{it: epsilon>0_eta<0_lambda-2tau>-1} the bound is $\max\{\theta,\kappa\} + \ell$. When $\ell >0,$ we obtain the desired bound; when $\ell =0,$ since by Definition \ref{def: shorthand}, $\theta = \omega + \kappa - \ell = \omega + \kappa,$ we also obtain the desired bound. 
\end{proof}
For braided L-space satellite patterns, the bounds in Proposition \ref{prop: ending_arrow_general_case} can be improved, as follows. Again the following proposition does not fully recover the bounds in the second half of \cite[Proposition 3.1]{HLPUnknotting}. However, it suffices for the proof of the unknotting number bound on the iterated braided L-space satellite patterns stated in Theorem \ref{thm: iterate_braided}.
\begin{prop}\label{prop: ending_arrow_braided_case}
     Suppose  the ending arrow of the standard complex in $\cCFK(K)$ is weighted by $Z^n$ for some  $n>0$. Let $P$ be a braided L-space pattern with $\omega = R_{\frac{\ell}{2}}-R_{\frac{\ell}{2}-1}$. Given $\lambda \in \Z$, we have the following  
     \begin{itemize}
    \item  Suppose $\varepsilon(K)=1$. 
    \medskip
     \begin{enumerate}[label=(\arabic*)]
    \item \label{it: ending_braided_case_1} When $\eta>0$,
        \begin{align*}
            \Ord(P(K, \lambda)) \geq \begin{cases}
            \ell (n -1)  + 1   \quad &\text{if} \quad  \lambda - 2\tau(K) \leq -2 \\
            \ell (n -1) + \omega   \quad &\text{if} \quad  \lambda - 2\tau(K) = -1 \\
             \ell n      \quad &\text{if} \quad  \lambda - 2\tau(K) \geq 0;
            \end{cases}
        \end{align*}
     \item \label{it: ending_braided_case_2} When $\eta<0$,
        \begin{align*}
            \Ord(P(K, \lambda)) \geq \begin{cases}
            \ell n -\omega  + 1   \quad &\text{if} \quad  \lambda - 2\tau(K) \leq -2 \\
            \ell n   \quad &\text{if} \quad  \lambda - 2\tau(K) = -1 \\
             \ell (n+1) - \omega      \quad &\text{if} \quad  \lambda - 2\tau(K) \geq 0.
            \end{cases}
            \end{align*}
    \end{enumerate}
\item
    Suppose $\varepsilon(K)=-1$. 
    \medskip
    \begin{enumerate}[start=3,label=(\arabic*)]
     \item \label{it: ending_braided_case_3} When  $\eta>0$,
        \begin{align*}
            \Ord(P(K, \lambda)) \geq          
             \begin{cases}
            \ell (n+1) -\omega   \quad &\text{if} \quad  \lambda - 2\tau(K) < 0 \\
             \ell n     \quad &\text{if} \quad  \lambda - 2\tau(K) \geq 0;
            \end{cases}.
        \end{align*}
      \item \label{it: ending_braided_case_4} When  $\eta<0$,
        \begin{align*}
            \Ord(P(K, \lambda)) \geq      
             \begin{cases}
            \ell n  - \omega + 1   \quad &\text{if} \quad  \lambda - 2\tau(K) < 0 \\
             \ell (n-1)  + 1   \quad &\text{if} \quad  \lambda - 2\tau(K) \geq 0.
            \end{cases}
            \end{align*}
    \end{enumerate}
    \end{itemize}
\end{prop}
\begin{proof}
Since $P$ is a braided L-space satellite operator,
by Definition \ref{def: shorthand} and Lemma \ref{lem: braided_property}, $\Xi \geq 0, \theta=0$ and $\kappa = \ell - \omega + \theta = \ell - \omega.$
    Compared to the proof of Proposition \ref{prop: ending_arrow_general_case}:
    \begin{enumerate}[label=$($\arabic*$)$]
        \item \textbf{Suppose $\varepsilon(K)=1$ and $\eta>0$:}
         When $\lambda - 2 \tau(K)< 0$, since $\Xi \geq 0,$
         only the subcase \ref{it: epsilon>0_eta>0_lambda<0_case_2} in Case \ref{it: epsilon>0_eta>0_lambda-2tau<0} is possible. Since $\theta=0,$ we obtain $\Ord(P(K, \lambda)) \geq \ell(n-1)  + 1$.

       Moreover,
          when $\lambda - 2 \tau(K)= -1$, we are in the special case \ref{it: epsilon>0_eta>0_lambda<0_specialcase} in Case \ref{it: epsilon>0_eta>0_lambda-2tau<0}, therefore $\Ord(P(K, \lambda)) \geq \ell(n-1)  + \omega$.
    
     When $\lambda - 2 \tau(K)\geq 0$, the original bound from Case \ref{it: epsilon>0_eta>0_lambda-2tau>-1} is unchanged.
         \item \textbf{Suppose $\varepsilon(K)=1$ and $\eta<0$:}
          When $\lambda - 2 \tau(K) <0$, only the subcase \ref{it: epsilon>0_eta<0_lambda<0_case_2} in Case \ref{it: epsilon>0_eta<0_lambda-2tau<0} is possible. Therefore $\Ord(P(K, \lambda)) \geq \ell(n-1)  + \kappa + 1 =\ell n -\omega +1$ and condition $n>1$ is no longer necessary. 
         
        Moreover, when $\lambda - 2 \tau(K) = -1$,  we are in the special case \ref{it: epsilon>0_eta<0_lambda<0_specialcase} in Case \ref{it: epsilon>0_eta<0_lambda-2tau<0}. Therefore $\Ord(P(K, \lambda)) \geq \ell(n-1)  + \omega + \kappa =\ell n$.
  
        When $\lambda - 2 \tau(K)\geq 0$, the original bound from Case \ref{it: epsilon>0_eta<0_lambda-2tau>-1} is unchanged. We have  $\Ord(P(K, \lambda)) \geq \ell n + \kappa=\ell (n+1) -\omega.$
  \item \textbf{Suppose $\varepsilon(K)=-1$ and $\eta>0$:}
         These are the same bounds as those obtained in Case \eqref{it: ending_general_epsilon<0_eta>0} in Proposition \ref{prop: ending_arrow_general_case}, after rewriting them slightly using $\theta=0$ and $\kappa = \ell - \omega.$
 \item \textbf{Suppose $\varepsilon(K)=-1$ and $\eta<0$:}
         Since $\Xi \geq 0,$
         in all three cases \ref{it: epsilon<0_eta<0_lambda-2tau<0}, \ref{it: epsilon<0_eta<0_lambda-2tau=0} and  \ref{it: epsilon<0_eta<0_lambda-2tau>0}, only the subcases numbered by \ref{it: epsilon<0_eta<0_lambda<0_case_2} are possible. 
   Thus we obtain the bounds after rewriting them slightly.
   \end{enumerate}
\end{proof}
We require one more lemma before proving Theorem \ref{thm: tor_ord}.
\begin{lem}\label{lem: Ord=1_cases}
    For a knot $K \subset S^3$, suppose $\Ord(K)=1$. Then either
    \begin{itemize}
        \item $\cCFK(K)^{\hat{\cR}}$ contains a non-ending arrow  of type $(+,-)$ in the sense of Definition~\ref{def:non-ending-types}; or
        \item $K=\pm T_{2,3}.$
    \end{itemize}
\end{lem}
\begin{proof}
    When $\Ord(K)=1$,
    by Lemma \ref{lem: cfk_ord=1},  $\cCFK(K)^{\hat{\cR}}\simeq \cC \oplus \cB_1 \oplus \cdots  \oplus \cB_k$ for some $k\geq 0$, where $\cC$ is a positive or negative staircase   and  each $\cB_i$ is a length-one box.  If $k>0,$ since 
    each $\cB_i$ contains an arrow   of type $(+,-)$,  the claim holds. Therefore we may assume $k=0$. 
    If  $\cC$ has more than $3$ generators, then it necessarily contains a non-ending arrow  of type $(+,-)$, and the claim holds. Otherwise,  we must have $\cCFK(K)^{\hat{\cR}} \simeq \cCFK(\pm T_{2,3})^{\hat{\cR}}$.
      By \cite{Ghigginidetect}, the knot Floer complex detects $\pm T_{2,3}.$  Therefore $K=\pm T_{2,3}$  in this case, proving the claim. 
\end{proof}
We are now ready to prove Theorem \ref{thm: tor_ord}, which we recall below.
\TorOrd*
\begin{proof} Since knot Floer homology is insensitive to reversing the string orientation, we may assume (as we have done throughout the paper) that $\ell\ge 0$.

If $\Ord(K)> 1,$
the bound follows from Proposition \ref{prop: non_ending_general_case} and \ref{prop: ending_arrow_general_case}.

   We consider the case when $\Ord(K)=1$. Since $K\neq T_{2,3},$ by Lemma \ref{lem: Ord=1_cases},  either $\cCFK(K)^{\hat{\cR}}$ contains a non-ending arrow  of type $(+,-)$ or $K=-T_{2,3}.$ In the first case, 
     by Case \eqref{it: non_ending_general_case_3} in Proposition \ref{prop: non_ending_general_case} we have $\Ord(P(K,\lambda)) \geq \ell + \theta$. 
    In the second case,  $\varepsilon(K)=-1$ and $\cCFK(K)^{\hat{\cR}}$ contains an ending arrow with $\eta > 0.$ By Case \eqref{it: ending_general_epsilon<0_eta>0} in Proposition \ref{prop: ending_arrow_general_case}, $\Ord(P(K,\lambda)) \geq \ell + \theta$. This completes the proof.
\end{proof}

By a similar argument, we have the following.
\TorOrdBraid*
\begin{proof}
   Follows by a similar argument as that in the proof of Theorem \ref{thm: tor_ord}, using Proposition \ref{prop: non_ending_braided_case} and \ref{prop: ending_arrow_braided_case}.
\end{proof}

\section{Proof of the bounds on iterated satellites}\label{sec: iterate}
In this section we will obtain a bound for iterated satellites $P_k \circ \cdots \circ P_1(K)$, focusing on the case when each $P_i$ is a braided L-space pattern.

For iterated L-space satellite operators in general, an unknotting number bound follows immediately by repeatedly applying Theorem \ref{thm: tor_ord}, as discussed in the introduction; see Theorem \ref{thm: iterate_general}. However, this bound is clearly not optimal. For example, in the case of cabling, where we have $\theta_i=0$,  the bound in Theorem \ref{thm: iterate_general} is strictly weaker than the bound given in \cite[Theorem 1.1]{HLPUnknotting}.

To improve the initial bound, we use an argument similar to that in \cite{HLPUnknotting}, showing that the type of the input arrow determines the type of the output arrow. This is achieved by keeping track of arrows weighted by positive powers of $W$ in the proofs of Propositions \ref{prop: non_ending_general_case} and \ref{prop: ending_arrow_general_case}. We therefore require a version of the construction in Section \ref{subsec: complex_Pk} with coefficients in $\hat\cR=\bF[W,Z]/U$ rather than $\bF[Z]$.

 We use the formula
 \begin{equation}\label{eq: CFK_PK_hatR}
     \cCFK(P(K,\lambda))^{\hat\cR} \simeq \cX_{\lambda}(K)^{\hat\cK} \boxtimes {}_{\hat\cK}\cH_{-}^{\hat\cK}\boxtimes {}_{\hat\cK} \cX^{\diamond}(L_P)^{\hat\cR}.
 \end{equation}
Denote the complex on the right hand side $\hat\bX(P,K,\lambda)^{\hat{\cR}}$. Compared to $\bX^{\diamond}(P,K,\lambda)$, we replace each $\cC^{\diamond}_t$ by $\hat\cC_t:=\cC_t/(U)$ and each $\cS^{\diamond}$ by $\hat\cS:=\cS/(U)$. The standard basis of $\hat\bX(P,K,\lambda)$
consists of those of  $\bX^{\diamond}(P,K,\lambda)$ with the addition of elements $\ys^t_i \in \hat\cC_t$ for $i=1,\dots,k_t$  and $\ys'_i \in \hat\cS$ for $i=1,\dots,k$, such that
\begin{align*}
\delta^1(\ys^t_i) &= \xs^t_{i-1} \otimes W^{\zeta^t_{i}} + \xs^t_i \otimes Z^{\xi^t_i} \quad i=1,\dots,k_t \quad \text{and} \\
\delta^1(\ys'_i) &= \xs'_{i-1} \otimes W^{\zeta'_{i}} + \xs'_i \otimes Z^{\xi'_i} \quad i=1,\dots,k  
\end{align*}
for some positive integer $\zeta^t_i$ and $\zeta'_i$. 

Each $\hat\cC_t$ is the positive staircase
$\cC(\zeta^t_1,-\xi^t_1,\dots,\zeta^t_{k_t},-\xi^t_{k_t})$
with generators
$\xs^t_0,\ys^t_1,\xs^t_1,\dots,\ys^t_{k_t},\xs^t_{k_t}$, order according to the definition of the standard complex.
Similarly, $\hat\cS$ is the positive staircase
$\cC(\zeta'_1,-\xi'_1,\dots,\zeta'_{k},-\xi'_{k})$
with generators
$\xs'_0,\ys'_1,\xs'_1,\dots,\ys'_{k},\xs'_{k}$, order according to the definition of the standard complex.
Since $\hat\cS \simeq \cCFK(P)$, the symmetry of the knot Floer complex implies that
$\zeta'_i = \xi'_{k+1-i}$.
Moreover, since $P \subset S^3$ is an L-space knot,
\cite[Theorem 7]{HeddenWatson} implies that $\zeta'_1=1$.

We give the generators $\xs^t_i$ and $\xs'_i$ \emph{algebraic grading} $0$ and  $\ys^t_i$ and $\ys'_i$ algebraic grading $1$.

The complex $\hat\bX(P,K,n)^{\hat{\cR}}$ can similarly be described as a 2-dimensional hypercube of the form 
\begin{equation*}
\hat\bX(P,K,n)^{\hat{\cR}}=\begin{tikzcd}[column sep=1.1cm, row sep=1.1cm]
 \hat\bE \ar[r, "\Phi^{\mu}+\Phi^{-\mu}"] \ar[d, "\Phi^{K}+\Phi^{-K}",swap] \ar[dr,dashed,"*"] & \hat\bF\ar[d, "\Phi^{K}+\Phi^{-K}"]\\
\hat\bJ \ar[r, swap,"\Phi^{\mu}+\Phi^{-\mu}"]& \hat\bM
\end{tikzcd}
\end{equation*}
where $*$ denotes a sum of four maps $\Phi^{K,\mu},\Phi^{K,-\mu},\Phi^{-K,\mu}$ and $\Phi^{-K,-\mu}$. Compared to $\bX(P,K,n)_{\bF[Z]}$, aside from the addition of the diagonal arrow, each $\Phi^{\pm\mu}$ and $\Phi^{\pm K}$ also contains extra terms. 
As before, $\Phi^{\pm\mu} = \bI_{\cX_{\lambda}(K)} \boxtimes f^{\pm \mu}$ and $\Phi^{\pm K} = \bI_{\cX_{\lambda}(K)} \boxtimes f^{\pm K}.$
Since our argument only involves $\hat\bE \oplus \hat\bF$, in the following we give the definitions of the maps \[f^{\pm \mu}: \bigoplus_s \hat\scE_{s,t} \to \bigoplus_s \hat\scF_{s,t}\] 
where
\begin{equation}\label{eq: hatscE-s-def}
\hat\scE_{s,t}=\begin{cases}
    \hat\cC_{t+\frac{1}{2}} &\text{ if }s>0\\
    \hat\cC_{t-\frac{1}{2}} &\text{ if } s<0\\
\end{cases}
\end{equation}
and $\hat\scF_{s,t} = \hat\cS$. See Figure \ref{fig:the-maps-f-pm-mu} for a schematic drawing of $f^{\pm \mu}$, and we refer the reader to \cite[Section 9.2]{CZZ} for the definitions of the other maps. 
\begin{figure}[h]
\[
\begin{tikzcd}[labels=description, row sep=2cm] \hat\scE_{*,t} 
\ar[d, "f^{\mu}"]
\\ 
\hat\scF_{*,t}
\end{tikzcd}
\hspace{-.3cm}
=
\hspace{-.3cm}\begin{tikzcd}[column sep=.8cm, row sep=0cm]
\cdots
&[-1cm]\hat\scE_{-\frac{5}{2},t}
	\ar[d,"L_\sigma",  pos=.4]
& \hat\scE_{-\frac{3}{2},t}
	 \ar[l, "W|1"']
	\ar[d,"L_\sigma",  pos=.4]
&\hat\scE_{-\frac{1}{2},t}
	\ar[r, bend left,out=30,in=150, "Z|L_Z"]
	 \ar[l, "W|1"']
	\ar[d,"L_\sigma",  pos=.4]
	\ar[loop above,looseness=13, "{(W,Z)|h_{W,Z}}"{description}]
	\ar[dr,pos=.68, "Z|h_{\sigma,Z}"{description}, end anchor={[xshift=1ex]}]
&[1.2cm] \hat\scE_{\frac{1}{2},t}
	\ar[r, "Z|1"]
	\ar[l, out=-160,in=-20,pos=.51,  "W|L_W"'{yshift=1pt}]
	\ar[d,"L_\sigma", pos=.4]
	\ar[loop above, looseness=13, "{(Z,W)|h_{Z,W}}"{description}]
	\ar[dl,  pos=.27,  "W|h_{\sigma,W}"{description},crossing over, end anchor={[xshift=-1ex]}]
& \hat\scE_{\frac{3}{2},t}
	\ar[r, "Z|1"]	 
	\ar[d,"L_\sigma",  pos=.4]
& \hat\scE_{\frac{5}{2},t}
	\ar[d,"L_\sigma", pos=.4]
&[-1cm] \cdots
\\[1.8 cm]
\cdots
&[.4cm]
 \hat\scF_{-\frac{5}{2},t}
& \hat\scF_{-\frac{3}{2},t} 
	 \ar[l, "W|1"']
&\hat\scF_{-\frac{1}{2},t}
	\ar[r, "Z|1"]
	 \ar[l, "W|1"']
&\hat\scF_{\frac{1}{2},t}
	\ar[r, "Z|1"] 
&\hat\scF_{\frac{3}{2},t}
	\ar[r, "Z|1"]
&\hat\scF_{\frac{5}{2},t}
&\cdots 
 \end{tikzcd}
 \]
 \[\begin{tikzcd}[row sep=2cm] \hat\scE_{*,t} 
\ar[d, "f^{-\mu}"{description}]
\\ 
\hat\scF_{*,t}
\end{tikzcd}
\hspace{-.25cm}=
\hspace{-.4cm}\begin{tikzcd}[column sep=1.3 cm, row sep=0cm]
\cdots
&[-1.5 cm]\hat\scE_{-\frac{5}{2},t}
&[-.3 cm] \hat\scE_{-\frac{3}{2},t}
	 \ar[l, "W|1"']
	\ar[dl, "L_\tau"',  pos=.6]
 & \hat\scE_{-\frac{1}{2},t}
 	\ar[r, bend left,out=30,in=150, "Z|L_Z"]
 	 \ar[l, "W|1"']
 	\ar[dl, "L_\tau"',  pos=.6]
 	\ar[loop above,looseness=13, "{(W,Z)|h_{W,Z}}"{description}]
 &[1 cm] \hat\scE_{\frac{1}{2},t}
 	\ar[r, "Z|1"]
 	\ar[l, out=-160,in=-20,pos=.51,  "W|L_W"'{yshift=1pt}]
 	\ar[dl, "L_\tau", pos=.6]
 	\ar[loop above, looseness=13, "{(Z,W)|h_{Z,W}}"{description}]
 	\ar[dll, pos=.37, "\substack{ W|h_{\tau,W}}"{description}, shorten =2mm,end anchor={[xshift=-2ex]}, start anchor={[xshift=2ex]} ]
 & \hat\scE_{\frac{3}{2},t}
 	\ar[dl, "L_\tau"',  pos=.6]
 &[-1.6cm] \cdots
 \\[1.9 cm]
 \cdots
 &[.4cm] \hat\scF_{ -\frac{5}{2},t} 	 
 & \hat\scF_{ -\frac{3}{2},t} 	 
 	 \ar[l, "W|1"']
 &\hat\scF_{-\frac{1}{2},t}
 	\ar[r, "Z|1"]
 	 \ar[l, "W|1"']
  	\ar[from=u,pos=.21, "\substack{Z|h_{\tau,Z}}"{description},crossing over]
 &\hat\scF_{\frac{1}{2},t}
 	\ar[r, "Z|1"]	 
 &\hat\scF_{\frac{3}{2},t}	 
 &\cdots 
  \end{tikzcd}
\]
\caption{Schematics of the maps $f^{\pm \mu}$}
\label{fig:the-maps-f-pm-mu}
\end{figure}
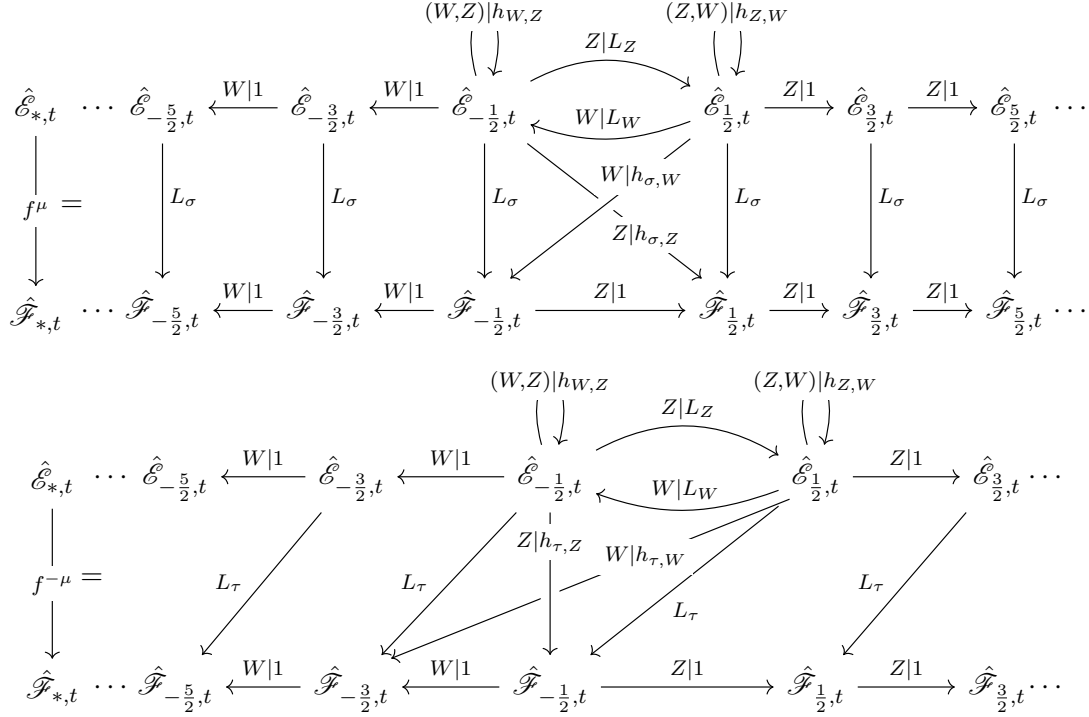
Here
$h_{W,Z},h_{Z,W}\colon \hat\cC_{t} \to \hat\cC_{t}$ are maps increasing the algebraic grading by $1$ and
satisfying
\[
\d(h_{W,Z})=L_W\circ L_Z
\qquad \text{and} \qquad
\d(h_{Z,W})=L_Z\circ L_W.
\]
The maps
$h_{\sigma,W}, h_{\sigma,Z}, h_{\tau,W}, h_{\tau,Z}
\colon \hat\cC_{t} \to \hat\cS$ increasing the algebraic grading by $1$
satisfy analogous relations. For example,
\[
\d(h_{\sigma,W})=L_{\sigma}\circ L_W,
\]
and similarly for the other three maps. 
\begin{define}
 Denote the generator 
 $\ys^{t}_1\in \hat\scE_{s,\frac{\ell-1}{2}} $ by $ \hat y_{s}$, where $t = \frac{\ell}{2}-1$ or $\frac{\ell}{2}$ depending on $s$ as in Equation~\eqref{eq: hatscE-s-def}; denote the generator
 $\ys'_1\in \hat\scF_{s,\frac{\ell-1}{2}}=\hat \cS $ by $ \hat y'_{s}$.
  By an abuse of notation, we continue to use the  previously defined shorthand notations $x_s$ and $x'_s$ in Definition \ref{def: xs-shorthand}, viewing them now as generators in $\hat\scE_{s,\frac{\ell-1}{2}}$ and $\hat\scF_{s,\frac{\ell-1}{2}}$, respectively.
\end{define}

Unlike in \cite{HLPUnknotting}, we are generally unable to determine the type of the output arrow explicitly. Instead, we exploit the specific features of our construction to gain partial control over the resulting arrow, narrowing the possibilities to a few cases. This yields the following lemma.
\begin{prop} \label{prop: longest_arrow}
Let $P$ be a braided L-space satellite pattern with $\ell>1$. 
    Suppose that $\Ord(K)=n$   
     and that $\cCFK(K)$ contains an arrow weighted by $Z^n$ in its standard complex or in one of the local system complexes, of one of the following types: 
    \begin{itemize}
        \item  a non-ending arrow of type $(+,+)$; or
        \item  a non-ending arrow of type $(+,-)$; or
        \item    an ending arrow with $\eta>0$ and $\varepsilon(K)=-1$.
    \end{itemize} 
    Then one of the following holds:

\begin{enumerate}[label=(\roman*)]
\item  $\Ord(P(K,\lambda))>\ell n$.
\item $\cCFK(P(K,\lambda))$ contains an arrow weighted by $Z^{\ell n}$ in its standard complex or one of its local system complexes, which is again of one of the three types listed above.
\end{enumerate}
\end{prop}
The proof of Proposition~\ref{prop: longest_arrow} relies on Lemma \ref{lem: alpha0_hori} and Lemma \ref{lem: refined_tor}. We start with Lemma \ref{lem: alpha0_hori}, which provides partial information about the structure of $\hat\bX(P,K,\lambda)$ through explicit computations.

\begin{lem}\label{lem: alpha0_hori}
Let $a$ be a generator in a reduced basis of  $\cCFK(K)^{\hat{\cR}}$ such that $\delta^1(a)=0$ in $\cX_{\lambda}(K)^{\hat{\cK}}$.
Given a braided L-space satellite pattern $P$ with $\ell>1$, then
\begin{itemize}
    \item $ a| x'_{-\frac{1}{2}}$ is a cycle in $\cCFK(P(K))$ (recall that this lies in $\hat{\bF}$, the product of the $F_{s,t}$ complexes);
    \item  $[a| x'_{-\frac{1}{2}}] \neq 0$ and  $[a| x'_{-\frac{1}{2}}]W =0$ in $H_*(\cCFK(P(K))/(Z))$.
\end{itemize}
\end{lem}
\begin{proof}
The first bullet point 
 follows from $\delta^1(a)=0$ and $\delta^1_1(x'_{-\frac{1}{2}})=0$. 
 
To prove the second  bullet point, we analyze all possible incoming arrows to
$a|x'_{-\frac{1}{2}}$
in $\hat\bX(P,K,\lambda)$.
The images of $\Phi^{\pm K}$ and $\Phi^{\pm K,\pm\mu}$ are not contained in $\hat{\bF}$ and hence the only possible differentials pointing to $a|x_{-\frac{1}{2}}'$ are from $\Phi^{\pm \mu}$ or from the internal differential of $\hat{\bF}$. 

We first consider the internal arrows in $\hat{\bF}$. Since the basis of $\cCFK(K)^{\hat{R}}$ is reduced, every component of the differential of $\cCFK(K)^{\hat{\cR}}$ is weighted by a power of $W$ or $Z$. We observe firstly that there are no differentials in the box tensor product pointing to $a|x_{-\frac{1}{2}}$ contributed by a single application of the differential of $\cCFK(K)^{\hat{\cR}}$ because there are no components of $\delta_2^1(W,-)$ or $\delta_2^1(Z,-)$ which point to $x'_{-\frac{1}{2}}\in \hat{\scF}_{-\frac{1}{2},\frac{\ell-1}{2}}$ (see Figure~\ref{fig:the-maps-f-pm-mu}). 

Next, we observe that there are no components of the differential of the box tensor product pointing to $a|x_{-\frac{1}{2}}$ which are contributed by two applications of the differential of $\cCFK(K)^{\hat{\cR}}$ because these maps all involve a factor of one of
\[
h_{Z,W},\ h_{W,Z},\ h_{\sigma,W},\ h_{\sigma,Z},\ h_{\tau,W},
\text{ and } h_{\tau,Z},
\]
and these maps have image in algebraic grading $1$, while
$x'_{-\frac{1}{2}}$ lies in algebraic grading 0.

We now consider the map $\Phi^{-\mu}$. This map is shown schematically in Figure~\ref{fig:the-maps-f-pm-mu}. If such an arrow points to $a|x'_{-\frac{1}{2}}$, it must originate from $a|x_{\frac{1}{2}}$. However,  by Lemma~\ref{lem: C_S_structure_maps}, the arrow of $\bI|L_\tau$ originating from
$a|x_{\frac{1}{2}}$ is weighted by $Z^{\ell+\theta}=Z^{\ell}$ (recall from Lemma~\ref{lem: braided_property} that $\theta=0$ for braided L-space satellite patterns). Since $\ell>1$, this arrow vanishes in
$\hat\bX(P,K,\lambda)/(Z)$.

Thus, the only remaining possibilities are the internal differentials of
$\hat\cS$ and the arrows of the form $\bI|L_\sigma$ coming from $\Phi^{\mu}$.  
We divide the argument into two cases depending on $\Xi$. Recall from Lemma~\ref{lem: braided_property} that  $\Xi \geq 0$ for braided L-space satellite patterns.

If $\Xi>0,$  then  by Lemma \ref{lem: C_S_structure_maps}, $g_3(P)>0$, and by \cite[Theorem 7]{HeddenWatson} the internal arrow of $\hat\cS$ from $\ys'_1$ to $\xs'_0$ is weighted by $W$. On the other hand,
following a similar argument as in the proof of Lemma \ref{lem: C_S_structure_maps},
by Equation~\eqref{eq: gradingxs} and the definition of $\hat\cC_{\frac{\ell}{2}-1}$,  for any basis element  $\ws \in \hat\cC_{\frac{\ell}{2}-1}$ we  have 
\[\gr_{\zs}(\ws)\ge \gr_{\zs}(\xs_0^{\ell/2-1}) = -2 R_{\frac{\ell}{2}-1} - \ell.\]
 Therefore by Equation \eqref{eq: grading_shift}
\[
\gr_{\zs}(L_\sigma(\ws)) \geq  -2 R_{\frac{\ell}{2}-1} + \ell -2 = \gr_{\zs}(\xs'_0) + 2\Xi.
\]
Since $\Xi>0,$ we conclude that  $\im(L_\sigma)$ does not contain $x'_{-\frac{1}{2}} \cdot b$ for any $b \in \hat \cR$.

Therefore the only incoming arrow to $ a| x'_{-\frac{1}{2}}$ is from $a| \hat{y}'_{-\frac{1}{2}}$.
  Since $\delta^1(a)=0$, the only outgoing arrows from $a| \hat{y}'_{-\frac{1}{2}}$ are the internal arrows of $\hat\cS$. Therefore we obtain
\[ \partial_{\hori} (a| \hat{y}'_{-\frac{1}{2}}) = a| x'_{-\frac{1}{2}}W \]
where $\d_{\hori}$ consists of the components of the differential of $\hat{\bX}(P,K,n)^{\hat{\cR}}$ which are weighted by powers of $W$.

If $\Xi=0,$ then by the computation of the gradings of $L_{\sigma}(\xs^{\ell/2-1}_0)$ in the proof of Lemma \ref{lem: C_S_structure_maps}, the arrow of $\bI | L_\sigma$ from $a| x_{-\frac{1}{2}}$ is weighted by $W.$ In fact, we will show 
\[\partial_{\hori} (a| x_{-\frac{1}{2}}) = a| x'_{-\frac{1}{2}}W.\]
Since $\delta^1(a)=0$,  aside from the arrow of $\bI | L_\sigma$,  the only other possible outgoing arrow from $a| x_{-\frac{1}{2}}$ is from $\bI | L_\tau$. 
By Lemma \ref{lem: C_S_structure_maps}, $L_\tau(x_{-\frac{1}{2}})$ has coefficient  $Z^\kappa$. Since $\kappa = \ell - \Xi - 1 = \ell - 1 > 0,$  such an arrow vanishes after setting $Z=0$, proving the claim.  

Therefore in each case, in $\hat\bX(P,K,\lambda)/(Z)$ there is no arrow weighted by $1$ pointing to  $ a| x'_{-\frac{1}{2}}$. (In the case when $\Xi=0,$ there is possibly an internal arrow of $\hat\cS$ pointing to  $ a| x'_{-\frac{1}{2}}$. But such an arrow is weighted by $W$.) This proves that  $[a| x'_{-\frac{1}{2}}] \neq 0$ in $H_*(\cCFK(P(K))/(Z))$. Moreover, in each case we identified an element in $\hat\bX(P,K,\lambda)/(Z)$ with $\partial_{\hori}$ equal to $a| x'_{-\frac{1}{2}}W$, proving that  $[a| x'_{-\frac{1}{2}}]W =0$ in $H_*(\cCFK(P(K))/(Z))$.
\end{proof}

We now introduce a technical notion, which we call the \emph{refined torsion order} of a knot like complex $\cC$. We introduce this to help deal with the following challenge: our bounds on the torsion order in the proofs of Proposition~\ref{prop: non_ending_general_case} and~\ref{prop: non_ending_braided_case} make use of the structure of the standard complex representative of $\cC$. However, our model $\hat{\bX}(P,K,n)^{\hat{\cR}}$ is not naturally presented as a standard complex, so there is some challenge in iterating the stronger form of our bounds. Therefore we introduce a refinement of the torsion order, denoted $\Ord'(K)$, which is basis independent and will help bridge this gap.

Let $\cC$ be a free complex over $\hat{\cR}$. If $x\in \cC$ is a cycle, write $[x]_{W=0}$ and $[x]_{Z=0}$ for the images of $x$ in $H_*(\cC/W)$ and $H_*(\cC/Z)$ respectively.

\begin{define} Suppose that $\cC$ is a knot like complex.  Define the \emph{refined torsion order} $\Ord'(\cC)\in \N\cup \{\infty\}$ as follows. Let $\Ord'(\cC)$ be the maximal value of $N$ such that there is a cycle $x\in \cC$ such $[x]_{Z=0}$ is non-zero and torsion, and $Z^{N-1}[x]_{W=0}\neq 0$ and $Z^{N}[x]_{W=0}=0$.
 \end{define}

Now consider the case that $\cC$ is a standard complex or local system complex. Let $x$ be one of the basis elements of $\cC$. We say that $x$ is a \emph{sink} if there is a $Z$-weighted arrow and a $W$-weighted arrow pointing to $x$.
 
 \begin{lem}\label{lem: refined_tor} Suppose that $\cC$ is a knot like complex over $\hat{\cR}$ which is decomposed as a sum of standard complexes and local system complexes. Suppose further that $\Ord'(\cC)$ is finite. Then either $\Ord(\cC)>\Ord'(\cC)$, or there is a sink $x$ in $\cC$ whose $Z$ arrow is weighted by $\Ord'(\cC)$.
 \end{lem}
 \begin{proof} Without loss of generality, we assume that $\cC$ is a direct sum of standard complexes and local system complexes.  Let $x$ be a cycle as in the definition of $\Ord'(\cC)$. Write $x=x_0+x_W+x_Z$ where $x_0$ has no $W$ or $Z$ powers, $x_W$ is a multiple of $W$ and $x_Z$ is a multiple of $Z$ (with respect to the decomposition of $\cC$ into standard complexes and local system complexes).

 We break the proof into two main cases:
 \begin{enumerate}
 \item $x_0=0$
 \item $x_0\neq 0$.
 \end{enumerate} 
 We first consider the case that $x_0=0$. In this case, $y=[Z^{-1}x_Z]_{W=0}$ is an element of homology with $Z^{N}\cdot [y]\neq 0\in H_*(\cC/W)$ and $Z^{N+1}\cdot [y]=0$. Therefore $\Ord(\cC)\ge N+1>\Ord'(\cC)$, as claimed.
 
  We now consider the case that $x_0\neq 0$. We observe firstly that $x_0$ must itself be a cycle. Furthermore, $x_0$ cannot contain any endpoints of the standard complex of $\cC$ since $[x]_{W=0}$ and $[x]_{Z=0}$ are torsion.  Therefore, $x_0$ must be a sum of sinks.
 
 We observe that $[x]_{W=0}=[x_0]_{W=0}+[x_Z]_{W=0}$. Write $z$ for this element in $H_*(\cC/W)$. By definition, $Z^N\cdot z=0$ but $Z^{N-1} \cdot z \neq 0$. Clearly, $Z^{N-1} \cdot [x_0]_{W=0}\neq 0$ or $Z^{N-1}\cdot[x_Z]_{W=0}\neq 0$. Furthermore, $Z^N\cdot[x_0]_{W=0}=0$ and $Z^N\cdot [x_Z]_{W=0}=0$. If $Z^{N-1}\cdot [x_0]_{W=0}\neq 0$, then there must be a $Z^N$ arrow pointed toward one of the summands of $x_0$, which implies the main claim. On the other hand, if $Z^{N-1}\cdot [x_Z]_{W=0}\neq 0$, then $\Ord(\cC)>N=\Ord'(\cC)$ by the same argument as when $x_0=0$.
 \end{proof}

We are ready to prove Proposition~\ref{prop: longest_arrow}.
\begin{proof}[Proof of Proposition~\ref{prop: longest_arrow}]
We recall that, for a general L-space satellite pattern $P$, if an arrow weighted by $Z^n$ in $\cCFK(K)$ is of one of the following types:
 \begin{itemize}
        \item  a non-ending arrow of type $(+,+)$; or
        \item  a non-ending arrow of type $(+,-)$; or
        \item    an ending arrow with $\eta>0$ and $\varepsilon(K)=-1$,
    \end{itemize} 
    then, corresponding to these three types, we constructed an incoming arrow to a certain generator $\alpha_0$
    in Case \eqref{it: non_ending_general_case_1} and Case \eqref{it: non_ending_general_case_3} of the proof of Proposition \ref{prop: non_ending_general_case}, and in Case \eqref{it: ending_general_epsilon<0_eta>0} of the proof of Proposition \ref{prop: ending_arrow_general_case}. Now suppose further that $P$ is braided. Then these constructions show that $\Ord(P(K,\lambda))\geq \ell n$ in all three cases. (See Case \eqref{it: non_ending_braided_case_1} and \eqref{it: non_ending_braided_case_3} of Proposition \ref{prop: non_ending_braided_case}, and Case \ref{it: ending_braided_case_3} of Proposition \ref{prop: ending_arrow_braided_case}.) If $\Ord(P(K,\lambda))> \ell n$, there is nothing to prove. Therefore we assume that $\Ord(P(K,\lambda))= \ell n$ henceforth.
    
In each case, the generator  $\alpha_0$  is of the form $a| x'_{-\frac{1}{2}}$, where $a=a_0$ in the first two cases and $a=a_1$ in the last. Moreover,  in each case, for the above choice of $a$, $[a| x'_{-\frac{1}{2}}]Z^{\ell n -1} \neq 0$  and $[a| x'_{-\frac{1}{2}}]Z^N = 0$ for some $N\geq \ell n$ in $H_*(\bX(P,K,\lambda)/(W))$. Since $\Ord(P(K,\lambda))= \ell n$, we may take $N=\ell n.$

Let $\cC=\cCFK(P(K,\lambda))^{\hat\cR}.$
It follows from  Lemma \ref{lem: alpha0_hori} and the previous discussion that $\Ord'(\cC)=\ell n$. Since $\Ord(\cC)=\ell n,$ by Lemma \ref{lem: refined_tor},  $\cC$ has a sink with an incoming arrow weighted by $Z^{\ell n}$. This arrow falls into one of the three cases listed above, thereby proving the claim.
\end{proof}
We are now ready to prove Theorem \ref{thm: iterate_braided}, which we recall below.
\IterateCable*
\begin{proof}
    By \cite[Theorem 1.1]{AlishahiEftekhary_Unknotting} we need only to show that $\Ord(P_k \circ \cdots \circ P_1(K)) \geq \ell_k \cdots \ell_1$. Since $K$ is not the unknot, we have $\Ord(K)>0.$
    
    If $\Ord(K)>1,$  iterating Theorem \ref{thm: tor_ord_braided} yields the desired bound.

    If $\Ord(K)=1,$  since $K \neq T_{2,3}$, by Lemma \ref{lem: Ord=1_cases},
    $\cCFK(K)^{\hat{\cR}}$ contains an arrow that is one of the three types listed in Proposition~\ref{prop: longest_arrow}. By  Proposition~\ref{prop: longest_arrow}, there are two possibilities: 
    \begin{itemize}[label=--]
        \item $\Ord(P_1(K)) > \ell_1 ,$ in which case iterating Theorem \ref{thm: tor_ord_braided} yields the desired bound; or
      \item $\Ord(P_1(K)) = \ell_1$ and  $\cCFK(P_1(K))^{\hat{\cR}}$ contains an arrow weighted by $Z^{\ell_1}$ that is of one of the three types listed in Proposition~\ref{prop: longest_arrow}. 
     \end{itemize}
     Proceeding inductively,
    if for some $s \in \{1,\dots, k-1\}$, we have $\Ord(P_{s} \circ \cdots \circ P_1(K)) > \ell_{s} \cdots \ell_1,$ then iterating Theorem \ref{thm: tor_ord_braided} yields the desired bound.  Otherwise, for every $s \in \{1,\dots, k-1\}$,  Proposition~\ref{prop: longest_arrow} implies that $\Ord(P_{s} \circ \cdots \circ P_1(K)) = \ell_{s} \cdots \ell_1$ and that
    $\cCFK(P_{s} \circ \cdots \circ P_1(K))^{\hat{\cR}}$ contains an arrow weighted by $Z^{\ell_{s} \cdots \ell_1}$ that is of one of the three types listed in Proposition~\ref{prop: longest_arrow}.  Hence, iterating Proposition~\ref{prop: longest_arrow}
    yields the desired bound. 
    \end{proof}

\bibliographystyle{custom}
\def\MR#1{}
\bibliography{biblio}	
\end{document}